\documentclass[a4paper,reqno,11pt]{amsart}
\usepackage[hmargin=3cm,vmargin=3cm]{geometry}
\usepackage{amsfonts}
\usepackage{amsmath}
\usepackage{amsthm}
\usepackage{amssymb}
\usepackage{mathtools}
\usepackage{stmaryrd}
\usepackage{mathrsfs}
\usepackage{bbding}
\usepackage{bbm}
\usepackage{upgreek}
\usepackage[utf8]{inputenc}
\usepackage{fancyhdr}
\usepackage{enumerate}
\usepackage{enumitem}
\usepackage[hang,flushmargin]{footmisc}

\DeclareFontFamily{U}{dutchcal}{\skewchar\font=45 }
\DeclareFontShape{U}{dutchcal}{m}{n}{<-> s*[1.0] dutchcal-r}{}
\DeclareFontShape{U}{dutchcal}{b}{n}{<-> s*[1.0] dutchcal-b}{}
\DeclareMathAlphabet{\mathlcal}{U}{dutchcal}{m}{n}
\SetMathAlphabet{\mathlcal}{bold}{U}{dutchcal}{b}{n}

\usepackage{graphicx}
\usepackage{float}
\usepackage[figureposition=bottom]{caption}
\usepackage[dvipsnames]{xcolor}
\usepackage{tikz}
\usetikzlibrary{calc,3d,angles,arrows.meta}

\usepackage[numbers]{natbib}

\usepackage{hyperref}
\hypersetup{
    hidelinks,
    hypertexnames=false,
}
\usepackage[capitalise]{cleveref}

\theoremstyle{plain}
\newtheorem{theorem}{Theorem}[section]
\newtheorem{lemma}[theorem]{Lemma}
\newtheorem{proposition}[theorem]{Proposition}
\newtheorem{proposition*}{Proposition}
\newtheorem{corollary}[theorem]{Corollary}

\theoremstyle{definition}

\newtheorem{example}{Example}[section]
\newtheorem*{example*}{Example}

\theoremstyle{remark}
\newtheorem{remark}{Remark}

\numberwithin{equation}{section}

\usepackage{accents}
\newcommand\thickbar[1]{\accentset{\rule{.4em}{.5pt}}{#1}}

\makeatletter
\@namedef{subjclassname@2020}{
  \textup{2020} Mathematics Subject Classification}
\makeatother

\DeclareMathOperator{\N}{\mathbb{N}}
\DeclareMathOperator{\Z}{\mathbb{Z}}

\DeclareMathOperator{\R}{\mathbb{R}}
\DeclareMathOperator{\Gr}{\Gamma}
\DeclareMathOperator{\tht}{\upvartheta}
\DeclareMathOperator{\Om}{\Upomega}
\DeclareMathOperator{\F}{\mathscr{F}}
\DeclareMathOperator{\p}{\mathbbm{P}}
\DeclareMathOperator{\E}{\mathbbm{E}}
\DeclareMathOperator{\g}{\mathfrak{g}}
\DeclareMathOperator{\valg}{\mathfrak{v}}
\DeclareMathOperator{\indexN}{\mathnormal{h}}

\DeclareMathOperator{\bull}{{\scriptscriptstyle\bullet}}
\DeclareMathOperator{\ab}{\operatorname{ab}}
\newcommand{\GHto}{\xrightarrow{\text{GH}}}
\DeclareMathOperator{\tor}{\operatorname{tor}}

\title[Shape of subadditive processes on groups]{Asymptotic shape of subadditive processes on groups of polynomial growth}
\author[C. F. Coletti \and
L. R. de Lima]
{Cristian F. Coletti \and
Lucas R. de Lima}

\address{Centro de Matem\'atica, Computa\c{c}\~ao e Cogni\c{c}\~ao, Universidade Federal do ABC\\
Av. dos Estados, 5001\\
Santo Andr\'e, S\~ao Paulo, 09210-580\\
Brazil.}
\email{cristian.coletti@ufabc.edu.br}

\address{Departamento de Matem\'atica, Instituto de Ci\^encias Matem\'aticas e de Computa\c{c}\~ao, Universidade de S\~ao Paulo\\
Av. Trabalhador S\~ao--Carlense, 400\\
S\~ao Carlos - SP, 13566-590\\
Brazil.}
\email{lrdelima@icmc.usp.br}

\keywords{Subadditive cocycles, shape theorem, first-passage percolation, groups of polynomial growth, asymptotic innerness, chemical distance, frog model}

\subjclass[2020]{Primary: 60F15, 60K35; Secondary: 37A30, 20F65}

\newcommand{\abfree}{\mathrm{ab,free}}

\begin{document}

\begin{abstract}
We prove an almost-sure shape theorem for stationary ergodic random growth processes and passage times on finitely generated groups of polynomial growth. Our conditions allow unbounded passage times under an upper-tail bound of moment type determined by the volume-growth degree, as well as zero-weight edges. We show that the rescaled set of sites reached by time $n$ converges in the asymptotic cone to a limit ball with deterministic pointed isometry class, and that rescaled passage times converge uniformly on word balls to a Carnot--Carath\'eodory distance. Our asymptotic innerness condition allows nearly minimizing chains whose largest word displacement is sublinear in the scale. For stationary ergodic first-passage percolation on the Cayley graph, an upper-tail bound and a linear lower bound at large distances suffice. Applications include random colorings and directed Richardson growth in a random environment. We also obtain a Carnot--Carath\'eodory scaling limit for chemical distance on the infinite cluster of supercritical Bernoulli bond percolation on these Cayley graphs, and extend the frog-model shape theorem to all groups of polynomial growth of degree at least three, including nonabelian groups.
\end{abstract}

\maketitle

\section{Introduction}

Consider a growth process started at a single vertex of an infinite graph. In first-passage percolation, the region reached by time $t$ consists of the vertices that can be reached along a path of total weight at most $t$. The same region has two interpretations: it records the spread of the process up to time $t$ and is the ball of radius $t$ in the random passage distance. A shape theorem describes the geometry of this region at large times. The study of limiting shapes in this setting has attracted growing interest in recent years \cite{auffinger2023,benjamini2015,bjoerklund2010,cantrell2017,coletti2021}.

We use \emph{cost} for the time or length associated with reaching one vertex from another. For the processes considered here, this cost is subadditive: reaching $w$ from $u$ costs at most the sum of the costs of reaching $v$ from $u$ and then $w$ from $v$. The second stage starts at $v$, so its cost is evaluated in the environment seen from that vertex. On $\Z^d$, translation moves $v$ to the origin, and stationarity preserves the law of the environment. Along successive multiples of a fixed lattice vector, the subadditive ergodic theorem then gives a directional growth rate. To obtain a shape theorem, one must pass from growth along these individual directions to an estimate valid throughout the growing region.

This use of spatial shifts makes vertex-transitive graphs a natural setting for the question. If a group $H$ of graph automorphisms acts transitively, every vertex can serve as the origin of a process whose law is invariant under $H$. Fixing a vertex $v$ identifies the vertex set with the coset space $H/H_v$, where $H_v$ is the stabilizer of $v$. On a Cayley graph there is a further structure: the group of translations acts freely as well as transitively, so each vertex determines a unique translation taking the identity to that vertex. By our convention, the displacement from $u$ to
$v$ is $vu^{-1}$. Passage costs can therefore be recorded by their displacements and by the environment at their starting points. The concatenation of the journeys then gives precisely the subadditive cocycle inequality.

Cayley graphs provide a natural framework for transport and propagation in statistically homogeneous random media. Their generators represent elementary moves, while random passage costs describe local delays, as in first-passage percolation \cite{benjamini2015,cantrell2017}, and bond percolation represents unavailable connections \cite{gorski2025chemical}. Related nilpotent structures arise in constrained motion: Murray and Sastry \cite{murray1991chained} develop nilpotent control models for steering cars and cars with trailers, where the order of elementary maneuvers affects the resulting configuration. Shape theorems identify the effective large-scale geometry of random passage costs, describing travel times and the region accessible within a given time.

We study finitely generated groups of polynomial growth. One reason concerns estimates whose starting point varies with the scale. Translating a large word ball by a fixed group element changes only a vanishing proportion of its vertices, a property known as the F{\o}lner property. The ergodic theorem for word balls identifies the limiting frequency of vertices where a translated estimate holds with its probability at the origin. Volume growth makes this information useful locally: for fixed $R,\delta>0$, every ball of radius $\delta n$ centered in $B_S(e,Rn)$ occupies a positive fraction of $B_S(e,(R+\delta)n)$, uniformly for large $n$. If the failure probability is smaller than this fraction, every starting vertex in $B_S(e,Rn)$ lies within $\delta n$ of one where the estimate holds. Our upper bound on passage costs controls the error in moving between them.

The other reason is geometric. By Gromov's theorem \cite{gromov1981}, a finitely generated group of polynomial growth is virtually nilpotent. Its rescaled word metric converges to a homogeneous distance on a graded nilpotent Lie group \cite{pansu1983,breuillard2014}. The word metric thus identifies the space in which the limiting shape lives, while the directional costs of the growth process determine a Carnot--Carath\'eodory distance on that space. We show that this distance approximates the rescaled passage costs uniformly over all ordered pairs of vertices in word balls of radius $O(n)$.

Our results extend the shape theorems of Benjamini and Tessera \cite{benjamini2015} for first-passage percolation and of Cantrell and Furman \cite{cantrell2017} for subadditive cocycles. On $\Z^d$, we also obtain a uniform two-point approximation for possibly asymmetric costs under a uniform tail-envelope hypothesis, complementing Bj\"orklund's radial theorem for symmetric random semimetrics \cite{bjoerklund2010}.

For first-passage percolation (FPP), Benjamini and Tessera \cite{benjamini2015} proved a shape theorem on virtually nilpotent Cayley graphs for i.i.d.\ edge weights with a finite exponential moment and a linear lower bound on the expected passage distance. Theorem~\ref{thm:ambient.fpp} generalizes their shape theorem to stationary ergodic edge environments. We allow dependence among the edge weights and impose two estimates directly on passage times: an upper-tail bound of moment type determined by the volume-growth degree, imposed above the word-distance scale, and an almost-sure linear lower bound outside a finite random ball. Both estimates follow under the hypotheses of Benjamini and Tessera. Auffinger and Gorski \cite{auffinger2023} studied the converse problem of realizing prescribed limit distances by stationary FPP. On virtually nilpotent groups, their realization theorem singles out distances invariant under the conjugation action.

Example~\ref{ex:honeycomb} describes the honeycomb lattice as a Cayley graph of $\Z^2\rtimes_{-I}\Z_2$, with every edge crossing the two cosets of $\Z^2$, and computes its unit-weight hexagonal limit.

For subadditive cocycles, Theorems~\ref{shape.thm} and~\ref{thm:shape.polynomial} relax both the deterministic bi-Lipschitz bounds and the innerness hypothesis of Cantrell and Furman \cite{cantrell2017}. From a probabilistic point of view, their bounds constrain even the cost of a single step. The upper bound excludes arbitrarily long local delays, and the lower bound excludes arbitrarily small local costs. Either phenomenon can occur in a process with a nondegenerate limiting shape. Our condition \eqref{all} permits unbounded costs by controlling $\p(c(x)\geq t)$ uniformly over $x$ once $t>\upbeta\|x\|_S$. Their innerness condition approximates costs using increments from a finite set at each prescribed accuracy. We replace this by \emph{asymptotic innerness}: for endpoints in a word ball of radius $O(n)$, a chain approximates the cost up to $o(n)$, and its largest word displacement is $o(n)$. The increments may therefore grow with $n$, while becoming negligible after rescaling. Lemma~\ref{lem:fixed.innerness} shows that the fixed-increment condition implies ours under \eqref{all}.

The passage to a subgroup of finite index also requires distinguishing a limit norm in fixed horizontal coordinates from the pointed isometry class of its metric space. Different ergodic components of the subgroup action may have different norms, related by the conjugation action of the ambient group. Proposition~\ref{prop:component.coordinates} gives an example satisfying the hypotheses of \cite[Theorem~A]{cantrell2017} in which the radial cost functions are not asymptotically equivalent in the same coordinates. Our finite-index formulation retains the component norms and their conjugation relation; the resulting pointed isometry class is deterministic.

A related motivation comes from chemical distance in supercritical Bernoulli bond percolation. Gorski and Procaccia \cite[Theorem~1]{gorski2025chemical} establish exponential upper-tail estimates for the chemical distance between connected vertices of transitive graphs of polynomial growth. In their discussion of scaling limits \cite[Section~3]{gorski2025chemical}, they explain that a shape theorem in this setting calls for weaker assumptions than those of Cantrell and Furman. Corollary~\ref{cor:chemical.shape} obtains this scaling limit on Cayley graphs of groups of polynomial growth. Corollary~\ref{cor:chemical.continuity} also identifies a locally Lipschitz dependence of the limiting horizontal norm on the retention parameter above criticality.

For a torsion-free nilpotent group, the lower hypothesis depends on its algebraic structure. In the nonabelian case we combine asymptotic innerness with a linear lower bound at sufficiently large word distances; its threshold may depend on the environment. In the abelian case a linear lower bound on the expected cost in terms of word length suffices, without a chain condition. For a general group of polynomial growth, one may verify the corresponding assumptions on each ergodic component of a nilpotent subgroup of finite index, or verify asymptotic innerness and radial lower growth directly on the ambient group. The latter formulation is useful when efficient chains leave the subgroup, as happens in first-passage percolation and the frog model.

On $\Z^d$, Boivin \cite{boivin1990} proved a shape theorem for stationary ergodic FPP. Bj\"orklund \cite[Theorem~1.2]{bjoerklund2010} showed that, under $L^{d,1}$ integrability of the generator costs, a stationary ergodic symmetric random semimetric admits a deterministic seminorm $L$ such that, almost surely,
\[
 d_\omega(0,x)=L(x)+o(\|x\|) \qquad (\|x\|\to\infty).
\]
Condition \eqref{all} implies $L^d$ integrability for every fixed endpoint on $\Z^d$. However, \eqref{all} requires one tail envelope for all endpoints above a linear word-distance threshold, whereas Bj\"orklund's hypothesis concerns the generator costs. This uniform information allows us to compare costs at moving basepoints. The order of the endpoints is retained throughout, including when the two directions have different passage costs.

The frog model illustrates why asymptotic innerness, directed estimates, and the ambient finite-index formulation matter. Although the individual walks are symmetric, the activation times $T_\omega(x,y)$ and $T_\omega(y,x)$ can differ in the same realization. A minimizing activation chain follows successive frogs from their birth vertices. Its total activation cost is exact, but its spatial increments need not belong to one fixed finite set. Gaussian displacement estimates show that the largest increment is $o(n)$ when the activation time is $O(n)$, which gives asymptotic innerness. Such chains may leave a nilpotent subgroup even when their endpoints belong to it. The ambient alternative in Theorem~\ref{thm:shape.polynomial} handles these chains without adding errors over all their increments. In \cref{sec:examples} we extend the frog-model shape theorem of Coletti and de Lima \cite[Theorem~1.2]{coletti2021} from finitely generated abelian groups to all finitely generated groups of polynomial growth of degree $D\geq3$, including groups with nonabelian nilpotent subgroups.

\subsection*{Basic description and main results}

Let $\Gr$ be a finitely generated group of polynomial growth, and let $\tht:\Gr\curvearrowright(\Om,\F,\p)$ be an ergodic measure-preserving action on a probability space. We consider a family of nonnegative random variables $\{c(x)\}_{x\in\Gr}$, normalized by $c(e)=0$, such that, $\p$-almost surely, for all $x,y\in\Gr$,
\begin{equation}\label{eq:subadditivity}
 c(xy)\leq c(y)+c(x)\circ\tht_y.
\end{equation}

Write $c(x,\omega)$ for $c(x)(\omega)$ and let $z\cdot\omega:=\tht_z (\omega)$. A function $c:\Gr\times\Om\to \R_{\geq 0}$ satisfying \eqref{eq:subadditivity} is called a \emph{subadditive cocycle}. A \emph{directed pseudometric} is a nonnegative function $d$ with $d(x,x)=0$ and $d(x,z)\leq d(x,y)+d(y,z)$; neither symmetry nor separation is required. A subadditive cocycle defines a random directed pseudometric $d_\omega$ by
\[
    d_{z\cdot\omega}(x,y):= \big(c(yx^{-1})\circ\tht_{x}\big)(z\cdot \omega),
\]
which is $\Gr$-right equivariant, \emph{i.e.}, for all $x,y,z \in \Gr$, and for every $ \omega \in \Om$,
\[
    d_\omega(x,y) = d_{z\cdot\omega}(xz^{-1},yz^{-1}).
\]
Conversely, a $\Gr$-right equivariant random directed pseudometric $d_\omega$ defines a subadditive cocycle by
\begin{equation*}
    c(x,\omega):= d_\omega(e, x).
\end{equation*}
Thus the correspondence is one-to-one.

We first assume that $\Gr$ is torsion-free nilpotent, and then treat general groups of polynomial growth. Fix a finite symmetric generating set $S\subseteq\Gr\setminus\{e\}$. We use generating sets without the identity throughout and write $\|\cdot\|_S$ and $d_S$ for the corresponding word length and word metric. For a scalar $a$, $|a|$ denotes its absolute value; for a set $A$, $|A|$ denotes its cardinality. Vector norms are written with double bars, and word lengths always carry the generating-set subscript.

The following conditions will be needed throughout the paper. We first state the assumptions for the case in which $\Gr$ is torsion-free nilpotent. Let $D\geq 1$ denote its polynomial volume-growth degree. We assume that there exist $\upbeta>0$ and a nonincreasing function $g:(0,\infty)\to[0,1]$ such that, for every $x\in\Gr$,
\begin{equation}\label{all}\tag{\scshape i}
 \p\bigl(c(x)\geq t\bigr)\leq g(t) \qquad \bigl(t>\upbeta\|x\|_S\bigr),
\end{equation}
where
\[
 \int_1^\infty t^{D-1}g(t)\,dt<\infty.
\]
The same upper-tail bound thus applies to every endpoint once the threshold exceeds $\upbeta\|x\|_S$. For a single survival function, the integral is the $D$-th moment tail condition. It implies $c(x)\in L^D$ for every fixed endpoint $x$, but for $D\geq2$ it need not imply $c(s)\in L^{D,1}$ even for a generator $s$. In particular, the integral condition holds whenever $g(t)=O(t^{-(D+\upkappa)})$ for some $\upkappa>0$. A finite-net argument uses this decay to control passage costs simultaneously over ordered pairs in growing word balls.

Our radial lower-growth condition requires a deterministic constant $b>0$ such that
\begin{equation}\label{small.prob}\tag{\scshape ii}
 \p\bigl(c(x,\omega)\geq b\|x\|_S
       \text{ for all sufficiently large }\|x\|_S\bigr)=1.
\end{equation}
The radius beyond which this bound holds may depend on $\omega$. Outside that radius it holds simultaneously for all endpoints. In particular, it confines chains of cost $O(n)$ to word balls of radius $O(n)$.

We also use an \emph{asymptotic innerness} condition at the starting environment. For almost every $\omega$ and every $R\in\N$, there is a nonnegative sequence $\upvarepsilon_R(n,\omega)\to0$ such that, for all sufficiently large $n$, every $x\in B_S(e,Rn)$ admits a finite chain $v_0=e,\ldots,v_k=x$ satisfying
\begin{equation}\label{innerness}\tag{\scshape iii}
 \begin{aligned}
 \max_{1\leq j\leq k}\|v_jv_{j-1}^{-1}\|_S
       &\leq\upvarepsilon_R(n,\omega)n,\\
 \sum_{j=1}^k c(v_jv_{j-1}^{-1},v_{j-1}\cdot\omega)
       &\leq c(x,\omega)+\upvarepsilon_R(n,\omega)n.
 \end{aligned}
\end{equation}
For $x=e$ we take the empty chain, with maximum displacement and cost zero. By subadditivity, the sum along any chain from $e$ to $x$ is at least $c(x,\omega)$. Thus \eqref{innerness} asks for an approximation whose error is negligible at scale $n$, using jumps that are also negligible at that scale. The largest jump may grow with $n$, and no bound is placed on the number of increments. The error sequence is common to every endpoint in the word ball. Loops can be erased without increasing either bound. Together, \eqref{small.prob} and \eqref{innerness} yield a continuous horizontal curve in the rescaled limit, even when the chain has many increments of small or zero cost. Both conditions also apply directly to an ambient group that is not nilpotent.

When $\Gr$ is abelian, we instead use the following lower bound in expectation: there exists $a>0$ such that
\begin{equation}\label{aml}\tag{\scshape ii\(^{\mathrm{ab}}\)}
 \E[c(x)]\geq a\|x\|_S \qquad(x\in\Gr).
\end{equation}
In the abelian case we assume only \eqref{all} and \eqref{aml}; neither \eqref{small.prob} nor \eqref{innerness} is imposed. In the nonabelian case we assume \eqref{all}, \eqref{small.prob}, and \eqref{innerness}. Lemma~\ref{lem:directional.lower} derives the horizontal lower bound needed to construct a nondegenerate limit from either lower hypothesis.

To describe that limit, let $\alpha:\Gr\to\mathfrak v_1=\Gr^{\ab}\otimes_{\Z}\R$ be the canonical horizontal projection, given by $\alpha(x)=(x[\Gr,\Gr])\otimes 1$, where $x[\Gr,\Gr]$ denotes the class of $x$ in the discrete abelianization. The image $\Lambda=\alpha(\Gr)$ is a full lattice in $\mathfrak v_1$, naturally isomorphic to $\Gr^{\ab}/\tor(\Gr^{\ab})$. Fix a Euclidean norm $\|\cdot\|$ on $\mathfrak v_1$.
Write $q_t(x)=t^{-1}\bull x$ for the rescaling map from $\Gr$ to its associated graded group $G_\infty$, constructed in \cref{sec:asymptotic.cone}, and denote the identity of $G_\infty$ by $\mathlcal e$. These are two notations for the same map: we use the bullet notation in the cone construction and coordinate calculations, and $q_t$ in the subsequent estimates and shape statements. Rescaling maps act on sets pointwise, so $q_t(A)=\{q_t(x):x\in A\}$. A subadditive, positively homogeneous function $\phi$ on $\mathfrak v_1$ that is positive away from the origin defines a right-invariant Carnot--Carath\'eodory distance $d_\phi$. The distance $d_\phi(g,h)$ is the infimum of the $\phi$-lengths of
horizontal curves from $g$ to $h$. We write $\Phi(g)=d_\phi(\mathlcal e,g)$. The function $\phi$ may be asymmetric, in which case the order of the endpoints matters.

We use $\GHto$ in the following pointed sense for directed distances. For a directed distance $d$, write $d^{\rm s}(x,y)=\max\{d(x,y),d(y,x)\}$ and identify points at $d^{\rm s}$-distance zero. Given pointed spaces $(X_n,d_n,p_n)$ and $(Y,d,p)$ with proper symmetrizations, we write $(X_n,d_n,p_n)\GHto(Y,d,p)$ if, for every $R>0$ and every $0<\eta<R$, all sufficiently large $n$ admit a map
\[
 f:\overline B_{d_n^{\rm s}}(p_n,R)\longrightarrow Y, \qquad f(p_n)=p,
\]
such that
\[
 \sup_{x,y\in\operatorname{dom}f} \bigl|d_n(x,y)-d(f(x),f(y))\bigr|\leq\eta,
 \qquad \sup_{z\in\overline B_{d^{\rm s}}(p,R-\eta)} \inf_{x\in\operatorname{dom}f}d^{\rm s}(z,f(x))\leq\eta.
\]
The symmetrized distances determine the closed comparison balls and measure how densely the image of $f$ covers the target ball. The distortion estimate compares the original distances on ordered pairs and therefore retains the cost in each direction. For symmetric distances this is the usual approximation formulation of pointed Gromov--Hausdorff convergence.

Write $A_\omega(t)=\{x\in\Gr:c(x,\omega)\leq t\}$ for the forward reached set. Hausdorff distances in the cone below are measured with the reference metric $d_\infty$.

\begin{theorem}[Limiting shape for torsion-free nilpotent groups]\label{shape.thm}
Let $\Gr$ be an infinite finitely generated torsion-free nilpotent group and $c$ a normalized nonnegative subadditive cocycle over an ergodic
p.m.p.\ $\Gr$-action. Assume \eqref{all}. If $\Gr$ is abelian, assume \eqref{aml}; otherwise, assume \eqref{small.prob} and \eqref{innerness}. Then there is a deterministic horizontal norm $\phi$, possibly asymmetric, characterized by
\[
 \phi(\alpha(x))=\lim_{n\to\infty}\frac{\E[c(x^n)]}{n}
 \qquad(x\in\Gr).
\]
Almost surely, for every $R>0$,
\begin{equation}\label{eq:thm.nilpotent.two.point}
 \sup_{u,v\in B_S(e,Rn)}
 \left|\frac{d_\omega(u,v)}{n}-d_\phi(q_n(u),q_n(v))\right|\longrightarrow0.
\end{equation}
On the same event,
\[
 d_H^{\,d_\infty}\!\left(q_n(A_\omega(n)),
       \{g:d_\phi(\mathlcal e,g)\leq1\}\right)\longrightarrow0,
\]
and, in the directed pointed sense defined above,
\begin{equation}\label{eq:GH:conv.THM}
 (\Gr,\tfrac{1}{n}d_\omega,e)\GHto(G_\infty,d_\phi,\mathlcal e) \qquad\p\text{-a.s.}
\end{equation}
The limit distance is homogeneous and bi-Lipschitz equivalent to $d_\infty$.
\end{theorem}

The limit space $G_\infty$ is also known as a Carnot group and $d_\infty$ coincides with the Carnot--Carath\'eodory metric obtained by the asymptotic cone of $\Gr$ as the limit of $\frac{1}{n}d_S$. More details about its construction and properties will be given in \cref{basic:def} along with the definitions of $\delta_t$ and $d_\phi$. The usage of the pointed Gromov--Hausdorff convergence arises naturally from its correspondence with geometric group theory.

Let now $(\Gr,\cdot)$ be a finitely generated group with polynomial growth rate. Gromov's theorem \cite{gromov1981} gives a nilpotent subgroup of
finite index. A torsion-free subgroup of finite index in that subgroup, followed by its normal core in $\Gr$, gives a normal torsion-free nilpotent subgroup $N\unlhd\Gr$ of finite index $\indexN=[\Gr:N]$; see also \cite[Appendix~B]{auffinger2023}.

Fix representatives $z_{(1)}=e,z_{(2)},\ldots,z_{(\indexN)}$ for the left cosets of $N$ and define
\[
 \uppi_N(x)=z_{(j)}^{-1}x\quad\text{when }x\in z_{(j)}N.
\]
This is a retraction of sets, not necessarily a homomorphism. It may collapse Cayley-graph edges; it is used to compare endpoints, not to define passage costs by contracting edges (Example~\ref{ex:honeycomb}). Choose a finite symmetric generating set $S_N\subseteq N\setminus\{e\}$ independently of $S$. The word metric of $N$ is bi-Lipschitz equivalent to the restriction of $d_S$. The graded group $G_\infty$ is constructed from $N$; the maps $q_t:N\to G_\infty$ are its nilpotent rescalings. Write $d_{\infty,N}$ for the word-cone metric associated with $S_N$. When applying the nilpotent lemmas to $N$, we use $d_{\infty,N}$ as the reference metric. It is bi-Lipschitz equivalent to the ambient word-cone metric $d_\infty$, so either reference metric gives the same notion of Hausdorff convergence.

Let $\mathfrak B_N$ be the finite partition into $N$-ergodic components, and write $\p_B=\p(\,\cdot\mid B)$ and $\E_B$ for the conditional probability and expectation. Existence of this partition follows from \cref{lm:erg.finite.index}.
The component forms of the radial lower bound and asymptotic innerness are, for every $B\in\mathfrak B_N$ and some $b_B>0$,
\begin{align}
 &\p_B\bigl(c(u)\geq b_B\|u\|_{S_N}\text{ for all sufficiently large }\|u\|_{S_N}\bigr)=1,
       \label{small.prob2}\tag{\scshape ii\(_N\)}\\
 &\text{the restriction to }N\text{ satisfies asymptotic innerness \eqref{innerness} under }\p_B.
       \label{innerness2}\tag{\scshape iii\(_N\)}
\end{align}
In \eqref{innerness2}, word distances and balls are taken with respect to $S_N$, the chains lie in $N$, and the error sequences may depend on $B$, $R$, and $\omega$. For each fixed word ball scale they must be common to all its endpoints. When $N$ is abelian, we instead require, for every $B\in\mathfrak B_N$,
\begin{equation}\label{aml2}\tag{\scshape ii\(^{\mathrm{ab}}_N\)}
 \E_B[c(u)]\geq a_B\|u\|_{S_N}\qquad(u\in N),\qquad a_B>0.
\end{equation}
The expectation in \eqref{aml2} is conditional on $B$; an unconditional mean bound does not suffice (Example~\ref{ex:components}). Alternatively, we may impose \eqref{small.prob} and \eqref{innerness} directly on $\Gr$. Ambient chains may leave $N$, and the proof compares only finitely many selected endpoints with their projections.

The component hypotheses concern the restriction $c|_{N\times\Om}$ on the original probability space. No quotient by torsion is required. In this finite-index setting, the horizontal space and projection are
\[
 \mathfrak v_1=N^{\ab}\otimes_{\Z}\R,\qquad
 \alpha=\alpha_N:N\to\mathfrak v_1,\qquad
 \alpha_N(u)=(u[N,N])\otimes1.
\]
Fix a Euclidean norm $\|\cdot\|$ on this horizontal space. For ambient points we write
\[
 Q_t(x)=q_t(\uppi_N(x))=t^{-1}\bull\uppi_N(x),\qquad x\in\Gr.
\]
Thus $q_t$ is defined on $N$, whereas $Q_t$ first projects to $N$; in particular, $Q_t|_N=q_t$.

We realize the ambient word-cone metric $d_\infty$ on this same graded group $G_\infty$. In these coordinates, the word-metric cone convergence recalled in \cref{thm:Pansu} gives, for every $R>0$,
\[
 \sup_{x,y\in B_S(e,Rn)}
 \left|\frac{d_S(x,y)}n-d_\infty(Q_n(x),Q_n(y))\right|
 \longrightarrow0.
\]
Its horizontal norm $\rho$ is characterized by
\[
 \rho(\alpha(u))=\lim_{n\to\infty}\frac{\|u^n\|_S}{n}
 \qquad(u\in N).
\]
Thus $d_\infty$ records the ambient generating set $S$, while $d_{\infty,N}$ records the independently chosen generating set $S_N$.

\begin{theorem}[Limiting shape for groups of polynomial growth]\label{thm:shape.polynomial}
Let $\Gr$ be an infinite finitely generated group of polynomial growth, and let $N\unlhd\Gr$ be torsion-free nilpotent of finite index. Let $c$ be a normalized nonnegative subadditive cocycle over an ergodic p.m.p.\ $\Gr$-action. Assume \eqref{all} and one of the following alternatives:
\begin{enumerate}[label=(\alph*)]
 \item on every $N$-ergodic component, \eqref{aml2} holds if $N$ is abelian, while \eqref{small.prob2} and \eqref{innerness2} hold if $N$ is nonabelian;
 \item the ambient cocycle on $\Gr$ satisfies \eqref{small.prob} and \eqref{innerness}.
\end{enumerate}
For each $B\in\mathfrak B_N$ there is a deterministic horizontal norm $\phi_B$, possibly asymmetric, characterized by
\[
 \phi_B(\alpha(u))=\lim_{n\to\infty}\frac{\E_B[c(u^n)]}{n}
 \qquad(u\in N).
\]
On one event of full $\p_B$-measure, for every $R>0$,
\begin{equation}\label{eq:thm.polynomial.two.point}
 \sup_{x,y\in B_S(e,Rn)}
 \left|\frac{d_\omega(x,y)}{n}
       -d_{\phi_B}(Q_n(x),Q_n(y))\right|\longrightarrow0,
\end{equation}
and
\begin{equation}\label{eq:thm.polynomial.reached}
 d_H^{\,d_\infty}\!\left(Q_n(A_\omega(n)),
       \{g:d_{\phi_B}(\mathlcal e,g)\leq1\}\right)\longrightarrow0.
\end{equation}
Consequently, in the directed pointed sense,
\begin{equation}\label{eq:GH:conv.THM2}
 (\Gr,\tfrac{1}{n}d_\omega,e)\GHto(G_\infty,d_{\phi_B},\mathlcal e) \qquad\p_B\text{-a.s.}
\end{equation}
The limit distances are homogeneous and bi-Lipschitz equivalent to $d_\infty$. The component limits are pointed-isometric; hence the almost sure limit has a deterministic pointed isometry class.
\end{theorem}

The component norms form a single orbit under the induced horizontal action of $\Gr/N$, as expressed in Equation \eqref{eq:component.norms}. With respect to a common invariant Euclidean inner product, their unit balls are orthogonal images of one another.

For stationary ergodic first-passage percolation on the ambient Cayley graph $\mathcal C(\Gr,S)$, asymptotic innerness follows from minimizing nearest-neighbor paths, so the assumptions simplify further. Even when $u,v\in N$, a minimizing path between them in $\mathcal C(\Gr,S)$ may leave $N$, so the restricted passage time is not an intrinsic FPP metric on $N$. Nevertheless, we show in \cref{thm:ambient.fpp} that the upper-tail bound \eqref{all} and the radial lower bound \eqref{small.prob} on $\Gr$ alone suffice: no innerness condition on $N$ is needed, and on each $N$-ergodic component $B$ the rescaled reached set $Q_n(\{x\in\Gr:T_\omega(e,x)\leq n\})$ converges in Hausdorff distance to the closed unit ball of $d_{\phi_B}$, while $(\Gr,\tfrac{1}{n}T_\omega,e)\GHto(G_\infty,d_{\phi_B},\mathlcal e)$ in the usual pointed Gromov--Hausdorff sense.

\subsection*{Organization of the paper}

In \cref{basic:def} we describe the asymptotic cone, illustrate the finite-index coordinates on the honeycomb lattice, construct the horizontal norm from the directional limits of the cocycle, and introduce first-passage percolation on Cayley graphs. In \cref{sec:limiting.shape} we use polygonal approximations of horizontal curves to prove the shape theorem for torsion-free nilpotent groups. We then pass to groups of polynomial growth by considering a nilpotent subgroup of finite index. The section concludes with \cref{sec:additional.FPP}, where we prove a shape theorem for first-passage times computed in the ambient Cayley graph and illustrate endpoint comparison on the honeycomb lattice. In \cref{sec:examples} we treat the coloring and directed Richardson models, establish a chemical-distance scaling limit and parameter continuity for supercritical Bernoulli bond percolation, and prove the frog-model shape theorem on groups of polynomial growth of degree at least three.

\section{Preliminaries} \label{basic:def}

We recall the word metric and asymptotic cone, give examples illustrating the finite-index and abelianization conventions, and construct the horizontal norm from directional costs. We also introduce first-passage percolation. Background on the group geometry can be found in \cite{breuillard2014,decornulier2011,decornulier2016,raghunathan1972}.

\subsection{Cayley graphs and volume growth}

Let $(\Gr,\cdot)$ be a group generated by a finite set $S$. The associated Cayley graph $\mathcal{C}(\Gr, S)$ represents elements of $\Gr$ as vertices, with edges connecting $x$ and $y$ if and only if $y = sx$ for some $s \in S$. Formally, the right-invariant Cayley graph $\mathcal{C}(\Gr,S)=(V,E)$ is defined by
\[
    V= \Gr \quad \text{and} \quad E=\big\{\{x,sx\}:x \in \Gr, s \in S\big\}.
\]

Denote by $u\sim v$ the relation $\{u,v\} \in E$. Let $\mathscr{P}(x,y)$ be the set of self-avoiding paths from $x$ to $y$, where each $\gamma \in \mathscr{P}(x,y)$ follows $\gamma=(x_0,\dots, x_m)$ with $m \in \N_0$, $x_i\sim x_{i+1}$, $x_0=x$, $x_m=y$, and $x_i \neq x_j$ for all $i \neq j$. We write $\mathtt{e} \in \gamma$ for $\mathtt{e} = \{x_{i}, x_{i+1}\} \in E$, and $\ell(\gamma)=m$ represents the length of the path. When $m=0$, the path consists of the single vertex $x=y$ and has length and passage time zero.

The word length on $\Gr$ with respect to $S$ is defined as follows: For any $x \in \Gr$, the length of the shortest word (or self-avoiding path) in $S$ that represents $x$ is its word length, denoted by $\|x\|_S = \inf_{\gamma\in\mathscr{P}(e, x)}\ell(\gamma)$. The word metric $d_S$ on $\Gr$ is given by $d_S(x, y) = \|yx^{-1}\|_S$.

For any directed pseudometric $d_\diamondsuit$ on a nonempty set $\mathbb X$, write $B_\diamondsuit(x,r)=\{y\in\mathbb X:d_\diamondsuit(x,y)<r\}$ for its open forward ball. A symmetric directed pseudometric is called a pseudometric; a pseudometric separating distinct points is a metric. The symmetrization $d_\diamondsuit^{\rm s}$ is a pseudometric and becomes a metric after identifying points at mutual distance zero. We use $t\vee t'=\max\{t,t'\}$, $t\wedge t'=\min\{t,t'\}$, $\N=\{1,2,\ldots\}$, and $\N_0=\N\cup\{0\}$.

A finitely generated group $\Gr$ has polynomial growth with respect to $S$ when $|B_S(e, r)| \in O\big(r^{D'}\big)$ for a $D'\in \N_0$ as $r\uparrow+\infty$. The growth is associated with the Cayley graph $\mathcal{C}(\Gr,S)$. The polynomial growth rate of $\Gr$ is a constant $D \in \N_0$ such that there exists $\mathtt{k}\in (1,+\infty)$ for all $r>1$ satisfying \[\mathtt{k}^{-1}r^D \leq |B_S(e,r)| \leq \mathtt{k}r^D.\] Thus $D = \min \big\{D'\in\N_0\colon |B_S(e,r)| \in O\big(r^{D'}\big)\big\}$. Moreover, one can verify that the polynomial growth rate of $\mathcal{C}(\Gr,S)$ does not depend on the choice of $S$.

Recall the definition of the commutator element $[x,y] = xyx^{-1}y^{-1}$ and the subgroup $[U,V] := \big\langle [u,v] \colon u \in U, v\in V \big\rangle$ for any $U,V \subseteq \Gr$. Set $\Gr_0:=\Gr$ and let $\Gr_n := [\Gr,\Gr_{n-1}]$ for all $n \in \N$. Thus $\{\Gr_n\}_{n\in\N}$ forms a lower central series with $\Gr_n \unlhd \Gr_{n-1}$ for all $n\in\N$. The group $\Gr$ is called \textit{nilpotent} when there is an $n \in \N$ such that $\Gr_n = \{e\}$, i.e., when the sequence stabilizes in the trivial group. More specifically, $\Gr$ is nilpotent of class $n$ when $n$ is the minimal value such that $\Gr_n$ is the trivial group. A group is abelian if and only if it is nilpotent of class $1$. The abelianization of a group $\Gr$ is given by $\Gr^{\ab}\cong \Gr/[\Gr,\Gr]$.

The group $\Gr$ is called \textit{virtually nilpotent} when there exists a normal subgroup $N \unlhd \Gr$ that is nilpotent with finite index ${\indexN}=[\Gr\colon N] < +\infty$. A noteworthy result obtained by Gromov \cite{gromov1981} is that a finitely generated group has polynomial growth exactly when it is virtually nilpotent. Therefore, the growth established by word metrics is strongly related to algebraic properties of the group.

For a finitely generated nilpotent group, its finite-order elements form a finite characteristic subgroup, called its torsion subgroup. This is not true for arbitrary groups. A group is torsion-free when its identity is its only finite-order element. Although $N$ is torsion-free, its abelianization may have torsion. The horizontal lattice is $N^{\abfree}=N^{\ab}/\tor(N^{\ab})$, identified with the image of $N$ in $N^{\ab}\otimes_{\Z}\R$.

Let $[U]_\varepsilon$ be the $\varepsilon$-neighborhood of $U \subseteq \mathbb{X}$ in a metric space $(\mathbb{X}, d_\diamondsuit)$, \textit{i.e.}, the set $[U]_\varepsilon = \bigcup_{u \in U} B_\diamondsuit(u,\varepsilon)$. The Hausdorff distance $d_H$ detects the largest variations between sets with respect to the given metric \[d_H(U,V) := \inf\{\varepsilon>0 \colon U \subseteq [V]_\varepsilon \text{ and } V \subseteq [U]_\varepsilon\}.\] For the word metric, the pointed convergence defined in the introduction is the usual pointed Gromov--Hausdorff convergence. We recall Pansu's theorem.

\begin{theorem}[Pansu \cite{pansu1983}] \label{thm:Pansu}
    Let $\Gr$ be a virtually nilpotent group generated by a symmetric and finite $S \subseteq \Gr$. Then
    \[
        \left( \Gr, \frac{1}{n}d_S, e \right) \GHto (G_\infty, d_\infty, \mathlcal{e}),
    \]
    where $G_\infty$ is a simply connected real graded Lie group (Carnot group). The metric $d_\infty$ is a right-invariant sub-Finsler (Carnot--Carath\'eodory) metric which is homogeneous with respect to a family of homotheties $\{\delta_t\}_{t>0}$, \textit{i.e.}, $d_\infty\big(\delta_t(\mathlcal{g}),\delta_t(\mathlcal{h})\big) = t~d_\infty(\mathlcal{g},\mathlcal{h})$ for all $t>0$ and $\mathlcal{g},\mathlcal{h} \in G_\infty$. 
\end{theorem}

Note that \cref{shape.thm,thm:shape.polynomial} are generalizations of the theorem above. Therefore, the shape theorems under investigation can be interpreted as the convergence of random metric spaces in large-scale geometry. The next subsection is dedicated to the construction of the asymptotic cone $G_\infty$ and related results.

\subsection{Rescaled distance and asymptotic cone} \label{sec:asymptotic.cone}

Consider for now $\Gr$  as a nilpotent and torsion-free group, unless stated otherwise. In this subsection, we use $\Gr$ instead of $N$ to simplify notation, but we will subsequently extend the results to the more general case.

Let $G$ denote the real Mal'cev completion of $\Gr$, the simply connected nilpotent real Lie group in which $\Gr$ embeds as a cocompact discrete subgroup. It is unique up to isomorphism and has the same nilpotency class as $\Gr$. Its Lie algebra is $(\g,[\cdot,\cdot]_1)$, and
$\log:G\to\g$ denotes the Lie logarithm map.

Define $\g^1 := \g$  and $\g^{i+1} := [\g,\g^i]_1$. It follows from the nilpotency of $\Gr$ that there exists $l \in \N$ such that $\Gr_l = \{e\}$. Thus $\g^{l+1} = (0)$. Since $[\g^{i},\g^{j}]_1 \subseteq \g^{i+j}$ and, in particular, $[\g^{i+1},\g^{j}]_1,[\g^{i},\g^{j+1}]_1 \subseteq \g^{i+j+1}$, the Lie bracket on $\g$ determines a bilinear map
\[(\g^{i}/\g^{i+1}) \otimes (\g^{j}/\g^{j+1}) \longrightarrow \g^{i+j}/\g^{i+j+1}\]
which in turn defines a Lie bracket $[\cdot,\cdot]_\infty$ on

\[\g_\infty := \bigoplus\limits_{i=1}^l \valg_i~~~\mbox{with}~~ \valg_i := \g^i/\g^{i+1}.\]

Consider the decomposition $\g = V_1 \oplus \cdots \oplus V_l$ given by $\g^i:= V_i \oplus \cdots \oplus V_l$. Thus, $(\g_\infty, [\cdot,\cdot]_\infty)$ is a graded Lie algebra. Let us define a family of linear maps $\delta_t : \g_\infty \to \g_\infty$ given by
\[\delta_t(v_1+v_2+ \dots + v_l) = tv_1+t^2v_2+ \dots+ t^lv_l\]
for each $t >0$ and $v_i \in \mathfrak{v}_i$ with $i \in \{1,\dots,l\}$. It follows from the definition of $\delta_t$ that $\delta_t([u,v]_\infty) = [\delta_t(u),\delta_t(v)]_\infty$ and $\delta_{tt'} = \delta_t \circ \delta_{t'}$ for all $u,v\in \g_\infty$ and $t,t'>0$. Hence, $\{\delta_t\}_{t>0}$ defines a family of automorphisms in the graded Lie group $G_\infty := \exp_\infty\left[ \g_\infty \right]$. Here we write $\exp_\infty:\g_\infty \to G_\infty$ and $\exp:\g \to G$ to differentiate the distinct exponential maps of $\g_\infty$ and $\g$. Similarly, $\log_\infty$ and $\log$ stand for their corresponding Lie logarithm maps.

Using this splitting, set $L:\g\to\g_\infty$ to be the linear identification whose restriction to $V_i$ is the quotient map onto $\valg_i$. Consider now $\sigma_t$ to be the linear automorphism on $\g$ so that $\sigma_t(v_i)=t^iv_i$ for each $v_i\in V_i$ and $i \in \{1, \dots, l\}$. Define the Lie brackets $[\cdot,\cdot]_t$ on $\g$ by
\[[v,w]_t = \sigma_{1/t}\big([\sigma_t(v),\sigma_t(w)]_1\big),~\mbox{ for all }t>0,\]
thus $(\g, [\cdot,\cdot]_t)$ is isomorphic to $(\g, [\cdot,\cdot]_1)$ via $\sigma_t$. Furthermore,
\[[L(v),L(w)]_\infty = \lim\limits_{t\uparrow+\infty} L([v,w]_t)\]
since, given $v \in V_i$ and $w \in V_j$, one has that the main term belongs to $V_{i+j}$, the other terms of superior order belong to $V_{i+j+1} \oplus \cdots \oplus V_l$ and it makes them insignificant in the rescaled limit (see \cite{breuillard2014,pansu1983} for a detailed discussion). For $x\in\Gr$ and $t>0$, define
\[
 q_t(x)=\frac1t\bull x
 :=(\exp_\infty\circ L\circ\sigma_{1/t}\circ\log)(x)
 =\delta_{1/t}(q_1(x)).
\]
Here $q_t$ maps the discrete group into $G_\infty$, while $\delta_s:G_\infty\to G_\infty$ dilates points already in the cone. In particular, $q_{t_n}(x_n)=t_n^{-1}\bull x_n$ for every sequence of positive scales $t_n$.
The convergence established by \cref{thm:Pansu} determines the metric $d_\infty$ such that
\[\left(\Gamma,\frac{1}{n}d_S,e\right) \GHto \left(G_\infty,d_\infty,\mathlcal{e}\right).\]
The statements $t_n^{-1}\bull x_n\to\mathlcal g$ and $q_{t_n}(x_n)\to\mathlcal g$ therefore express the same convergence in $(G_\infty,d_\infty)$. The corresponding metric statement shows that, given sequences $\{x_n\}_{n\in\N}$, $\{x_n'\}_{n\in\N}$ in $\Gamma$, and $t_n\uparrow+\infty$ as $n \uparrow +\infty$ with $\lim_{n\uparrow +\infty} \frac{1}{t_n}\bull x_n = \mathlcal{g}$ and $\lim_{n\uparrow +\infty} \frac{1}{t_n}\bull x_n' = \mathlcal{g}'$,
\begin{equation*}
    d_\infty (\mathlcal{g},\mathlcal{g}') = \lim_{n\uparrow +\infty} \frac{1}{t_n}d_S(x_n,x_n').
\end{equation*}

The abelianized Lie algebras are defined by $\g_\infty^{\ab} := \g_\infty/[\g_\infty,\g_\infty]_\infty \cong \mathfrak{v}_1$ and $\g^{\ab}:=\g/[\g,\g]_1$. In particular, $\g^{\ab}\cong\g_\infty^{\ab}$. An \textit{admissible} or horizontal curve in $G_\infty$ is an absolutely
continuous curve $\upgamma:[t_0,t_1]\to G_\infty$ whose right-trivialized velocity $v_\upgamma(t)$ belongs to $\mathfrak v_1$ almost everywhere.
Every $d_\infty$-Lipschitz curve is horizontal almost everywhere. Let $\phi:\g_\infty^{\ab}\to[0,+\infty)$ be a positive homogeneous subadditive function, possibly asymmetric. Its length functional is
\[
 \ell_\phi(\upgamma)=\int_{t_0}^{t_1}\phi(v_\upgamma(t))\,dt.
\]
The associated right-invariant distance is
\begin{equation}\label{eq:def.d.phi}
 d_\phi(\mathlcal g,\mathlcal g')=
 \inf\{\ell_\phi(\upgamma):\upgamma\text{ is horizontal from }
             \mathlcal g\text{ to }\mathlcal g'\text{ in }G_\infty\}.
\end{equation}
It is homogeneous under the Carnot dilations and is symmetric when
$\phi$ is symmetric.

In fact, the construction of $d_\phi$ can be employed to define $d_\infty$. The bi-Lipschitz property is a consequence of the results in \cref{sec:norms.and.mean}. One can also verify that the metric $d_\infty$ is right-invariant and homogeneous with respect to $\delta_t$. The chosen splitting identifies $\g^{\ab}=\g/\g^2$ with $V_1$. Define the corresponding projections
\[\pi:\g\to V_1 \quad\text{and}\quad \pi_\infty:\g_\infty \to \mathfrak{v}_1\cong \g_\infty^{\ab}\]
so that, if $v=\sum_{i=1}^l v_i \in \g_\infty$ with $v_i\in\mathfrak{v}_i$, then $\pi_\infty(v)=v_1$ and $\pi = L^{-1}\circ \pi_\infty\circ L$. These projections recover the horizontal map as $\alpha(x)=\pi_\infty L\log x=L\pi\log x$ for $x\in\Gr$. The next lemma compiles several well-known results that will be employed throughout the text. We state the results and their proofs can be found in \cite{cantrell2017,decornulier2016,pansu1983}.

\begin{lemma} \label{lm:conv.seq.Gr}

    Consider $\Gr$ a finitely generated torsion-free nilpotent group, then all of the following hold true:
    \begin{itemize}
    \item[(i)] Let $\mathlcal{g} \in G_\infty$. Then there exists a sequence $\{x_n\}_{n \in \N} \subseteq \Gr$ such that \[\lim_{n\uparrow + \infty}\frac{1}{n}\bull x_n= \mathlcal{g}.\]

    \item[(ii)]
    Let $\{x_n\}_{n \in \N}, \{y_n\}_{n \in \N} \subseteq \Gr$, $\mathlcal{g}, \mathlcal{h} \in G_\infty$, and  $t_n \uparrow +\infty$ as $n \uparrow +\infty$ be such that $\lim_{n\uparrow + \infty}\frac{1}{t_n}\bull x_n = \mathlcal{g}$ and $\lim_{n\uparrow + \infty}\frac{1}{t_n}\bull y_n = \mathlcal{h}$. Then
    \[\lim_{n\uparrow + \infty}\frac{1}{t_n}\bull (x_n y_n) = \mathlcal{gh}.\]

    \item[(iii)]
    Let $x \in \Gr$, then
    \begin{align*}
        \lim_{n\uparrow + \infty}\frac{1}{n}\bull x^n &= \left(\exp_\infty \circ L \circ \pi \circ \log\right)(x)\\
        &= \left(\exp_\infty \circ ~\pi_\infty \circ L\circ \log\right)(x).
    \end{align*}
    \end{itemize}
\end{lemma}
\begin{remark}
    The conditions imposed on $\Gr$ might appear somewhat restrictive. However, we will subsequently regain many properties by making necessary adjustments for virtually nilpotent $\Gr$ through $N \unlhd \Gr$ (see \cref{sec:virt.nilpotent.proofs}).
\end{remark}



The discrete Heisenberg group illustrates these constructions in nilpotency class two.

\begin{example}[The discrete Heisenberg group] \label{ex:Heisenberg.group}
    Let $R$ be a commutative ring with identity. Identify $H_3(R):=\{(\mathtt{x},\mathtt{y},\mathtt{z}):\mathtt{x},\mathtt{y},\mathtt{z}\in R\}$ with the upper triangular matrices
    \[
        (\mathtt{x}, \mathtt{y}, \mathtt{z}) := 
        \begin{pmatrix}
        \mathtt{1} & \mathtt{x} & \mathtt{z}\\
        \mathtt{0} & \mathtt{1} & \mathtt{y} \\
        \mathtt{0} & \mathtt{0} & \mathtt{1}
        \end{pmatrix}.
    \]
    
     \begin{figure}[htb!]
    \centering
    \begin{tikzpicture}[scale=0.7]
        \draw[blue] (-1.2,-2,-1) -- (-2.8,-2,-1);
        \draw[blue] (1.2,-2,1) -- (2.8,-2,1);
        \filldraw[draw=black!0,fill=black!30,opacity=0.1]
                (3,-2,2) -- (-3,-2,2) -- (-3,-2,-2) -- (3,-2,-2) -- cycle;
                
        \filldraw[draw=black!0,fill=black!30,opacity=0.15]
                (3,-1,2) -- (-3,-1,2) -- (-3,-1,-2) -- (3,-1,-2) -- cycle;
        
        \draw[blue] (-3,0,0) --(3,0,0);
        \draw[blue] (-1,0,-1) --(1,0,-1);
        \draw[blue] (-1,0,1) --(1,0,1);
        \draw[red] (0,0,-2) --(0,0,2);
        
        \draw[red] (-1,-1.5,-1.5) -- (-1,1.5,1.5);
    \draw[red] (-2,-3,-1.5) -- (-2,3,1.5);
        \draw[red] (1,1.5,-1.5) -- (1,-1.5,1.5);
        \draw[blue] (-0.5,1,-1) -- (2.5,1,-1);
        \draw[red] (0,1,-0.2) -- (0,1,-1.8);
        \draw[red] (2,3,-1.5) -- (2,-3,1.5);
        \draw[blue] (1.2,2,-1) -- (2.8,2,-1);
        \draw[blue] (-1.2,2,1) -- (-2.8,2,1);
        \draw[blue] (1.8,-1,1) -- (0.2,-1,1);
        \draw[blue] (-1.6,1,1) -- (-0.4,1,1);
        \draw[blue] (-1.6,-1,-1) -- (-0.4,-1,-1);

        \node at (-2,0,0) [circle, fill=black,scale=0.4] {};
        \node at (2,0,0) [circle, fill=black,scale=0.4] {};
        \node at (-1,0,0) [circle, fill=black,scale=0.4] {};
        \node at (1,0,0) [circle, fill=black,scale=0.4] {};
        \node at (0,0,0) [circle, fill=black,scale=0.4] {};
        \node at (0,0,1) [circle, fill=black,scale=0.4] {};
        \node at (0,0,-1) [circle, fill=black,scale=0.4] {};
        \filldraw[draw=black!0,fill=black!30,opacity=0.35]
                (3,0,2) -- (-3,0,2) -- (-3,0,-2) -- (3,0,-2) -- cycle;
        
        \node at (0,1,-1) [circle, fill=black,scale=0.4] {};
        \node at (2,1,-1) [circle, fill=black,scale=0.4] {};
        
        \node at (-1,-1,-1) [circle, fill=black,scale=0.4] {};
        \node at (-1,1,1) [circle, fill=black,scale=0.4] {};
        \node at (-2,-2,-1) [circle, fill=black,scale=0.4] {};
        \node at (-2,2,1) [circle, fill=black,scale=0.4] {};
        \node at (1,1,-1) [circle, fill=black,scale=0.4] {};
        \node at (1,-1,1) [circle, fill=black,scale=0.4] {};
        \node at (2,2,-1) [circle, fill=black,scale=0.4] {};
        \node at (2,-2,1) [circle, fill=black,scale=0.4] {};
        
        \filldraw[draw=black!0,fill=black!30,opacity=0.15]
                (3,1,2) -- (-3,1,2) -- (-3,1,-2) -- (3,1,-2) -- cycle;
                
        \filldraw[draw=black!0,fill=black!30,opacity=0.1]
                (3,2,2) -- (-3,2,2) -- (-3,2,-2) -- (3,2,-2) -- cycle;
    \end{tikzpicture}
    \caption{A section of the Heisenberg discrete Cayley graph $\mathcal{C}(H_3(\mathbb{Z}),S)$ embedded in $\R^3$.}
    \end{figure}

    The Heisenberg group on $R$ is $H_3(R)$ with the matrix multiplication. In particular, $H_3(\mathbb{Z})$ is known as discrete Heisenberg group. Let $\Gr = H_3(\mathbb{Z})$, $\mathtt{X} = (1,0,0)$, $\mathtt{Y} = (0,1,0)$, $\mathtt{Z} = (0,0,1)$, and $S =\{\mathtt{X}^{\pm 1},\mathtt{Y}^{\pm 1}\}$. Observe that
    \begin{align*}
        (\mathtt{x}, \mathtt{y}, \mathtt{z})\cdot(\mathtt{x}', \mathtt{y}', \mathtt{z}') &= (\mathtt{x} + \mathtt{x}', ~\mathtt{y} + \mathtt{y}', ~\mathtt{z} + \mathtt{z}' + \mathtt{x}\mathtt{y}'),\\
        (\mathtt{x}, \mathtt{y}, \mathtt{z})^{-1} &= (-\mathtt{x}, -\mathtt{y}, \mathtt{xy}-\mathtt{z}), \text{ and}\\
        \big[(\mathtt{x}, \mathtt{y}, \mathtt{z}),~(\mathtt{x}', \mathtt{y}', \mathtt{z}')\big] &= (0,0,\mathtt{xy}'-\mathtt{x}'\mathtt{y}).
    \end{align*}

    Therefore, for all $m,n\in \mathbb{Z}$,
    \begin{equation} \label{eq:Heisenberg.op}
        \mathtt{X}^m = (m,0,0), \quad \mathtt{Y}^n=(0,n,0),~~\text{and} \quad [\mathtt{X}^m,\mathtt{Y}^n]=\mathtt{Z}^{m\cdot n} = (0,0,m\cdot n).
    \end{equation}

    One can easily see that $S$ is a finite generating set of $\Gr$. Furthermore, $\Gr_1 = [\Gr,\Gr] = \langle\mathtt{Z}\rangle$ and $\Gr_2=[\Gr,\Gr_1] = \{e\}$. Hence, $\Gr$ is nilpotent of class $2$ and $S \subseteq \Gr \setminus [\Gr,\Gr]$. Consider $\|\cdot\|_S$ the word length of $\mathcal{C}(\Gr,S)$. It follows from \eqref{eq:Heisenberg.op} that
    \[\|\mathtt{Z}^{m}\|_S \in O(\sqrt{m}) \quad \text{as }m \uparrow +\infty.\]
    It highlights how the rescaled word length $\frac{1}{n}\|x^n\|_S$ vanishes as $n \uparrow +\infty$ when $x \in [\Gr,\Gr]$. 
    
    Due to the properties above, one can write $(\mathtt{x},\mathtt{y})=(\mathtt{x},\mathtt{y},\mathtt{z})[\Gr,\Gr]$. Note that $S^{\ab} = \left\{(\pm 1,0), (0,\pm 1)\right\}$ is a finite generating set of the abelianized group $\Gr^{\ab}= \Gr/[\Gr,\Gr]$ which yields an isomorphism of $\mathcal{C}(\Gr^{\ab},S^{\ab})$ and the square $\mathbb{Z}^2$ lattice.

    By construction of the asymptotic cone, the Mal'cev completion $G$ of $\Gr$ is the continuous Heisenberg group $H_3(\R)$ with its associated Lie algebra $\mathfrak{h}= \g$, in this case, $\g\simeq \g_\infty$ and $G \simeq G_\infty$. The Heisenberg algebra $\mathfrak{h}$ is given by $\mathfrak{h} = \operatorname{span}_{\R}\{e_{12},e_{13}, e_{23}\}$ with $\big\{e_{ij}:i,j \in \{1,2,3\}\big\}$ the canonical basis of $M_{3\times3}(\R)$.

    Since for all $\mathtt{A},\mathtt{B} \in \mathfrak{h}$ one has $[\mathtt{A},\mathtt{B}]_\infty = \mathtt{A}\mathtt{B}-\mathtt{B}\mathtt{A}\in\operatorname{span}_{\R}\{e_{13}\}$ by matrix multiplication, it then follows that $\mathfrak{h} = \mathfrak{v}_1 \oplus \mathfrak{v}_2$ with $\mathfrak{v}_1 \simeq \operatorname{span}_{\R}\{e_{12},e_{23}\}$ and $\mathfrak{h}^{\ab}\simeq \g^{\ab}_\infty \simeq\mathfrak{v}_1$. 
    
    Let $\mathtt{A} = \mathtt{u}\cdot e_{12} + \mathtt{v}\cdot e_{23} + \mathtt{w}\cdot e_{13}$, then $\exp_\infty (\mathtt{A})= \left(\mathtt{u},\mathtt{v}, \mathtt{w}+ \frac{1}{2}\mathtt{uv}\right)$. Since $(\mathtt{x},\mathtt{y},\mathtt{z})^{n} = \left(n\mathtt{x},n\mathtt{y},n\mathtt{z} + \frac{n(n-1)}{2}\mathtt{xy}\right)$ one can verify by the procedure defined in this section that $\frac{1}{n} \bull (\mathtt{x},\mathtt{y},\mathtt{z})^{n} = \left(\mathtt{x},\mathtt{y},\frac{1}{n}\mathtt{z}-\frac{1}{2n}\mathtt{xy} + \frac{1}{2}\mathtt{xy}\right) \in G_\infty$. It implies that, for all $\mathtt{x},\mathtt{y},\mathtt{z} \in \mathbb{Z}$,
    \[
        \lim_{n \uparrow +\infty} \frac{1}{n} \bull (\mathtt{x},\mathtt{y},\mathtt{z})^{n} = \left(\mathtt{x},\mathtt{y},\frac{1}{2}\mathtt{xy}\right) = \exp_\infty \left(\pi_\infty\Big( L\log(\mathtt{x},\mathtt{y},\mathtt{z}) \Big)\right).
    \]
\end{example}

From this point onward we use $q_t(x)=t^{-1}\bull x$ for the rescaling just illustrated in Example~\ref{ex:Heisenberg.group}; the map is unchanged. On bounded word balls these maps have the usual uniform metric and multiplication convergence. In particular, if $\|u_t\|_S,\|v_t\|_S=O(t)$, then
\begin{equation}\label{eq:cone.metric}
 d_\infty(q_t(u_t),q_t(v_t))-t^{-1}\|v_tu_t^{-1}\|_S\longrightarrow0
\end{equation}
uniformly on each such ball, and $q_t(v_tu_t^{-1})$ is uniformly close to $q_t(v_t)q_t(u_t)^{-1}$. Also,
\begin{equation}\label{eq:powers.geometry}
 q_n(x^n)\longrightarrow\exp_\infty\alpha(x).
\end{equation}
These facts follow from the graded Baker--Campbell--Hausdorff construction and the word-metric cone theorem. They use the horizontal image lattice
$\Lambda$, not torsion-freeness of $\Gr^{\ab}$; see \cite{cantrell2017,auffinger2023}.

\subsection{The finite-index subgroup and the horizontal directions}
\label{sec:examples.virt.nil}

For an infinite finitely generated group $\Gr$ of polynomial growth, the horizontal space in Theorem~\ref{thm:shape.polynomial} is constructed from a torsion-free nilpotent subgroup $N$ of finite index. The following example shows why the distinction between $\Gr$ and $N$ matters even when $N$ is abelian.

\begin{example}[Generalized dihedral groups]\label{ex:dihedral}
Let $(N,+)=(\Z^d,+)$, with $d\geq1$ and $(\Z_2,+)$ with $\Z_2=\{0,1\}$. The generalized dihedral group is 
\[
 \Gr=N\rtimes_{-I}\Z_2, \quad \text{with} \quad
\qquad
(v,h)(w,k)=\bigl(v+(-1)^h w,\ h+k\bmod2\bigr).
\]
The nontrivial element $r=(0,1)$ of $\Z_2$ acts by inversion. The subgroup $N$ has index two, so $\Gr$ has polynomial growth of degree $d$. Conjugation by $r$ sends each $u\in N$ to $-u$. Consequently, in $\Gr^{\ab}$ the image of $u$ agrees with that of $-u$ and therefore has order at most two. Since $r$ also has order two, $\Gr^{\ab}$ is finite. Thus
\[
 \Gr^{\ab}\otimes_{\Z}\R=0,
 \qquad
 N^{\ab}\otimes_{\Z}\R=\R^d.
\]
The second space gives the horizontal directions of the asymptotic cone. In particular, using the abelianization of the ambient group would lose every direction in which the limiting shape is measured.

For $d\geq3$, these groups give examples to which the frog-model result in \cref{ex:frog} applies. The extension also covers polynomial-growth groups with nonabelian nilpotent subgroups.
\end{example}


\begin{example}[The honeycomb graph and its two cosets]\label{ex:honeycomb} In Example~\ref{ex:dihedral}, take $\Gr=\Z^2\rtimes_{-I}\Z_2$. Write $t_z=(z,0)$ and $r=(0,1)$, so $rt_zr=t_{-z}$, and choose
\[
 N=\{t_z:z\in\Z^2\},\qquad
 S=\{r,rt_{e_1},rt_{e_2}\}
   =\{(0,1),(-e_1,1),(-e_2,1)\}.
\]
Each generator is an involution, and $r(rt_{e_i})=t_{e_i}$. Thus $S$ is symmetric and generates $\Gr$. With $A_z=t_z$ and $B_z=rt_z=(-z,1)$, the edges $\{x,sx\}$ are exactly
\begin{equation}\label{eq:honeycomb.edges}
 A_z\sim B_z,\qquad A_z\sim B_{z+e_1},\qquad
 A_z\sim B_{z+e_2}.
\end{equation}
This is the honeycomb lattice. The subgroup $N\cong\Z^2$ is normal of index two, and the growth degree is $D=2$.

The two vertex classes are $N=\{A_z\}$ and $rN=\{B_z\}$. Every edge joins these classes. Since $S\cap N=\varnothing$, the subgraphs induced by $N$ and by $rN$ have no edges. In particular, every nonconstant path between vertices of $N$ leaves $N$. Figure~\ref{fig:honeycomb}(a)--(b) displays the planar graph and the two cosets as coordinate lattices in parallel planes, with only cross-layer graph edges.
\begin{figure}[tbp]
\centering
\begin{tikzpicture}[
  font=\small,text=black,
  >={Stealth[length=1.6mm]},
  hcA/.style={circle,fill=blue!65!black,inner sep=1.25pt},
  hcB/.style={circle,draw=orange!80!black,fill=white,
              line width=.65pt,inner sep=1.15pt},
  hcEdge/.style={draw=black!28,line width=.5pt},
  hcMatch/.style={draw=orange!65!black,line width=.7pt},
  hcChosen/.style={draw=blue!65!black,line width=1.0pt},
  hcLabel/.style={fill=white,inner sep=1pt,font=\scriptsize}
]
  \pgfmathsetmacro{\hcUy}{sqrt(3)/2}
  \node at (2,3.82) {(a) Planar honeycomb};
  \node at (7.3,3.82) {(b) Two coset layers};
  \node at (12.25,3.82) {(c) Unit-weight limit};

  \begin{scope}[shift={(2,1.75)},scale=.55]
    \clip (-3.25,-2.65) rectangle (3.25,2.65);
    \fill[blue!7]
      (1.5,\hcUy)--(0,{2*\hcUy})--(-1.5,\hcUy)
      --(-1.5,{-\hcUy})--(0,{-2*\hcUy})--(1.5,{-\hcUy})--cycle;
    \foreach \m in {-3,...,3} {
      \foreach \n in {-3,...,3} {
        \pgfmathsetmacro{\ax}{1.5*(\m+\n)}
        \pgfmathsetmacro{\ay}{\hcUy*(\m-\n)}
        \draw[orange!70!black,line width=.7pt]
          (\ax,\ay)--({\ax-1},\ay);
        \draw[black!50,line width=.55pt]
          (\ax,\ay)--({\ax+.5},{\ay+\hcUy})
          (\ax,\ay)--({\ax+.5},{\ay-\hcUy});
      }
    }
    \draw[blue!65!black,dashed,line width=.9pt]
      (1.5,\hcUy)--(0,{2*\hcUy})--(-1.5,\hcUy)
      --(-1.5,{-\hcUy})--(0,{-2*\hcUy})--(1.5,{-\hcUy})--cycle;
    \foreach \m in {-3,...,3} {
      \foreach \n in {-3,...,3} {
        \pgfmathsetmacro{\ax}{1.5*(\m+\n)}
        \pgfmathsetmacro{\ay}{\hcUy*(\m-\n)}
        \node[hcA] at (\ax,\ay) {};
        \node[hcB] at ({\ax-1},\ay) {};
      }
    }
    \node[hcLabel,above right,xshift=2pt,yshift=2pt] at (0,0) {$A_0$};
    \node[hcLabel,below left,xshift=-2pt,yshift=-2pt] at (-1,0) {$B_0$};
  \end{scope}
  \node[font=\footnotesize] at (2,-.2)
    {\textcolor{blue!65!black}{$N$: filled}\quad
     \textcolor{orange!80!black}{$rN$: open}};

  \begin{scope}[shift={(6.4,.65)},
    x={(.85cm,0cm)},y={(.45cm,.35cm)},z={(0cm,1.2cm)}]
    \def\hcHeight{1.25}
    \fill[blue!5] (-.32,-.32,0)--(2.32,-.32,0)
      --(2.32,2.32,0)--(-.32,2.32,0)--cycle;
    \fill[orange!7] (-.32,-.32,\hcHeight)--(2.32,-.32,\hcHeight)
      --(2.32,2.32,\hcHeight)--(-.32,2.32,\hcHeight)--cycle;
    \foreach \m in {0,1,2} {
      \foreach \n in {0,1,2} {
        \draw[hcMatch] (\m,\n,0)--(\m,\n,\hcHeight);
        \ifnum\m<2
          \draw[hcEdge] (\m,\n,0)--({\m+1},\n,\hcHeight);
        \fi
        \ifnum\n<2
          \draw[hcEdge] (\m,\n,0)--(\m,{\n+1},\hcHeight);
        \fi
      }
    }
    \draw[hcMatch,line width=1.15pt]
      (1,1,0)--(1,1,\hcHeight);
    \draw[hcChosen] (1,1,0)--(2,1,\hcHeight);
    \draw[hcChosen] (1,1,0)--(1,2,\hcHeight);
    \foreach \m in {0,1,2} {
      \foreach \n in {0,1,2} {
        \node[hcA] at (\m,\n,0) {};
        \node[hcB] at (\m,\n,\hcHeight) {};
      }
    }
    \node[anchor=east,font=\footnotesize,text=orange!80!black]
      at (-.55,.8,\hcHeight) {$rN$};
    \node[anchor=east,font=\footnotesize,text=blue!65!black]
      at (-.55,.8,0) {$N$};
    \node[hcLabel,anchor=east,xshift=-2pt]
      at (1,1,\hcHeight) {$B_z$};
    \node[hcLabel,anchor=west,xshift=2pt]
      at (2,1,\hcHeight) {$B_{z+e_1}$};
    \node[hcLabel,anchor=south,yshift=2pt]
      at (1,2,\hcHeight) {$B_{z+e_2}$};
    \node[hcLabel,anchor=west,xshift=3pt]
      at (1,1,0) {$A_z$};
  \end{scope}
  \node[font=\footnotesize] at (7.3,-.2)
    {edges only between layers};

  \begin{scope}[shift={(12.25,1.75)},scale=1.45]
    \draw[->,black!35] (-1.05,0)--(1.05,0);
    \draw[->,black!35] (0,-1.12)--(0,1.12);
    \filldraw[fill=blue!10,draw=blue!65!black,line width=.85pt]
      (.75,{\hcUy/2})--(0,\hcUy)--(-.75,{\hcUy/2})
      --(-.75,{-\hcUy/2})--(0,{-\hcUy})--(.75,{-\hcUy/2})--cycle;
    \node[circle,fill=black,inner sep=.7pt] at (0,0) {};
    \node[right,font=\scriptsize] at (.75,{\hcUy/2}) {$u/2$};
    \node[right,font=\scriptsize] at (.75,{-\hcUy/2}) {$v/2$};
  \end{scope}
  \node[font=\footnotesize] at (12.25,-.2)
    {deterministic unit edge times};
\end{tikzpicture}
\caption{The honeycomb graph, its two cosets, and its unit-weight limit.
Filled vertices lie in $N$ and open vertices in $rN$.
Panels (a) and (c) use the planar embedding
$A_{(m,n)}=mu+nv$, $B_{(m,n)}=mu+nv-(1,0)$, where
$u=(3/2,\sqrt3/2)$ and $v=(3/2,-\sqrt3/2)$.
Panel (b) is a schematic three-dimensional coordinate drawing with
$A_{(m,n)}=(m,n,0)$ and $B_{(m,n)}=(m,n,H)$, $H>0$.
The shaded planes contain no graph edges: every edge joins the two layers. The three neighbours of one $A_z$ are highlighted. The orange edges $A_zB_z$ join vertices identified by $\uppi_N$; their costs remain part of the ambient passage metric. For deterministic unit edge times, the limiting ball in (c) is $K_{\mathrm{hc}}=\operatorname{conv}\{\pm u/2,\pm v/2,\pm(u-v)/2\}$. The dashed outline in (a) is $2K_{\mathrm{hc}}$, the convex hull of the six $A$-vertices at word distance $2$ from $A_0$. The limit hexagon is rotated by $30^\circ$ relative to the small hexagonal faces of the graph.}
\label{fig:honeycomb}
\end{figure}

For the transversal $\{e,r\}$, the comparison maps are \begin{equation}\label{eq:honeycomb.projection} \uppi_N(z,\epsilon)=t_{(-1)^\epsilon z},\qquad \uppi_N(A_z)=\uppi_N(B_z)=t_z,\qquad Q_t(A_z)=Q_t(B_z)=z/t. \end{equation} Thus $\uppi_N$ is a two-to-one retraction with $d_S(x,\uppi_N(x))\leq1$, but not a homomorphism: for $z\ne0$, $\uppi_N(t_zr)=t_{-z}\ne t_z=\uppi_N(t_z)\uppi_N(r)$. It collapses each matching edge $A_zB_z$ and sends the other two edge classes to the lattice steps $\{z,z+e_i\}$. The two vertex types therefore have the same rescaled coordinates. These are geometric comparisons of endpoints; passage costs remain those of paths in the ambient graph.

The word metric, equivalently the passage metric for unit edge times, has an explicit limit norm. Every path between two $A$-vertices has even length. Pairing successive edges gives increments in
\[
 \{0,\pm e_1,\pm e_2,\pm(e_1-e_2)\},
\]
and each nonzero increment can be realized by two honeycomb edges. Their word length at $(m,n)$ is $\max\{|m|,|n|,|m+n|\}$: each increment changes these three quantities by at most one; the bound is attained by axial steps for coordinates of the same sign, or by $\pm(e_1-e_2)$ followed by axial steps otherwise. Since $2\max\{|m|,|n|,|m+n|\}=|m|+|n|+|m+n|$, this gives
\begin{equation}\label{eq:honeycomb.exact.distance}
 d_S(A_0,A_{(m,n)})=|m|+|n|+|m+n|.
\end{equation}
Since $d_S(A_z,B_z)=1$, changing either vertex type changes the distance by at most one per endpoint. Thus the ambient word-cone metric in cell coordinates is $d_\infty(z,w)=\rho_{\mathrm{hc}}(w-z)$, where
\begin{align}
 \rho_{\mathrm{hc}}(x,y)&=|x|+|y|+|x+y|,\label{eq:honeycomb.norm}\\
 \{\rho_{\mathrm{hc}}\leq1\}
 &=\operatorname{conv}\left\{
       \pm e_1/2,\ \pm e_2/2,\ \pm(e_1-e_2)/2\right\}.
       \label{eq:honeycomb.ball}
\end{align}
For the usual planar embedding, put $u=(3/2,\sqrt3/2)$ and $v=(3/2,-\sqrt3/2)$, and set
\[
 A_{(m,n)}=mu+nv,\qquad B_{(m,n)}=mu+nv-(1,0).
\]
All edges then have Euclidean length one, and the limiting ball is the regular hexagon with vertices $\pm u/2,\pm v/2,\pm(u-v)/2$, shown in Figure~\ref{fig:honeycomb}(c).
\end{example}

The finite-index extension need not split. For example, the Klein bottle group $\Gr=\langle a,b\mid aba^{-1}=b^{-1}\rangle$ is torsion-free and contains $N=\langle a^2,b\rangle\cong\Z^2$ as a normal subgroup of index two. A splitting would produce an element of order two in $\Gr$. The comparison map $Q_n$ requires only a choice of coset representatives and applies in this case as well.

A separate issue arises even when the ambient group itself is torsion-free and nilpotent: its discrete abelianization may still contain torsion.

\begin{example}[Torsion in the abelianization]\label{ex:torsion.abelianization}
For $m\geq2$, equip $\Z^3$ with
\[
 (x,y,z)(x',y',z')=(x+x',y+y',z+z'+mxy').
\]
This is a torsion-free nilpotent group $N_m$. If a positive power of an element is the identity, its first two coordinates vanish, and then its third coordinate vanishes. With $a=(1,0,0)$, $b=(0,1,0)$ and $z_0=(0,0,1)$, the commutator subgroup is $\langle z_0^m\rangle$; hence
\[
 (N_m)^{\ab}\cong\Z^2\oplus\Z/m\Z,
 \qquad (N_m)^{\abfree}\cong\Z^2.
\]
The horizontal lattice is therefore the free abelianization, even though the original group has no torsion.
\end{example}

\subsection{Establishing a Candidate for the Limiting Shape}\label{sec:norms.and.mean}

From this point until \cref{sec:main.proofs}, let us once again regard $\Gr$ as a torsion-free nilpotent group. Our objective is to characterize the limiting shape using a norm that defines a distance in $G_\infty$. We first establish uniform upper control, then construct the horizontal
norm and prove the directional ergodic limit.

The map $\alpha:\Gr\to\mathfrak v_1$ has kernel $\{x:x^m\in[\Gr,\Gr]\text{ for some }m\geq1\}$. Consequently we use the image lattice $\Lambda=\alpha(\Gr)$; torsion-freeness of $\Gr^{\ab}$ is not assumed.

\begin{lemma}[Uniform upper control]\label{lem:upper}
Under \eqref{all}, for every $R,\eta>0$, almost surely for all sufficiently large integers $n$ and all $u,v\in B_S(e,Rn)$,
\begin{equation}\label{eq:upper}
 d_\omega(u,v)\leq\upbeta d_S(u,v)+\eta n.
\end{equation}
Consequently, if four endpoints lie in a fixed word ball of radius $O(n)$,
\begin{equation}\label{eq:endpoint}
 |d_\omega(u,v)-d_\omega(u',v')|  \leq\upbeta\{d_S(u,u')+d_S(v,v')\}+2\eta n
\end{equation}
eventually. Both assertions hold on one event for all rational $R,\eta$.
\end{lemma}
\begin{proof}
Put $N_k=2^k$ and fix $\delta>0$ with $\upbeta\delta\leq\eta/16$. Choose, independently of the environment, a maximal $\delta N_k$-separated net $\mathcal N_k\subseteq B_S(e,2RN_k)$. The balls $B_S(w,\delta N_k/2)$ centered at its points are disjoint and lie in $B_S(e,(2R+\delta)N_k)$. Polynomial volume growth gives $|\mathcal N_k|\leq C_{R,\delta}$ for all large $k$. For each $u\in B_S(e,2RN_k)$, choose a deterministic point $\pi_k(u)\in\mathcal N_k$ with $d_S(u,\pi_k(u))\leq\delta N_k$.

Monotonicity of $g$ gives, for every fixed $a>0$ and all sufficiently large $k$,
\[
 (aN_k)^Dg(aN_k)
 \leq\frac{D}{1-2^{-D}}
       \int_{aN_{k-1}}^{aN_k}t^{D-1}g(t)\,dt.
\]
Thus $\sum_kN_k^Dg(aN_k)<\infty$ under \eqref{all}.

Consider the event that both directed costs between each $u\in B_S(e,2RN_k)$ and $\pi_k(u)$ are at most $(\upbeta\delta+\eta/4)N_k$, and that
\[
 d_\omega(w,w')\leq\upbeta d_S(w,w')+\eta N_k/4
 \qquad(w,w'\in\mathcal N_k).
\]
Every threshold is strictly above the corresponding $\upbeta$-word-distance threshold. By stationarity, \eqref{all}, monotonicity, and the union bound, the probability of failure is at most
\[
 \bigl(2|B_S(e,2RN_k)|+|\mathcal N_k|^2\bigr)g(\eta N_k/4)
 \leq C'_{R,\delta}N_k^Dg(\eta N_k/4),
\]
which is summable in $k$. By Borel--Cantelli the event holds eventually. The triangle inequality through $\pi_k(u),\pi_k(v)$, together with
$d_S(\pi_k(u),\pi_k(v))\leq d_S(u,v)+2\delta N_k$, then gives
\[
 d_\omega(u,v)
 \leq\upbeta d_S(u,v)+(4\upbeta\delta+3\eta/4)N_k
 \leq\upbeta d_S(u,v)+\eta N_k.
\]
If $N_k\leq n<N_{k+1}$, then $B_S(e,Rn)\subseteq B_S(e,2RN_k)$ and $\eta N_k\leq\eta n$, proving \eqref{eq:upper} at every sufficiently large integer scale. Apply the directed triangle inequalities in both directions through the alternative endpoints to obtain \eqref{eq:endpoint}. A countable intersection over positive rational $R,\eta$ gives the common event. This proof uses only polynomial volume growth and stationarity, so it also applies to the ambient virtually nilpotent group; no independence is assumed.
\end{proof}

\begin{lemma}[Fixed increments and asymptotic innerness]\label{lem:fixed.innerness}
Assume \eqref{all}. Suppose that, almost surely, for every $\upvarepsilon\in(0,1)$ there is a finite set $F(\upvarepsilon,\omega)$ such that every $x\in\Gr$ admits a chain from $e$ to $x$, with increments in $F(\upvarepsilon,\omega)$, of total cost at most $(1+\upvarepsilon)c(x,\omega)$. Then \eqref{innerness} holds.
\end{lemma}
\begin{proof}
Fix $\omega$ on the common full-measure event and $R\in\N$. Lemma~\ref{lem:upper} gives $c(x,\omega)\leq C_Rn$ for $x\in B_S(e,Rn)$ and all large $n$. For $m\geq2$, put $L_m=\max\{\|z\|_S:z\in F(1/m,\omega)\}$. Choose increasing integers $N_m\to\infty$ such that $N_m\geq m\max_{2\leq j\leq m}L_j$. For large $n$, set $m(n)=\max\{m:N_m\leq n\}$. Then $m(n)\to\infty$, and the corresponding chains have maximum word jump at most $n/m(n)$ and total cost at most
\[
 (1+1/m(n))c(x,\omega)\leq c(x,\omega)+C_Rn/m(n).
\]
Taking $\upvarepsilon_R(n,\omega)=\max\{1,C_R\}/m(n)$ proves \eqref{innerness}. The argument uses only polynomial volume growth and applies to restrictions on the finite-index components as well.
\end{proof}

\begin{lemma}[Directional lower growth]\label{lem:directional.lower}
Assume \eqref{all}. If \eqref{small.prob} holds, or if $\Gr$ is abelian and \eqref{aml} holds, then there is a constant $a_0>0$ such that
\begin{equation}\label{eq:directional.lower}
 \lim_{n\to\infty}\frac{\E[c(x^n)]}{n}
       \geq a_0\|\alpha(x)\|\qquad(x\in\Gr).
\end{equation}
When $\Gr$ is abelian, \eqref{aml} and \eqref{eq:directional.lower} are equivalent up to the value of the constant.
\end{lemma}
\begin{proof}
Under \eqref{all}, $F(x)=\E[c(x)]$ is finite and subadditive, so Fekete's lemma gives
\[
 \ell(x):=\lim_{n\to\infty}\frac{F(x^n)}n
       =\inf_{n\geq1}\frac{F(x^n)}n\leq F(x).
\]
Put $L=\max_{s\in S}\|\alpha(s)\|>0$. Since $\alpha$ is a homomorphism,
\[
 n\|\alpha(x)\|=\|\alpha(x^n)\|\leq L\|x^n\|_S.
\]
If \eqref{small.prob} holds and $\alpha(x)\ne0$, then $\|x^n\|_S\to\infty$, and Fatou's lemma gives
\[
 \ell(x)\geq\E\!\left[\liminf_{n\to\infty}\frac{c(x^n)}n\right]
       \geq \frac{b}{L}\|\alpha(x)\|.
\]
For $\alpha(x)=0$ the desired bound follows from nonnegativity. If instead $\Gr$ is abelian and \eqref{aml} holds, applying it to $x^n$ gives
$F(x^n)/n\geq(a/L)\|\alpha(x)\|$, with the same conclusion.

Conversely, when $\Gr$ is abelian, $\alpha$ identifies $\Gr$ with the lattice $\Lambda$, and word length is comparable to the Euclidean norm on that lattice. Thus \eqref{eq:directional.lower} and $F(x)\geq\ell(x)$ imply $F(x)\geq a\|x\|_S$ for some $a>0$.
\end{proof}

\begin{lemma}
\label{lem:stable}
Let $F(x)=\thickbar c(x):=\E c(x)$. Under \eqref{all}, $F$ is finite and subadditive. Define
\[
 f(v)=\inf\{F(x):\alpha(x)=v\},\quad v\in\Lambda.
\]
There is a continuous positively homogeneous subadditive function
$\phi:\mathfrak{v}_1\to[0,\infty)$ such that
\begin{equation}\label{eq:stable}
 \phi(\alpha(x))=\lim_{n\to\infty}\frac{F(x^n)}n =\inf_{n\geq1}\frac{F(x^n)}n =\lim_{n\to\infty}\frac{f(n\alpha(x))}n.
\end{equation}
Under either lower hypothesis of Lemma~\ref{lem:directional.lower}, there are constants $a_0,b_1>0$ such that
\[
 a_0\|v\|\leq\phi(v)\leq b_1\|v\|\qquad(v\in\mathfrak v_1).
\]
If $d_\omega$ is symmetric almost surely, then $\phi(-v)=\phi(v)$ for every $v\in\mathfrak v_1$. In this case $\phi$ is a seminorm, and under either lower hypothesis it is a norm.
\end{lemma}
\begin{proof}
Integrating \eqref{all} gives $F(s)<\infty$ for every $s\in S$. By subadditivity and stationarity, $F$ is subadditive and $F(x)\leq b_0\|x\|_S$, where $b_0=\max_{s\in S}F(s)$. Applying this bound to $xy^{-1}$ and $yx^{-1}$ gives
\[
 |F(x)-F(y)|\leq b_0\|xy^{-1}\|_S.
\]
By Fekete's lemma,
\[
 \ell(x):=\lim_{n\to\infty}\frac{F(x^n)}n
 =\inf_{n\geq1}\frac{F(x^n)}n.
\]
If $\alpha(x)=\alpha(y)$, then \eqref{eq:powers.geometry} gives $q_n(x^n),q_n(y^n)\to\exp_\infty(\alpha(x))$. Hence \eqref{eq:cone.metric} yields $\|x^ny^{-n}\|_S=o(n)$, and the Lipschitz bound for $F$ gives $\ell(x)=\ell(y)$. In particular, if $\alpha(y)=k\alpha(x)$ for some $k\geq1$, then $y$ and $x^k$ have the same horizontal image, so $\ell(y)=\ell(x^k)=k\ell(x)$.

For $v,w\in\Lambda$, choose representatives $x,y\in\Gr$ with
$\alpha(x)=v$ and $\alpha(y)=w$. Since $\alpha(xy)=v+w$,
\[
 f(v+w)\leq F(xy)\leq F(x)+F(y).
\]
Taking the infimum over both representatives proves that $f$ is subadditive. The quotient word length on $\Lambda$ is bounded above by a constant multiple of the Euclidean norm, so $f(v)\leq b_1\|v\|$ for some $b_1<\infty$. Applying subadditivity in both directions gives $|f(v)-f(w)|\leq b_1\|v-w\|$. Its homogenization
\[
 \psi(v):=\lim_{n\to\infty}\frac{f(nv)}n
 =\inf_{n\geq1}\frac{f(nv)}n
\]
is therefore Lipschitz, subadditive, and positively integer-homogeneous on $\Lambda$. Extend $\psi$ by positive rational homogeneity to the rational span of $\Lambda$ and then by continuity to $\mathfrak v_1$. Denote this extension by $\phi$.

It remains to identify $\phi(\alpha(x))$ with $\ell(x)$. Put $v=\alpha(x)$. Since $f(nv)\leq F(x^n)$, we have $\phi(v)\leq\ell(x)$. Conversely, for every $k\geq1$ and every $y\in\Gr$ with $\alpha(y)=kv$, we have
\[
 F(y)\geq\ell(y)=k\ell(x).
\]
Taking the infimum over this fiber gives $f(kv)\geq k\ell(x)$. Dividing by $k$ and passing to the limit yields $\phi(v)\geq\ell(x)$, proving \eqref{eq:stable}.

Under either lower hypothesis, Lemma~\ref{lem:directional.lower} and the identity just proved give $\phi(v)\geq a_0\|v\|$ on $\Lambda$. The bound extends to $\mathfrak v_1$ by homogeneity and continuity. If $d_\omega$ is symmetric almost surely, stationarity gives $F(x^{-1})=F(x)$. Thus $f(-v)=f(v)$ on $\Lambda$, and consequently $\phi(-v)=\phi(v)$ on $\mathfrak v_1$.
\end{proof}

\begin{lemma}[Directional ergodic limit]\label{lem:directional}
For every $x\in \Gr$, under \eqref{all},
\[
 \frac{c(x^n,\omega)}n\longrightarrow\phi(\alpha(x))
 \quad\text{almost surely and in }L^1.
\]
\end{lemma}
\begin{proof}
Kingman's theorem for the transformation $\omega\mapsto x\cdot\omega$ gives an almost-sure and $L^1$ limit $\lambda_x(\omega)$; this transformation need not be ergodic. We show that its limit is invariant under every $z\in \Gr$. The two quantities to compare are
\[
 c(x^n,z\cdot\omega)=d_\omega(z,x^nz),\qquad c(x^n,\omega)=d_\omega(e,x^n).
\]
Their starting points are a bounded word distance apart. Their endpoints are at word distance $\|x^nzx^{-n}\|_S=o(n)$: under cone scaling, $q_n(x^n)\to\exp_\infty\alpha(x)$ and $q_n(z)\to\mathlcal e$, so the scaled conjugate tends to $\mathlcal e$. All endpoints lie in $B_S(e,Rn)$ for a fixed $R$. Lemma \ref{lem:upper} gives a difference $o(n)$. On a countable intersection of Kingman events and their translates, $\lambda_x(z\cdot\omega)=\lambda_x(\omega)$ for every $z$. Ergodicity of the full $\Gr$-action makes $\lambda_x$ constant. Its value is $\lim_n F(x^n)/n=\phi(\alpha(x))$ by $L^1$ convergence and Lemma \ref{lem:stable}.
\end{proof}
Recall the definition of $d_\phi$ in \eqref{eq:def.d.phi}. Under either lower hypothesis of Lemma~\ref{lem:directional.lower}, the comparison $a_0\|v\|\leq\phi(v)\leq b_1\|v\|$ gives the bi-Lipschitz relation between $d_\infty$ and $d_\phi$. As in the introduction, write $\Phi(\mathlcal g)=d_\phi(\mathlcal e,\mathlcal g)$. The next subsection deals with one of the most relevant models in this context.

\subsection{First-Passage Percolation models} \label{sec:fpp} Hammersley and Welsh introduced the First-Passage Percolation (FPP) as a mathematical model in 1965 to study the spread of fluid through a porous medium. In FPP models, a graph with random edge weights is considered, where these weights represent the time taken for the fluid to pass through the corresponding edge. 

Let $\mathcal{G}=(V,E)$ be a graph and let $\tau=\{\tau(\mathtt e)\}_{\mathtt e\in E}$ be a collection of nonnegative random edge weights. For $\{u,v\}\in E$, write $\tau(u,v)=\tau(\{u,v\})$; the subscript $\omega$ indicates a realization of this random variable. These edge lengths, also called \textit{passage times}, give the following construction.

The random passage time of a path $\gamma \in \mathscr{P}(x,y)$ is given by $T(\gamma) = \sum_{\mathtt{e} \in \gamma} \tau(\mathtt{e})$. Let us now define the first-passage time of $y$ with the process starting at $x$ by
\[
    T(x,y) := \inf_{\gamma \in \mathscr{P}(x,y)} T(\gamma).
\]

We allow the trivial path of length zero, so $T(x,x)=0$. For finite nonnegative weights on the unoriented edges of a connected graph, $T$ is a symmetric pseudometric: concatenating paths gives the triangle inequality, and reversing a path preserves its cost. Distinct vertices may have zero passage distance.

Now let $\mathcal G=\mathcal C(\Gr,S)$ and suppose that the joint law of the edge weights is stationary under right translations. We use the translation convention
\[
 \tau_{z\cdot\omega}(u,v)=\tau_\omega(uz,vz).
\]
Translating every path from $u$ to $v$ by $z$ on the right gives $T_{z\cdot\omega}(u,v)=T_\omega(uz,vz)$. Hence $c(x,\omega):=T_\omega(e,x)$ is a subadditive cocycle: for $x,y\in\Gr$,
\[
 c(xy,\omega)
 \leq T_\omega(e,y)+T_\omega(y,xy)
 =c(y,\omega)+c(x,y\cdot\omega).
\]
In particular, $d_\omega(u,v)=T_\omega(u,v)$.

For adjacent vertices, $T_\omega(x,sx)\leq\tau_\omega(x,sx)$; the inequality can be strict if another path has smaller cost. In \cref{sec:additional.FPP} we prove the shape theorem for this passage metric on the ambient Cayley graph.

\section{The Limiting Shape}\label{sec:limiting.shape} This section begins by introducing auxiliary results for approximation
along polygonal paths. We then prove the two main theorems.

\subsection{Approximation of admissible curves along polygonal paths}\label{sec:approx.polygonal} The proof strategy involves approximating admissible curves with polygonal paths. Throughout this subsection, $\Gr$ is torsion-free nilpotent and $c$ is a subadditive cocycle satisfying the hypotheses of \cref{shape.thm}. We first establish uniform convergence along powers at moving basepoints.

\begin{lemma}
\label{lem:good}
Let an infinite finitely generated group $\Gr$ of polynomial growth act ergodically by p.m.p. transformations, and let $E_m\uparrow\Om$ modulo null sets. For every $R,\delta>0$ there exists a deterministic $m$ such that almost surely, for all sufficiently
large $n$, every $u\in B_S(e,Rn)$ has a point $v$ with
\[
 d_S(u,v)\leq\delta n,\qquad v\cdot\omega\in E_m.
\]
\end{lemma}
\begin{proof}
Word balls in a finitely generated group of polynomial growth form a tempered F\o lner sequence. Indeed $|B_S(e,n)|\sim v_S n^D$ by \cite[Theorem~1.1]{breuillard2014}, bounded translates
are trapped between radii $n-O(1)$ and $n+O(1)$, and $\bigcup_{j<n}B_S(e,j)^{-1}B_S(e,n)\subseteq B_S(e,2n)$. The pointwise ergodic theorem for tempered F\o lner sequences \cite{lindenstrauss2001} therefore applies to $\mathbbm{1}_{E_m^c}$.

There are constants $b_\delta>0$ and $B_{R,\delta}<\infty$ such that for all large $n$, $|B_S(u,\delta n)|\geq b_\delta n^D$ and $|B_S(e,(R+\delta)n)|\leq B_{R,\delta}n^D$. Choose $m$ with $\p(E_m^c)<b_\delta/(2B_{R,\delta})$. The ergodic theorem implies that fewer than $b_\delta n^D$ points of the larger ball have $v\cdot\omega\notin E_m$, eventually. No entire ball $B_S(u,\delta n)$ inside it can consist of bad points. This proves the claim simultaneously for all its possible centers $u$.
\end{proof}

\begin{lemma}[Uniform lower control]\label{lem:lower}
Let $\Gr$ be an infinite finitely generated group of polynomial growth, with an ergodic p.m.p. action. Under \eqref{all} and \eqref{small.prob}, for every $R>0$, almost surely,
\begin{equation}\label{eq:escape-pairs}
 d_\omega(u,v)\geq b\,d_S(u,v)-o_R(n)
 \qquad(u,v\in B_S(e,Rn)),
\end{equation}
where the error is uniform over the endpoints and $b$ is the constant in \eqref{small.prob}.
\end{lemma}
\begin{proof}
The events $E_M=\{c(y)\geq b\|y\|_S\text{ whenever }\|y\|_S\geq M\}$ increase to a set of full measure. Fix rational $\delta\in(0,1)$. Lemma~\ref{lem:good} supplies a deterministic $M$ such that eventually every $u\in B_S(e,Rn)$ has a point $w$ within $\delta n$ with $w\cdot\omega\in E_M$. Then
\[
 d_\omega(w,v)\geq b\,d_S(w,v)-bM
\]
even when $d_S(w,v)<M$. The directed triangle inequality gives
$d_\omega(u,v)\geq d_\omega(w,v)-d_\omega(w,u)$. Lemma~\ref{lem:upper} bounds the subtracted term by $\upbeta\delta n+\eta n$ eventually, for each rational $\eta>0$. Thus the uniform deficit in \eqref{eq:escape-pairs}, divided by $n$, has limsup at most $(b+\upbeta)\delta+\eta$. Let $\delta,\eta\downarrow0$ along rational sequences and intersect over integer $R$. No innerness assumption is used.
\end{proof}

\begin{lemma}
\label{lem:uniform.powers}
Under \eqref{all}, for each fixed $y\in \Gr$ and $R,L<\infty$, almost surely,
\begin{equation}\label{eq:uniform.powers}
 \sup_{\substack{\|u\|_S\leq Rn\\0\leq m\leq Ln,\ m\in\Z}}
 \frac{|c(y^m,u\cdot\omega)-m\phi(\alpha(y))|}{n}
 \longrightarrow0.
\end{equation}
All these assertions hold simultaneously for countably many $y$ and integer $R,L$. In particular they imply the polygonal ergodic theorem with basepoints $y_{j-1}^n\cdots y_1^n$.
\end{lemma}
\begin{proof}
Fix $\upvarepsilon>0$ and put
\[
 E_M=\{\omega: |c(y^m,\omega)-m\phi(\alpha(y))|\leq\upvarepsilon m
                         \text{ for every integer }m\geq M\}.
\]
By Lemma \ref{lem:directional}, $\p(E_M)\uparrow1$. Fix $0<\delta<1$. By Lemma~\ref{lem:good}, we may choose a deterministic $M$ such that, almost surely for all sufficiently large $n$, every $u$ with $\|u\|_S\leq Rn$ admits a point $v$ with $d_S(u,v)\leq\delta n$ and $v\cdot\omega\in E_M$. Uniformly for such $u,v$ and integers $0\leq m\leq Ln$,
\begin{equation}\label{eq:conjugation.modulus}
 \tfrac{1}{n}d_S(y^mu,y^mv)\leq\rho_{y,R,L}(\delta)+o(1),
 \qquad\rho_{y,R,L}(\delta)\longrightarrow0\quad(\delta\downarrow0).
\end{equation}
To see this, pass to cone coordinates on the compact set containing $q_n(y^m),q_n(u),q_n(v)$ and use uniform continuity of multiplication and inversion together with \eqref{eq:cone.metric}. Choose $\rho_{y,R,L}$ to be a continuous nondecreasing majorant of the resulting modulus, so an input bound $\delta+o(1)$ gives $\rho_{y,R,L}(\delta)+o(1)$. Left multiplication need not preserve the right-invariant word metric.

Fix $\eta>0$. Since $c(y^m,u\cdot\omega)=d_\omega(u,y^mu)$ and $c(y^m,v\cdot\omega)=d_\omega(v,y^mv)$, the four-point estimate \eqref{eq:endpoint} in Lemma~\ref{lem:upper} bounds $\tfrac{1}{n}|c(y^m,u\cdot\omega)-c(y^m,v\cdot\omega)|$ by $\upbeta(\delta+\rho_{y,R,L}(\delta))+2\eta+o(1)$. For $m\geq M$, the condition $v\cdot\omega\in E_M$ gives $|c(y^m,v\cdot\omega)-m\phi(\alpha(y))|\leq\upvarepsilon m\leq\upvarepsilon Ln$. For the finitely many $0\leq m<M$, Lemma~\ref{lem:upper} bounds $c(y^m,u\cdot\omega)$ by $\upbeta M\|y\|_S+\eta n$, while $m\phi(\alpha(y))\leq M\phi(\alpha(y))$ is bounded independently of $n$. Take $n\to\infty$, and then let $\eta\downarrow0$, $\delta\downarrow0$, and $\upvarepsilon\downarrow0$. Restricting these parameters to rational sequences gives a single full-measure event.
\end{proof}

\begin{lemma}
\label{lem:cost.clock}
Let $\Gr$ be a finitely generated torsion-free nilpotent group, with word metric $d_S$, rescaling maps $q_n$, and word-cone metric $d_\infty$. For each $n$, let $x_{n,0},\ldots,x_{n,k_n}\in B_S(e,Rn)$ be a finite chain, and let
\[
 0=S_{n,0}\le S_{n,1}\le\cdots\le S_{n,k_n}\le Cn,
\]
where $R,C>0$ are fixed. For an empty chain, interpret the maximum spatial jump as zero. Suppose that, for some $b>0$ and nonnegative sequences $h_n,a_n=o(n)$,
\begin{align}
 \max_{1\le j\le k_n}d_S(x_{n,j-1},x_{n,j})&\le h_n,
       \label{eq:clock.mesh}\\
 b\,d_S(x_{n,i},x_{n,j})&\le S_{n,j}-S_{n,i}+a_n
       \qquad(0\le i\le j\le k_n).
       \label{eq:clock.coercivity}
\end{align}
Assume that $q_n(x_{n,0})\to g_0$ and $q_n(x_{n,k_n})\to g_1$.
Define
\[
 I_n(t)=\max\{j:S_{n,j}\le Cnt\},\qquad
 \zeta_n(t)=q_n(x_{n,I_n(t)}),\qquad 0\le t\le1.
\]
Then every subsequence has a further subsequence on which $\zeta_n$ converges uniformly to a $d_\infty$-Lipschitz horizontal curve from $g_0$ to $g_1$, with Lipschitz constant at most $C/b$. No bound on $k_n$ or on the individual cost increments is required.
\end{lemma}

\begin{proof}
Fix $0\le s<t\le1$. If $I_n(s)=I_n(t)$, the selected vertices agree. Otherwise put $i=I_n(s)$ and $j=I_n(t)$. Omit the first spatial increment after $i$. Since $S_{n,i+1}>Cns$ and $S_{n,j}\le Cnt$, \eqref{eq:clock.mesh} and \eqref{eq:clock.coercivity} give
\[
 \frac1n d_S(x_{n,I_n(s)},x_{n,I_n(t)})  \le \frac{h_n}{n}+\frac{C}{b}(t-s)+\frac{a_n}{bn}.
\]
This estimate also holds when $I_n(s)=I_n(t)$. Uniform metric convergence on the rescaled word ball therefore yields numbers $r_n\to0$ such that
\[
 d_\infty(\zeta_n(s),\zeta_n(t)) \le \frac{C}{b}|t-s|+r_n \qquad(s,t\in[0,1]).
\]
The images lie in one compact $d_\infty$-ball. From any subsequence, extract a further subsequence converging at every point of a countable dense set of times containing $0$ and $1$. The limits satisfy the $C/b$ Lipschitz bound and extend to a Lipschitz curve $\zeta$ on $[0,1]$. Applying the displayed estimate on finite time nets proves uniform convergence to $\zeta$, even though the maps $\zeta_n$ may be discontinuous.

The maximum convention in $I_n$ traverses each zero-cost block at one time. In particular, \eqref{eq:clock.coercivity} gives $b\,d_S(x_{n,0},x_{n,I_n(0)})\le a_n$, so $\zeta(0)=g_0$. Also $I_n(1)=k_n$, whence $\zeta(1)=g_1$. Finally, a curve Lipschitz for the Carnot metric $d_\infty$ is horizontal almost everywhere.
\end{proof}

\begin{lemma}[Geometric approximations of horizontal curves]\label{lem:polygon}
Let $\zeta:[0,1]\to G_\infty$ be $d_\infty$-Lipschitz and
$\zeta(0)=\mathlcal e$. The following two approximations hold.
\begin{enumerate}[label=(\alph*)]
\item For every $\delta>0$ there are $y_1,\ldots,y_k\in \Gr$ and $p\in\N$ such that, for $r_n=y_k^{\lfloor n/p\rfloor}\cdots y_1^{\lfloor n/p\rfloor}$,
\begin{equation}\label{eq:polygon.positions}
 \limsup_n d_\infty(q_n(r_n),\zeta(1))<\delta,\qquad \left|p^{-1}\sum_{j=1}^k\phi(\alpha(y_j)) -\ell_\phi(\zeta)\right|<\delta.
\end{equation}
\item For every $\delta>0$ there are finitely many intervals $[a_i,b_i]$ with disjoint interiors, in increasing order, and $y_i\in \Gr$, $p_i\in\N$, such that, putting $w_i=\alpha(y_i)/p_i$,
\begin{align}
 \sum_i d_\infty(\zeta(b_i),\exp_\infty(w_i)\zeta(a_i))&<\delta, \label{eq:tangent.errors}\\
 \sum_i\phi(w_i)&\geq\ell_\phi(\zeta)-\delta. \label{eq:tangent.length}
\end{align}
\end{enumerate}
\end{lemma}
\begin{proof}
Let $v(t)\in\mathfrak v_1$ be the right-trivialized horizontal velocity of $\zeta$. Since $\zeta$ is $d_\infty$-Lipschitz, $v$ is essentially bounded and
\[
 \ell_\phi(\zeta)=\int_0^1\phi(v(t))\,dt.
\]

As illustrated in Figure~\ref{fig:polygonal-approximation}, the steps in (a) are concatenated, whereas each step in (b) starts at the corresponding point $\zeta(a_i)$ on the original curve.

We first construct a polygonal curve in $G_\infty$ and then realize its steps as limits of powers in $\Gr$. This is the endpoint version of the polygonal approximation in \cite[Proposition~3.1]{cantrell2017}.

Approximate $v$ in $L^1$ by piecewise constant controls bounded by $\|v\|_{L^\infty}$. The corresponding horizontal curves starting at $\mathlcal e$ converge uniformly to $\zeta$ by continuous dependence on the control, and their $\phi$-lengths converge to $\ell_\phi(\zeta)$ because $\phi$ is Lipschitz. Choose one such control for which the endpoint and length errors are both less than $\delta/3$. If the control equals $v_j$ on an interval of length $h_j$, put $u_j=h_jv_j$. The endpoint and length of the resulting polygonal curve are
\[
 P=\exp_\infty(u_k)\cdots\exp_\infty(u_1), \qquad L=\sum_{j=1}^k\phi(u_j).
\]

Since $\Lambda=\alpha(\Gr)$ is a full lattice, we may approximate all the $u_j$ by vectors $w_j=\alpha(y_j)/p$ with a common denominator $p$. By continuity, we may choose these approximations so that replacing $u_j$ by $w_j$ changes both $P$ (in $d_\infty$) and $L$ by less than $\delta/3$. Now put $m_n=\lfloor n/p\rfloor$. By \eqref{eq:powers.geometry} and $m_n/n\to1/p$,
\[
 q_n(y_j^{m_n})\longrightarrow\exp_\infty(w_j).
\]
Compatibility of rescaling with finite products therefore gives
\[
 q_n(r_n)\longrightarrow \exp_\infty(w_k)\cdots\exp_\infty(w_1).
\]
Since $\phi(w_j)=p^{-1}\phi(\alpha(y_j))$, the endpoint and length estimates give \eqref{eq:polygon.positions}.

Here each step starts at the actual point $\zeta(a_i)$. Fix $0<\eta<\delta/3$. At a Lebesgue point $t$ of $v$, consider the translated and rescaled curve
\[
 \gamma_{t,h}(s) =\delta_{1/h}\bigl(\zeta(t+hs)\zeta(t)^{-1}\bigr), \qquad 0\le s\le1.
\]
It starts at $\mathlcal e$ and has control $s\mapsto v(t+hs)$, which converges in $L^1([0,1])$ to the constant $v(t)$ as $h\downarrow0$. Thus $\gamma_{t,h}(1)\to\exp_\infty(v(t))$. By right invariance and homogeneity of $d_\infty$, this says precisely that
\[
 d_\infty\bigl(\zeta(t+h),\exp_\infty(hv(t))\zeta(t)\bigr)=o(h).
\]
At the same point, Lipschitz continuity of $\phi$ gives
\[
 \left|\int_t^{t+h}\phi(v(s))\,ds-h\phi(v(t))\right|=o(h).
\]

The intervals $[t,t+h]\subset[0,1]$ on which both errors are at most $\eta h$ form a Vitali cover of a full-measure subset of $(0,1)$. Choose a disjoint subfamily covering that set up to a null set. Since $\phi(v)$ is integrable, we may retain finitely many of these intervals while omitting less than $\eta$ of the $\phi$-length. List the retained intervals $[a_i,b_i]$ in increasing order and put $u_i=(b_i-a_i)v(a_i)$. Summing the two local estimates and using $\sum_i(b_i-a_i)\le1$, we obtain
\[
 \sum_i d_\infty\bigl(\zeta(b_i),\exp_\infty(u_i)\zeta(a_i)\bigr) \le\eta, \qquad \sum_i\phi(u_i)\ge\ell_\phi(\zeta)-2\eta.
\]
Finally, there are only finitely many $u_i$. By density of $\bigcup_{p\ge1}p^{-1}\Lambda$ and continuity, we may replace them by $w_i=\alpha(y_i)/p_i$ so that each of the two sums changes by less than $\eta$. Since $3\eta<\delta$, this proves \eqref{eq:tangent.errors} and \eqref{eq:tangent.length}.
\end{proof}

\begin{figure}[H]
\centering
\begin{tikzpicture}[
  x=1cm,y=1cm,font=\footnotesize,
  curve/.style={draw=black!70,line width=.85pt},
  poly segment/.style={draw=teal!72!black,line width=1.15pt,line join=round},
  gap/.style={draw=black,dash pattern=on 0pt off 2.1pt,line cap=round,line width=1pt},
  vertex/.style={circle,draw=teal!72!black,fill=white,inner sep=1.65pt},
  point/.style={circle,fill=black!75,inner sep=1.35pt}]
\def\tracecurve{(0,0) .. controls (.55,.47) and (.94,1.56) .. (1.78,1.60)
 .. controls (2.05,1.61) and (2.23,1.39) .. (2.50,1.20)
 .. controls (2.77,1.01) and (3.22,.62) .. (3.60,.65)
 .. controls (4.14,.76) and (4.60,1.65) .. (5.12,1.74)}
\def\landmarks{%
  \coordinate (O) at (0,0);
  \coordinate (B) at (1.78,1.60);
  \coordinate (U) at (1.78,1.18);
  \coordinate (A) at (3.60,.65);
  \coordinate (V) at (5.12,1.28);
  \coordinate (G) at (5.12,1.74);}
\def\commonpoints{%
  \foreach \v in {O,B,A,G} \node[point] at (\v) {};
  \foreach \v in {U,V} \node[vertex] at (\v) {};}
\begin{scope}
\node[anchor=west,font=\footnotesize\bfseries] at (-.18,2.95)
 {(a) Concatenated steps};
\landmarks
\draw[curve] \tracecurve;
\draw[poly segment] (O)--(U)--(2.36,1.37)--(A)--(V);
\draw[gap] (V)--(G);
\commonpoints
\node[vertex] at (2.36,1.37) {};
\node[below] at (O) {$\mathlcal{e}$};
\node[above] at (G) {$\zeta(1)$};
\node[below right] at (V) {$P_k$};
\node[above] at (B) {$\zeta$};
\node at (2.58,-.80) {$\displaystyle
 \left|\sum_{j=1}^k\phi(w_j)-\ell_\phi(\zeta)\right|<\delta$};
\end{scope}
\begin{scope}[xshift=6.65cm]
\node[anchor=west,font=\footnotesize\bfseries] at (-.18,2.95)
 {(b) Steps on selected intervals};
\landmarks
\draw[curve] \tracecurve;
\draw[poly segment] (O)--(U) (A)--(V);
\draw[gap] (U)--(B) (V)--(G);
\commonpoints
\node[below] at (O) {$\mathlcal{e}=\zeta(a_1)$};
\node[above] at (B) {$\zeta(b_1)$};
\node[below] at (U) {$E_1$};
\node[below] at (A) {$\zeta(a_2)$};
\node[above] at (G) {$\zeta(b_2)=\zeta(1)$};
\node[below right] at (V) {$E_2$};
\node at (2.58,-.80) {$\displaystyle
 \sum_i d_\infty\bigl(E_i,\zeta(b_i)\bigr)<\delta$};
\end{scope}
\end{tikzpicture}
\caption{The two approximations in Lemma~\ref{lem:polygon}, shown
schematically in horizontal coordinates. In (a), the steps $w_j=\alpha(y_j)/p$ are
concatenated, and the dotted gap marks the error between
$P_k=\lim_{n\to\infty}q_n(r_n)$ and $\zeta(1)$. In (b), each step
$w_i=\alpha(y_i)/p_i$ starts at $\zeta(a_i)$, and the dotted gaps mark the errors between
$E_i=\exp_\infty(w_i)\zeta(a_i)$ and $\zeta(b_i)$.}
\label{fig:polygonal-approximation}
\end{figure}

\subsection{Proof of the first theorem}\label{sec:main.proofs}
This subsection is dedicated to proving \cref{shape.thm}. We retain the notation and hypotheses of that theorem, including asymptotic innerness and the radial lower bound when $\Gr$ is nonabelian.

\begin{proposition}[Radial shape]\label{prop:radial}
Under the hypotheses of Theorem~\ref{shape.thm}, there is a single event of full measure on which
\begin{equation}\label{eq:radial}
 c(x,\omega)-\Phi(q_1(x))=o(\|x\|_S)\qquad(\|x\|_S\to\infty).
\end{equation}
Equivalently, for every sequence $\|x_n\|_S=O(n)$ with $q_n(x_n)\to g$, $c(x_n,\omega)/n\to\Phi(g)$. The sequences may be selected after the outcome $\omega$ is fixed.
\end{proposition}
\begin{proof}
Work on the simultaneous full-measure event of Lemmas \ref{lem:upper} and \ref{lem:uniform.powers}, together with the events in \eqref{innerness}, \eqref{small.prob}, and Lemma~\ref{lem:lower} in the nonabelian case, for integer radii and rational approximation parameters. The group elements and parameters involved are countable. First prove the upper bound along an arbitrary sequence with $q_n(x_n)\to g$. Choose a horizontal curve from $\mathlcal e$ to $g$ of $\phi$-length at most $\Phi(g)+\upvarepsilon$. Parametrizing by $d_\infty$-arclength and rescaling the parameter interval to $[0,1]$, we may take the curve to be $d_\infty$-Lipschitz without changing its endpoints or its $\phi$-length. If $g=\mathlcal e$, take the constant curve. Apply Lemma~\ref{lem:polygon}(a) with $\delta=\upvarepsilon$, and write $m_n=\lfloor n/p\rfloor$ and $r_{n,j}=y_j^{m_n}\cdots y_1^{m_n}$, so $r_{n,k}=r_n$. All intermediate vertices have word length $O(n)$. Subadditivity and uniform convergence along powers give
\[
 \limsup_n\frac{c(r_n)}n \leq p^{-1}\sum_{j=1}^k\phi(\alpha(y_j)) \leq\Phi(g)+2\upvarepsilon.
\]
The final vertex is within $\upvarepsilon+o(1)$ of $q_n(x_n)$ in cone distance. By \eqref{eq:cone.metric} and Lemma \ref{lem:upper}, the additional cost to $x_n$ is at most $\upbeta\upvarepsilon n+o(n)$, with an arbitrarily small extra $\eta n$. Let $\upvarepsilon,\eta\downarrow0$.

For the lower bound suppose first that $\Gr$ is nonabelian. If $g=\mathlcal e$, nonnegativity suffices. Otherwise $\|x_n\|_S$ grows linearly in $n$. Let $b$ be as in \eqref{small.prob}. Lemma~\ref{lem:lower} gives the uniform estimate \eqref{eq:escape-pairs} on every fixed word ball.

Choose $R_0\in\N$ such that $\|x_n\|_S<R_0n$ for all large $n$, and put $\rho_n=\upvarepsilon_{R_0}(n,\omega)n=o(n)$. Take chains $x_{n,0}=e,\ldots,x_{n,k_n}=x_n$ from \eqref{innerness}, with maximum word jump at most $\rho_n$ and total cost at most $c(x_n)+\rho_n$. Denote their cumulative costs by $S_{n,j}$, starting at zero. By Lemma~\ref{lem:upper}, $S_{n,k_n}\leq Cn$ eventually for a fixed $C>0$. Each vertex has $c(x_{n,j})\leq S_{n,j}$, so \eqref{small.prob} confines the whole chain to a fixed word ball of radius $Rn$ for all large $n$. Subadditivity and \eqref{eq:escape-pairs} give, uniformly for $i\leq j$,
\[
 b\,d_S(x_{n,i},x_{n,j})\leq S_{n,j}-S_{n,i}+o_R(n).
\]

Choose a subsequence along which $c(x_n)/n$ converges to its limit inferior, and set
\[
 I_n(t)=\max\{j:S_{n,j}\le Cnt\},\qquad 0\le t\le1.
\]
The chains satisfy Lemma~\ref{lem:cost.clock}: their mesh is at most $\rho_n=o(n)$, and the preceding displacement estimate supplies the uniform error $a_n=o_R(n)$. Passing to a further subsequence, the maps $t\mapsto q_n(x_{n,I_n(t)})$ converge uniformly to a $d_\infty$-Lipschitz horizontal curve $\zeta$ from $\mathlcal e$ to $g$. In particular, the lemma handles zero-cost blocks and arbitrary numbers of increments. Fix $\delta>0$ and apply Lemma~\ref{lem:polygon}(b) to $\zeta$. Let $[a_i,b_i]$ be the selected intervals and $w_i=\alpha(y_i)/p_i$ the corresponding horizontal increments. Put
\[
 u_{n,i}=x_{n,I_n(a_i)},\qquad v_{n,i}=x_{n,I_n(b_i)},\qquad m_{n,i}=\lfloor n/p_i\rfloor,
\]
with $I_n(1)=k_n$. 

For each $i$, the selected block uses the increments with indices $I_n(a_i)<j\leq I_n(b_i)$. These index sets are disjoint, so nonnegativity allows us to discard all increments outside their union. By \eqref{innerness} and subadditivity on each retained block,
\begin{align*}
 c(x_n)+\rho_n
 &\geq \sum_{j=1}^{k_n}
 c(x_{n,j}x_{n,j-1}^{-1},x_{n,j-1}\cdot\omega)\\
 &\geq \sum_i\sum_{I_n(a_i)<j\leq I_n(b_i)}
 c(x_{n,j}x_{n,j-1}^{-1},x_{n,j-1}\cdot\omega)\\
 &\geq \sum_i d_\omega(u_{n,i},v_{n,i}).
\end{align*}
It remains to bound these endpoint distances from below. The intervals were chosen so that the corresponding tangent steps satisfy
$\sum_i\phi(w_i)\geq\ell_\phi(\zeta)-\delta$ by \eqref{eq:tangent.length}. Let $v'_{n,i}=y_i^{m_{n,i}}u_{n,i}$. Uniform curve convergence, \eqref{eq:powers.geometry}, and uniform multiplication convergence give
\[
 q_n(v'_{n,i})\longrightarrow \exp_\infty(w_i)\zeta(a_i),\qquad q_n(v_{n,i})\longrightarrow\zeta(b_i).
\]
Consequently, by \eqref{eq:cone.metric} and \eqref{eq:tangent.errors}, $\limsup_n n^{-1}\sum_i d_S(v_{n,i},v'_{n,i})<\delta$. All endpoints stay in a fixed word ball of radius $O(n)$; here the finitely many $y_i$ are fixed before letting $n$ tend to infinity. Lemma \ref{lem:upper} and the triangle inequality imply
\[
 \liminf_n\frac1n\sum_i d_\omega(u_{n,i},v_{n,i}) \geq\lim_n\frac1n\sum_i d_\omega(u_{n,i},v'_{n,i})-\upbeta\delta.
\]
To justify this last line, first retain an additional $\eta n$ error for each of the finitely many connections, let $n\to\infty$, and then let $\eta\downarrow0$. The uniform-powers lemma identifies the limit on the right with $\sum_i\phi(w_i)$. Thus
\[
 \liminf_n\frac{c(x_n)}n \geq\ell_\phi(\zeta)-(1+\upbeta)\delta \geq\Phi(g)-(1+\upbeta)\delta.
\]
Let $\delta\downarrow0$; the additional chain error $\rho_n/n$ has already vanished. The curve and approximating tangent segments may depend on $\omega$: all uniform power assertions were already imposed for every group element on the same event, so there is no uncountable-intersection step.

If $\Gr$ is abelian, $G_\infty=\mathfrak{v}_1$ and $\Phi=\phi$. Approximate $g$ by $\alpha(y)/p$ with $y\in \Gr$ and $p\in\N$. Compare $x_n$ in both directions with $y^{\lfloor n/p\rfloor}$ using Lemma \ref{lem:upper}. The directional limit and continuity of $\phi$ give the upper and lower bounds as the approximation error tends to zero. Neither \eqref{innerness} nor \eqref{small.prob} is used.

Finally, if \eqref{eq:radial} failed, choose violating points $x_k$ with $n_k=\|x_k\|_S\uparrow\infty$. By \eqref{eq:cone.metric}, $d_\infty(\mathlcal e,q_{n_k}(x_k))=1+o(1)$, so compactness of closed $d_\infty$-balls gives a convergent subsequence. The preceding sequence argument applies unchanged along any subsequence of integer scales. Its conclusion and homogeneity $\Phi(q_{n_k}(x_k))=n_k^{-1}\Phi(q_1(x_k))$ contradict the violation. Conversely, \eqref{eq:radial} and finiteness of word balls imply that, for every $\eta>0$, there is $C_\eta(\omega)<\infty$ such that
\[
 |c(x,\omega)-\Phi(q_1(x))|\leq\eta\|x\|_S+C_\eta(\omega) \qquad\text{for all }x\in\Gr.
\]
If $\|x_n\|_S\leq Rn$ and $q_n(x_n)\to g$, divide by $n$ and use $n^{-1}\Phi(q_1(x_n))=\Phi(q_n(x_n))\to\Phi(g)$. The resulting limsup error is at most $\eta R$. Letting $\eta\downarrow0$ proves the converse.
\end{proof}

\begin{corollary}[Shape of the reached set]\label{cor:forward.shape}
Under the hypotheses of Theorem~\ref{shape.thm}, let
\[
 A_\omega(n)=\{x\in\Gr:c(x,\omega)\leq n\}, \qquad K_\phi=\{g\in G_\infty:d_\phi(\mathlcal e,g)\leq1\}.
\]
Then, almost surely,
\[
 d_H^{\,d_\infty}\bigl(q_n(A_\omega(n)),K_\phi\bigr) \longrightarrow 0,
\]
where $d_H^{\,d_\infty}$ denotes Hausdorff distance in $(G_\infty,d_\infty)$.
\end{corollary}

\begin{proof}
Work on the event of Proposition~\ref{prop:radial}. Since $\Phi$ is positive on the compact $d_\infty$-unit sphere, the radial estimate and \eqref{eq:cone.metric} give $c(x,\omega)\geq a_*\|x\|_S$ for all sufficiently large $\|x\|_S$, for some $a_*>0$. Thus the rescaled reached sets lie in a fixed compact ball. If $x_n\in A_\omega(n)$ and $q_n(x_n)\to g$ along a subsequence, the radial estimate gives $\Phi(g)\leq1$. This proves the outer Hausdorff bound.

Conversely, let $\Phi(g)<1$. Choose $x_n\in\Gr$ with $q_n(x_n)\to g$ and $\|x_n\|_S=O(n)$. Then $c(x_n,\omega)/n\to\Phi(g)<1$, so $x_n\in A_\omega(n)$ eventually. By homogeneity, $\delta_r g\to g$ as $r\uparrow1$ and $\Phi(\delta_r g)=r\Phi(g)$. Hence $\{\Phi<1\}$ is dense in $K_\phi$. Taking a finite net of the compact set $K_\phi$ with centers in $\{\Phi<1\}$ proves the inner Hausdorff bound.
\end{proof}
\begin{proposition}[Two-point shape]\label{prop:two.point}
Under the assumptions of Proposition \ref{prop:radial}, for every $R>0$, almost surely,
\begin{equation}\label{eq:two.point}
 \sup_{u,v\in B_S(e,Rn)}
 \left|\tfrac{1}{n}d_\omega(u,v)-d_\phi(q_n(u),q_n(v))\right|
 \longrightarrow0.
\end{equation}
Moreover, there is a deterministic $a_*>0$ such that, almost surely, $c(x,\omega)\geq a_*\|x\|_S$ for all sufficiently large $\|x\|_S$.
\end{proposition}
\begin{proof}
Fix $R>0$ and $\upvarepsilon>0$. For $M\geq1$, define
\[
 E_M=\bigl\{\omega: |c(x,\omega)-\Phi(q_1(x))|\leq\upvarepsilon \|x\|_S \text{ for every }x\in\Gr\text{ with }\|x\|_S\geq M\bigr\}.
\]
By Proposition~\ref{prop:radial}, $\p(E_M)\uparrow1$. Fix $\delta>0$. By Lemma~\ref{lem:good}, we may choose $M$ so that, almost surely for all sufficiently large $n$, every $u\in B_S(e,Rn)$ admits a point $w$ with $d_S(u,w)\leq\delta n$ and $w\cdot\omega\in E_M$.

Let $u,v\in B_S(e,Rn)$, choose such a $w$, and put $x=vw^{-1}$. Then $\|x\|_S=d_S(w,v)\leq(2R+\delta)n$ and $d_\omega(w,v)=c(x,w\cdot\omega)$. If $\|x\|_S\geq M$, the definition of $E_M$ and the identity $\Phi(q_n(x))=n^{-1}\Phi(q_1(x))$ give
\[
 \left|\frac{1}{n}d_\omega(w,v)-\Phi(q_n(x))\right| \leq\upvarepsilon\frac{\|x\|_S}{n} \leq\upvarepsilon(2R+\delta).
\]
If $\|x\|_S<M$, Lemma~\ref{lem:upper} gives, for every $\eta>0$,
\[
 0\leq\frac{1}{n}d_\omega(w,v) \leq\frac{\upbeta M}{n}+\eta
\]
uniformly for all sufficiently large $n$. Also, $\Phi(q_n(x))=n^{-1}\Phi(q_1(x))$ tends uniformly to zero on the finite set $\{x\in\Gr:\|x\|_S<M\}$. Taking the limsup and then letting $\eta\downarrow0$ shows that the difference tends uniformly to zero in this case.

To identify the limiting distance, observe that right invariance gives
\[
 d_\phi(q_n(w),q_n(v)) =\Phi\bigl(q_n(v)q_n(w)^{-1}\bigr).
\]
By uniform convergence of multiplication on rescaled word balls, $q_n(vw^{-1})$ and $q_n(v)q_n(w)^{-1}$ are $o(1)$ apart in $d_\infty$, uniformly over the endpoints under consideration. Since $\Phi$ is $d_\infty$-Lipschitz, it follows that
\[
 \Phi(q_n(x))=d_\phi(q_n(w),q_n(v))+o(1).
\]

It remains to return from $w$ to $u$. Since only the initial endpoint changes, the triangle inequality gives
\[
 |d_\omega(u,v)-d_\omega(w,v)| \leq\max\{d_\omega(u,w),d_\omega(w,u)\} \leq\upbeta\delta n+\eta n
\]
for all sufficiently large $n$, by Lemma~\ref{lem:upper}. Likewise, the comparison of $d_\phi$ with $d_\infty$ and
\eqref{eq:cone.metric} give
\[
 \bigl|d_\phi(q_n(u),q_n(v))-d_\phi(q_n(w),q_n(v))\bigr| \leq C d_\infty(q_n(u),q_n(w)) \leq C\delta+o(1).
\]
Combining these estimates yields
\[
 \limsup_{n\to\infty}
 \sup_{u,v\in B_S(e,Rn)} \left|\frac{1}{n}d_\omega(u,v)-d_\phi(q_n(u),q_n(v))\right| \leq \upvarepsilon(2R+\delta)+(\upbeta+C)\delta+\eta.
\]
Letting $\eta,\delta,\upvarepsilon\downarrow0$ along rational sequences proves \eqref{eq:two.point}.

Finally, the lower comparison of $d_\phi$ with $d_\infty$ gives a deterministic constant $a>0$ such that
\[
 \Phi(g)=d_\phi(\mathlcal e,g)\geq a\,d_\infty(\mathlcal e,g) \qquad\text{for all }g\in G_\infty.
\]
By \eqref{eq:cone.metric} with $n=\|x\|_S$, $d_\infty(\mathlcal e,q_{\|x\|_S}(x))=1+o(1)$ as $\|x\|_S\to\infty$. Hence \eqref{eq:radial} and homogeneity give
\[
 \liminf_{\|x\|_S\to\infty}\frac{c(x,\omega)}{\|x\|_S} =\liminf_{\|x\|_S\to\infty}\Phi(q_{\|x\|_S}(x))\geq a \qquad\p\text{-almost surely}.
\]
Taking $a_*=a/2$ proves the asserted eventual lower bound.
\end{proof}

\begin{remark}\label{rmk:zd}
On $\Z^d$, the radial conclusion of Proposition~\ref{prop:radial} has the same form as Bj\"orklund's theorem for stationary ergodic symmetric random semimetrics \cite[Theorem~1.2]{bjoerklund2010}. For a nonnegative random variable $X$, the Lorentz condition is
\[
 X\in L^{D,1} \quad\Longleftrightarrow\quad \int_0^\infty\p(X>t)^{1/D}\,dt<\infty.
\]
In general, $L^{D,1}\subsetneq L^D$ for $D\geq2$, whereas $L^{1,1}=L^1$. Condition~\eqref{all} implies $c(x)\in L^D$ for every fixed endpoint. For $D\geq2$, it need not imply generator $L^{D,1}$ integrability. It also requires a common tail envelope over all endpoints above a linear word-distance threshold. This uniform requirement differs from Bj\"orklund's generator hypothesis; the inclusion of the integrability spaces alone does not order the full assumptions.

For Proposition~\ref{prop:two.point}, the environment in $d_\omega(u,v)=c(v-u,u\cdot\omega)$ moves with $u$. Lemma~\ref{lem:good} finds a nearby vertex where the radial estimate holds with a common threshold, and uniform upper control bounds the connecting costs in both directions. This transfers the estimate to every starting point and gives the uniform comparison for all ordered pairs, including asymmetric costs.
\end{remark}

\begin{proof}[Proof of \cref{shape.thm}]
Lemma~\ref{lem:stable} gives the stated characterization of $\phi$; the horizontal bounds give the asserted comparison with $d_\infty$. Proposition~\ref{prop:two.point} proves \eqref{eq:thm.nilpotent.two.point}, and Corollary~\ref{cor:forward.shape} proves the reached-set convergence. Fix $R>0$ and $0<\eta<R$. Radial coercivity in Proposition~\ref{prop:two.point} places the $\tfrac{1}{n}d_\omega^{\rm s}$-ball of radius $R$ in a word ball of radius $Cn$, eventually. In fact, every closed $d_\omega^{\rm s}$-ball about $e$ is finite, and the triangle inequality gives the same conclusion for balls about any point. Thus the zero-distance quotient is proper at every scale $n$. On this word ball, \eqref{eq:two.point} gives vanishing directed distortion for $q_n$, hence vanishing distortion for the symmetrizations. For the quotient, choose one representative per class, choosing $e$ at the basepoint; changing representatives in this ball changes images by $o(1)$.

Take a finite $\eta/3$-net of the closed $d_\phi^{\rm s}$-ball of radius $R-\eta$. Approximate each net point by $q_n(x_n)$ with $\|x_n\|_S=O(n)$ using word-cone convergence. Applying \eqref{eq:two.point} to $(e,x_n)$ and $(x_n,e)$ on a sufficiently large word ball gives $\tfrac{1}{n}d_\omega^{\rm s}(e,x_n)\leq R-\eta+o(1)$. The approximating points thus lie in the required source ball and their images are eventually $\eta$-dense in the smaller target ball. This proves \eqref{eq:GH:conv.THM}. In the symmetric case $\phi$ is symmetric by Lemma~\ref{lem:stable}.
\end{proof}
\subsection{Proof of the second theorem}\label{sec:virt.nilpotent.proofs}
Let $\Gr$ now be a group of polynomial growth, with the normal torsion-free nilpotent subgroup $N$ and transversal fixed in the introduction. We now use $q_t$ only for points of $N$ and $Q_t=q_t\circ\uppi_N$ for points of $\Gr$, as specified before Theorem~\ref{thm:shape.polynomial}. All random variables remain on the original probability space.

\begin{lemma}[Finite-index ergodic components]\label{lm:erg.finite.index}
Let $\Gr$ act ergodically by p.m.p. transformations and let $N\unlhd\Gr$ have finite index $h$. Modulo null sets, the $N$-invariant sigma-algebra has a finite partition $\mathfrak B_N$ into $k\leq h$ atoms. The restricted action is ergodic on each $(B,\F|_B,\p_B)$, where $\F|_B=\{E\cap B:E\in\F\}$ is the trace sigma-algebra and $\p_B=\p(\,\cdot\mid B)$. The $\Gr$-action permutes the atoms transitively, and every atom has probability $1/k\geq1/h$.
\end{lemma}
\begin{proof}
For an $N$-invariant event $A$ of positive probability, the union of its $h$ transversal translates is $\Gr$-invariant modulo null sets and has probability one. Thus $\p(A)\geq1/h$. Starting with $\Om$, split any non-atomic positive invariant set into two positive invariant subsets. This procedure terminates after at most $h$ atoms, since disjoint positive invariant sets each have measure at least $1/h$. Their being atoms is exactly ergodicity of the restricted action under the conditional measure. Normality implies that $\Gr$ permutes these atoms. More than one orbit would give a nontrivial $\Gr$-invariant event, so the permutation is transitive. Measure preservation gives their equal probabilities.
\end{proof}

\begin{lemma}
\label{lm.T.bounded.nilpotent.as}
Under \eqref{all}, for every $R>0$, almost surely,
\[
 \mathcal E_n(R,\omega):=\max_{x\in B_S(e,Rn)} \max\{d_\omega(x,\uppi_N(x)),d_\omega(\uppi_N(x),x)\}=o(n).
\]
This assertion uses the two directed costs when $d_\omega$ is asymmetric. Consequently
\[
 \sup_{x\in B_S(e,Rn)}|c(x)-c(\uppi_N(x))|=o(n).
\]
\end{lemma}
\begin{proof}
Let $L=\max_i\|z_{(i)}\|_S$. For $x=z_{(i)}u$ with $u=\uppi_N(x)$,
\[
 d_S(x,u)=\|xu^{-1}\|_S=\|z_{(i)}\|_S\leq L, \qquad \|u\|_S\leq\|x\|_S+L.
\]
For $x\in B_S(e,Rn)$, both endpoints therefore lie in $B_S(e,(R+1)n)$ for all large $n$. Lemma~\ref{lem:upper}, whose proof also applies to $\Gr$, gives for every rational $\eta>0$, eventually and uniformly in $x$,
\[
 \max\{d_\omega(x,u),d_\omega(u,x)\}\leq\upbeta L+\eta n.
\]
Taking the limsup after division by $n$, and then letting $\eta\downarrow0$ through rational values, proves $\mathcal E_n(R,\omega)=o(n)$. Intersect over integer $R$. The two triangle inequalities from $e$ prove the asserted comparison of $c(x)$ and $c(\uppi_N(x))$.
\end{proof}

\begin{proposition}[Two-endpoint comparison]\label{prop:finite.index.metric}
Under \eqref{all}, for every $R>0$, almost surely,
\[
 \sup_{x,y\in B_S(e,Rn)}
 |d_\omega(x,y)-d_\omega(\uppi_N(x),\uppi_N(y))|=o(n).
\]
\end{proposition}
\begin{proof}
Apply the triangle inequality through $\uppi_N(x)$ and $\uppi_N(y)$, and then in the reverse direction through $x$ and $y$. The absolute difference is at most $2\mathcal E_n(R,\omega)$, which is $o(n)$ by the preceding lemma.
\end{proof}

\begin{proposition}[Radial shape from ambient chains]
\label{prop:ambient.chains}
Let $\Gr$, $N$, and $\mathfrak B_N$ be as above. Suppose that the ambient cocycle satisfies \eqref{all}, \eqref{small.prob}, and asymptotic innerness \eqref{innerness} on $\Gr$. On each $N$-ergodic component $B$, there is a deterministic horizontal norm $\phi_B$, possibly asymmetric, such that, $\p_B$-almost surely,
\[
 c(u,\omega)-\Phi_B(q_1(u))=o(\|u\|_{S_N}) \qquad (u\in N,\ \|u\|_{S_N}\to\infty), \qquad \Phi_B(g)=d_{\phi_B}(\mathlcal e,g).
\]
Moreover, for every $R>0$, $\p_B$-almost surely,
\[
 \sup_{u,v\in B_{S_N}(e,Rn)} \left|\frac{d_\omega(u,v)}n -d_{\phi_B}(q_n(u),q_n(v))\right|\longrightarrow0.
\]
The norm is characterized by
\[
 \phi_B(\alpha(u)) =\lim_{n\to\infty}\frac{\E_B[c(u^n)]}{n} \qquad (u\in N).
\]
In particular, $c|_N$ has a deterministic eventual linear lower bound with respect to $\|\cdot\|_{S_N}$.
\end{proposition}

\begin{proof}
Fix a component $B$. Conditioning preserves \eqref{all}: one may use the nonincreasing envelope $\min\{1,g(t)/\p(B)\}$, whose weighted tail integral is finite. The word lengths are comparable on $N$. Thus \eqref{small.prob} implies the component radial bound \eqref{small.prob2}, with a deterministic constant $b_N>0$. Lemma~\ref{lem:directional.lower}, applied to the ergodic $N$-action on $(B,\p_B)$ with horizontal projection $\alpha_N$, gives the directional lower bound \eqref{eq:directional.lower}, whether or not $N$ is abelian. Lemmas~\ref{lem:stable}, \ref{lem:directional}, and \ref{lem:uniform.powers} then give the nondegenerate norm $\phi_B$, its stated characterization, and uniform convergence along powers in $N$. These lemmas do not use innerness. The geometric Lemma~\ref{lem:polygon} also applies, with reference metric $d_{\infty,N}$.

Let $x_n\in N$ satisfy $\|x_n\|_{S_N}=O(n)$ and $q_n(x_n)\to g$. The polygonal upper-bound argument of Proposition~\ref{prop:radial} uses only uniform upper control and convergence along powers in $N$. Consequently,
\[
 \limsup_n\frac{c(x_n,\omega)}n\le\Phi_B(g).
\]
All the chains used for this upper bound lie in $N$, but their costs are the restrictions of the original ambient cocycle.

If $g=\mathlcal e$, the lower bound follows from nonnegativity. Assume henceforth that $g\ne\mathlcal e$; we use chains in $\Gr$. Put $L=\max_j\|z_{(j)}\|_S$, so $d_S(p,\uppi_N(p))\le L$ for every $p\in\Gr$, and choose $K\ge1$ with $\|u\|_{S_N}\le K\|u\|_S$ on $N$. Lemma~\ref{lem:lower} and the triangle inequality for the word metric give, for a deterministic $b_0>0$ and each fixed $R$,
\begin{equation}\label{eq:ambient.chain.displacement}
 b_0\,d_{S_N}(\uppi_N(p),\uppi_N(q)) \le d_\omega(p,q)+o_R(n) \qquad (p,q\in B_S(e,Rn)).
\end{equation}
Indeed, the left-hand word distance is at most $K\{d_S(p,q)+2L\}$, so one may take $b_0=b/K$, where $b$ is the constant in \eqref{small.prob}. The fixed term $2bL$ is absorbed in $o_R(n)$.

Choose an integer $R_0$ with $\|x_n\|_S<R_0n$ for all large $n$. Asymptotic innerness supplies a finite chain $v_{n,0}=e,\ldots,v_{n,k_n}=x_n$ whose cumulative costs
\[
 S_{n,j}=\sum_{r=1}^j d_\omega(v_{n,r-1},v_{n,r}),\qquad S_{n,0}=0,
\]
satisfy
\[
 S_{n,k_n}\le c(x_n)+\upvarepsilon_{R_0}(n,\omega)n, \qquad \max_{1\le j\le k_n}d_S(v_{n,j-1},v_{n,j}) \le \upvarepsilon_{R_0}(n,\omega)n=o(n).
\]
Uniform upper control gives $S_{n,k_n}\le Cn$ eventually, with $C=\upbeta R_0+2$. By subadditivity, $c(v_{n,j})\le S_{n,j}\le Cn$. Choose an integer $R_1>C/b+1$, where $b$ is the constant in \eqref{small.prob}. The ambient lower bound then confines every vertex of every such chain to $B_S(e,R_1n)$, eventually. These constants depend only on $R_0,\upbeta,b$ and can be chosen in advance for every integer $R_0$; all required events hold simultaneously for integer radii. The projected mesh also vanishes: if $u_{n,j}=\uppi_N(v_{n,j})$, then
\[
 h_n:=\max_{1\le j\le k_n}d_{S_N}(u_{n,j-1},u_{n,j}) \le K\{\upvarepsilon_{R_0}(n,\omega)n+2L\}=o(n).
\]
We do not compare the cost of each increment with the cost of its projection.

Start with a subsequence realizing the limit inferior of $c(x_n)/n$. For $t\in[0,1]$ set
\[
 I_n(t)=\max\{j:S_{n,j}\le Cnt\},\qquad u_n(t)=u_{n,I_n(t)}=\uppi_N(v_{n,I_n(t)}).
\]
By subadditivity on each block and \eqref{eq:ambient.chain.displacement}, the projected vertices satisfy
\[
 b_0\,d_{S_N}(u_{n,i},u_{n,j}) \le S_{n,j}-S_{n,i}+o_{R_1}(n) \qquad(0\le i\le j\le k_n),
\]
with one error uniform over all indices. Their mesh is $h_n=o(n)$, and $\|u_{n,j}\|_{S_N}\le K(R_1n+L)$ confines them to a fixed word ball of radius $O(n)$ in $N$. Lemma~\ref{lem:cost.clock}, applied to $N$ with its reference metric $d_{\infty,N}$, gives a further subsequence on which $q_n(u_n(\cdot))$ converges uniformly to a $d_{\infty,N}$-Lipschitz horizontal curve $\zeta$ from $\mathlcal e$ to $g$. Here $u_{n,0}=e$ and $u_{n,k_n}=x_n$, since the two ambient endpoints lie in $N$.

Fix $\delta>0$. Lemma~\ref{lem:polygon}(b), applied to $\zeta$, gives finitely many intervals $[s_i,t_i]$ with disjoint interiors and vectors $w_i=\alpha(y_i)/p_i$, where $y_i\in N$, such that
\[
 \sum_i d_{\infty,N} (\zeta(t_i),\exp_\infty(w_i)\zeta(s_i))<\delta, \qquad \sum_i\phi_B(w_i)\ge\ell_{\phi_B}(\zeta)-\delta.
\]
The corresponding blocks of chain indices have disjoint interiors. By nonnegativity and subadditivity, their total endpoint cost is at most $S_{n,k_n}$. Proposition~\ref{prop:finite.index.metric}, applied only at these finitely many endpoint pairs, therefore gives \begin{equation}\label{eq:ambient.chain.blocks} c(x_n)+\upvarepsilon_{R_0}(n,\omega)n \ge\sum_i d_\omega(u_n(s_i),u_n(t_i))-o_\delta(n). \end{equation} Figure~\ref{fig:ambient-endpoint-comparison} illustrates this comparison. Its number of terms is fixed before $n$ tends to infinity.

\begin{figure}[htbp]
\centering
\begin{tikzpicture}[x=.95cm,y=1cm,font=\small,
  >={Stealth[length=1.5mm]},
  ambient/.style={draw=black!58,line width=.75pt,line cap=round,line join=round},
  chosen/.style={draw=blue!60!black,line width=2.0pt,line cap=round,line join=round},
  subgroup/.style={draw=black!28,line width=.7pt},
  project/.style={->,densely dotted,draw=black!72,line width=.7pt,shorten >=2.6pt,shorten <=2.6pt},
  compare/.style={->,dashed,draw=blue!60!black,line width=.95pt,shorten >=3pt,shorten <=3pt}]
  \coordinate (p) at (2.10,2.15);
  \coordinate (q) at (4.75,2.28);
  \coordinate (r) at (6.85,2.24);
  \coordinate (s) at (9.50,2.16);

  \draw[subgroup] (-.25,.35)--(11.45,.35);
  \node[anchor=east] at (-.45,2.22) {$\Gamma$};
  \node[anchor=east] at (-.45,.35) {$N$};

  \draw[ambient] (0,.35)--(.45,1.15)--(.85,1.65)--(1.25,2.05)--(1.65,1.92)--(p);
  \draw[chosen] (p)--(2.50,2.34)--(2.88,2.18)--(3.24,2.52)--(3.65,2.30)--(4.05,2.48)--(4.40,2.14)--(q);
  \draw[ambient] (q)--(5.15,2.06)--(5.55,2.24)--(5.95,2.02)--(6.40,2.20)--(r);
  \draw[chosen] (r)--(7.25,2.46)--(7.62,2.24)--(8.00,2.44)--(8.38,2.18)--(8.76,2.36)--(9.12,2.06)--(s);
  \draw[ambient] (s)--(9.95,1.88)--(10.35,1.50)--(10.75,1.02)--(11.20,.35);

  \fill (0,.35) circle (1.6pt) node[below=3pt] {$e$};
  \fill (11.20,.35) circle (1.6pt) node[below=3pt] {$x_n$};

  \foreach \v/\xx/\tim in {p/2.10/s_1,q/4.75/t_1,r/6.85/s_2,s/9.50/t_2}{
    \draw[project] (\v)--(\xx,.35);
    \fill[blue!60!black] (\v) circle (2pt) node[above=5pt,text=black,font=\footnotesize] {$v_{n,I_n(\tim)}$};
    \draw[fill=white,line width=.8pt] (\xx,.35) circle (2pt);
    \node[below=3pt,font=\footnotesize] at (\xx,.35) {$u_n(\tim)$};
  }
  \node[right=1pt,fill=white,inner sep=1.5pt,font=\scriptsize] at (2.10,1.25) {$\uppi_N$};

  \draw[compare] (2.10,.35)--(4.75,.35)
    node[midway,above=2pt,font=\scriptsize,text=blue!60!black] {$d_\omega(u_n(s_1),u_n(t_1))$};
  \draw[compare] (6.85,.35)--(9.50,.35)
    node[midway,above=2pt,font=\scriptsize,text=blue!60!black] {$d_\omega(u_n(s_2),u_n(t_2))$};
\end{tikzpicture}
\caption{Endpoint comparison for ambient asymptotic innerness. Two selected blocks with disjoint interiors are highlighted on a near-minimizing chain in $\Gamma$ joining $e,x_n\in N$. Subadditivity bounds each block cost from below by the ambient cost between its original endpoints. The dotted arrows show the retraction $v_{n,I_n(t)}\mapsto u_n(t)=\uppi_N(v_{n,I_n(t)})$ at those endpoints only. The dashed arrows label the ambient pair costs $d_\omega(u_n(s_i),u_n(t_i))$ evaluated at the projected points in $N$, as in \eqref{eq:ambient.chain.blocks}. For fixed $\delta$, replacing the finitely many selected endpoint pairs introduces a total error $o_\delta(n)$, regardless of the number of increments in the original chain.}
\label{fig:ambient-endpoint-comparison}
\end{figure}

Put $u'_{n,i}=y_i^{\lfloor n/p_i\rfloor}u_n(s_i)$. Uniform curve convergence and the cone multiplication rules imply
\[
 \limsup_n\frac1n\sum_i d_{S_N}(u_n(t_i),u'_{n,i})\le\delta.
\]
The triangle inequality and Lemma~\ref{lem:upper} on $N$ consequently give
\[
 \liminf_n\frac1n\sum_i d_\omega(u_n(s_i),u_n(t_i)) \ge\sum_i\phi_B(w_i)-\upbeta_B\delta, \]
where $\upbeta_B$ is a valid upper-threshold constant on $(N,\p_B)$. Here Lemma~\ref{lem:uniform.powers} identifies the costs to $u'_{n,i}$, and the additive upper-control errors are sent to zero after $n\to\infty$, with the finite set of intervals held fixed. Using \eqref{eq:ambient.chain.blocks} and $\upvarepsilon_{R_0}(n,\omega)\to0$, we obtain
\[
 \liminf_n\frac{c(x_n)}n \ge\ell_{\phi_B}(\zeta)-(1+\upbeta_B)\delta \ge\Phi_B(g)-(1+\upbeta_B)\delta.
\]
Letting $\delta\downarrow0$ proves the sequence form of the radial limit. All uniform estimates were imposed for integer radii and all elements of $N$ on a common full-measure event. Thus the selected chains, subsequence, and tangent intervals may depend on $\omega$.

The compactness argument on word shells in Proposition~\ref{prop:radial} now gives the stated uniform radial estimate on $N$. The good-environment argument in the proof of Proposition~\ref{prop:two.point} uses only that radial estimate, Lemma~\ref{lem:good}, and uniform upper control. It therefore gives the stated uniform two-point convergence on $N$ as well. The eventual lower bound on $N$ was already obtained from \eqref{small.prob} at the beginning of the proof.
\end{proof}

\begin{proof}[Proof of \cref{thm:shape.polynomial}]
By Lemma~\ref{lm:erg.finite.index}, the $N$-ergodic components have positive mass and are permuted transitively by $\Gr$. Fix one such component $B$. For every $u\in N$,
\[
 \p_B(c(u)\geq t)\leq \frac{\p(c(u)\geq t)}{\p(B)}.
\]
Since $N$ has finite index, it has the same volume-growth degree as $\Gr$, and $\|\cdot\|_{S_N}$ is comparable to the restriction of $\|\cdot\|_S$ to $N$. Thus \eqref{all} holds on $(N,\p_B)$ with the nonincreasing envelope $\widetilde g_B(t)=\min\{1,g(t)/\p(B)\}$ and a corresponding word-length threshold constant. Its weighted tail integral is finite because $\widetilde g_B\leq g/\p(B)$.

Under the component assumptions, Theorem~\ref{shape.thm} and Proposition~\ref{prop:two.point}, applied to $N$ on $B$, give a deterministic horizontal norm $\phi_B$, radial coercivity on $N$, and uniform two-point convergence on $N$. When $N$ is abelian, the conditional mean bound \eqref{aml2} supplies the required lower hypothesis; otherwise, the assumptions are \eqref{small.prob2} and \eqref{innerness2}. Under the ambient assumptions \eqref{small.prob} and \eqref{innerness}, the same conclusions follow instead from Proposition~\ref{prop:ambient.chains}; no innerness of the restricted cocycle is required. In either case, Lemma~\ref{lem:stable} gives
\[
 \phi_B(\alpha(u)) =\lim_{n\to\infty}\frac{1}{n}\E_B[c(u^n)] \qquad (u\in N).
\]

We first transfer the two-point estimate to $\Gr$. Put $L=\max_j\|z_{(j)}\|_S$. If $x\in B_S(e,Rn)$, then $x=z_{(j)}\uppi_N(x)$ gives $\|\uppi_N(x)\|_S\leq Rn+L$. Hence $\|\uppi_N(x)\|_{S_N}=O_R(n)$ uniformly in $x$, and the uniform two-point estimate obtained above on $N$ applies to the projected endpoints $\uppi_N(x)$ and $\uppi_N(y)$. By Proposition~\ref{prop:finite.index.metric}, replacing the two ambient endpoints by their projections changes the passage cost by $o(n)$ uniformly on this ball. Since $Q_n=q_n\circ\uppi_N$, we obtain, $\p_B$-almost surely, for every $R>0$,
\begin{equation}\label{eq:finite.uniform}
 \sup_{x,y\in B_S(e,Rn)} \left| \frac{1}{n}d_\omega(x,y) -d_{\phi_B}\bigl(Q_n(x),Q_n(y)\bigr)
 \right|\longrightarrow 0.
\end{equation}
Here it suffices to intersect the full-measure events for positive integer $R$. Put $\Phi_B(g)=d_{\phi_B}(\mathlcal e,g)$. Taking the ordered pair $(e,x)$ in \eqref{eq:finite.uniform} with $n=\|x\|_S$ and $R=2$, and using $\Phi_B(Q_n(x))=n^{-1}\Phi_B(Q_1(x))$, gives
\begin{equation}\label{eq:finite.radial}
 c(x,\omega)-\Phi_B(Q_1(x))=o(\|x\|_S) \qquad (\|x\|_S\to\infty).
\end{equation}

We also need coercivity on the ambient group. Lemma~\ref{lm.T.bounded.nilpotent.as}, applied with $n=\|x\|_S$ and $R=2$, gives
\[
 |c(x,\omega)-c(\uppi_N(x),\omega)|=o(\|x\|_S) \qquad (\|x\|_S\to\infty).
\]
Since $\|\uppi_N(x)\|_S\geq \|x\|_S-L$, comparability of the word lengths on $N$ and radial coercivity there give a deterministic $a'_B>0$ such that, $\p_B$-almost surely, $c(x,\omega)\geq a'_B\|x\|_S$ for all sufficiently large $\|x\|_S$. Taking the minimum over the finitely many components gives a positive deterministic constant valid on $\Gr$.

We next compare the norms on different components. For $z\in\Gr$, write $zB=\tht_z(B)$, understood modulo null sets. Normality of $N$ makes $u\mapsto zuz^{-1}$ an automorphism of $N$. Extending it to the Mal'cev completion of $N$, passing to the associated graded Lie algebra, and exponentiating gives a graded automorphism
$\Psi_z:G_\infty\to G_\infty$.
Its differential at $\mathlcal e$ restricts to a linear map
$\Psi_z^{(1)}:\mathfrak v_1\to\mathfrak v_1$, with
\[
 \begin{aligned}
 \Psi_z^{(1)}\alpha(u)&=\alpha(zuz^{-1}) &&(u\in N),\\
 \Psi_z(\exp_\infty v)&=\exp_\infty(\Psi_z^{(1)}v) &&(v\in\mathfrak v_1).
 \end{aligned}
\]
In particular, $\Psi_{z^{-1}}=\Psi_z^{-1}$. For $u\in N$, put $a_n=u^nz$ and $b_n=z^{-1}u^nz$. Then
\[
 d_\omega(z,a_n)=c(u^n,z\cdot\omega), \qquad d_\omega(e,b_n)=c((z^{-1}uz)^n,\omega).
\]
Since $a_nb_n^{-1}=z$, stationarity under the ambient law
$\p$ gives
\[
 d_\omega(b_n,a_n)\overset{\rm law}=c(z),
 \qquad
 d_\omega(a_n,b_n)\overset{\rm law}=c(z^{-1}).
\]
Since $D\geq1$, \eqref{all} implies $\int_1^\infty g(t)\,dt<\infty$. Monotonicity and the integral test give $\sum_{n\geq1}g(\eta n)<\infty$ for every $\eta>0$. For all sufficiently large $n$, \eqref{all} bounds the probability that either of these fixed-displacement connection costs exceeds $\eta n$ by $2g(\eta n)$. Borel--Cantelli, applied for positive rational $\eta$, shows that both costs are $o(n)$, $\p$-almost surely and hence $\p_B$-almost surely. The two directed costs between $e$ and $z$ are finite and fixed in $n$. Applying the triangle inequality in both directions therefore gives
\[
 \bigl|d_\omega(z,a_n)-d_\omega(e,b_n)\bigr|=o(n).
\]

Under $\p_B$, the directional limit of $n^{-1}d_\omega(e,b_n)$ is $\phi_B(\alpha(z^{-1}uz))$. Since $\tht_z$ sends $\p_B$ to $\p_{zB}$, the directional limit of $n^{-1}d_\omega(z,a_n)$ is $\phi_{zB}(\alpha(u))$. The two limits agree. Intersecting over the countably many $u\in N$ and $z\in\Gr$ gives these identities on a common full-measure event. By homogeneity and continuity, they extend from $\alpha(N)$ to all horizontal vectors:
\begin{equation}\label{eq:component.norms}
 \phi_{zB}(v)=\phi_B\bigl(\Psi_{z^{-1}}^{(1)}v\bigr)
 \qquad(z\in\Gr,\ v\in\mathfrak v_1).
\end{equation}
Consequently, $\Psi_{z^{-1}}$ takes each horizontal curve with its $\phi_{zB}$-length to a horizontal curve of the same $\phi_B$-length, because the right-trivialized velocities transform by $\Psi_{z^{-1}}^{(1)}$. Taking infima over curves gives
\[
 d_{\phi_{zB}}(g,h)=d_{\phi_B}\bigl(\Psi_{z^{-1}}(g),\Psi_{z^{-1}}(h)\bigr)
 \qquad(g,h\in G_\infty).
\]
Since $\Psi_{z^{-1}}(\mathlcal e)=\mathlcal e$, this is a pointed isometry from the limit associated with $zB$ to the one associated with $B$. Transitivity on the components therefore gives a deterministic pointed isometry class. The maps satisfy $\Psi_{zz'}=\Psi_z\circ\Psi_{z'}$ and $\Psi_u=\mathrm{id}_{G_\infty}$ for $u\in N$, since inner automorphisms act trivially on the associated graded. Thus both the graded action and its horizontal action factor through $\Gr/N$.

Finally, fix $R>0$ and $0<\eta<R$. Ambient coercivity places the closed $n^{-1}d_\omega^{\rm s}$-ball of radius $R$ inside a word ball of radius $Cn$ for all sufficiently large $n$. On this ball, \eqref{eq:finite.uniform} gives vanishing directed distortion for $Q_n$, and hence vanishing distortion for the symmetrizations. As in the proof of \cref{shape.thm}, choose one representative from each zero-distance class, choosing $e$ for the basepoint. Changing representatives in the ball changes their images by $o(1)$ in $d_{\phi_B}^{\rm s}$, uniformly.

Take a finite $\eta/3$-net of the closed $d_{\phi_B}^{\rm s}$-ball of radius $R-\eta$ about $\mathlcal e$. Since $Q_n|_N=q_n$, each net point $g$ is the limit of $Q_n(u_n)$ for some $u_n\in N$ with $\|u_n\|_{S_N}=O(n)$. Applying \eqref{eq:finite.uniform} to both $(e,u_n)$ and $(u_n,e)$ gives
\[
 \frac{1}{n}d_\omega^{\rm s}(e,u_n) =d_{\phi_B}^{\rm s}(\mathlcal e,g)+o(1) \leq R-\eta+o(1).
\]
Thus all the approximating vertices lie in the source ball of radius $R$ for sufficiently large $n$, and the images of their chosen representatives are $\eta$-dense in the target ball of radius $R-\eta$. This proves \eqref{eq:GH:conv.THM2}.

For the reached-set conclusion, on each $N$-ergodic component $B$ the same argument, using \eqref{eq:finite.radial} and $Q_n|_N=q_n$, gives
\[
 d_H^{\,d_\infty}\!\left( Q_n(A_\omega(n)), \{g\in G_\infty:d_{\phi_B}(\mathlcal e,g)\leq1\} \right)\longrightarrow0 \qquad\p_B\text{-almost surely}.
\]
This proves \eqref{eq:thm.polynomial.reached}; the uniform two-point conclusion \eqref{eq:thm.polynomial.two.point} is \eqref{eq:finite.uniform}.
\end{proof}

\begin{example}[Unconditional lower growth is insufficient]
\label{ex:components}
Let $J(x_1,x_2)=(x_2,x_1)$ and $\Gr=\Z^2\rtimes_J\Z_2$, with multiplication
\[
 (v,h)(w,k)=(v+J^h w,h+k), \qquad v,w\in\Z^2,\quad h,k\in\Z_2,
\]
where the second coordinate is added modulo two. Put $N=\Z^2\times\{0\}$ and fix the symmetric generating set
\[
 S=\{(\pm e_1,0),(\pm e_2,0),r\},\qquad r=(0,1),
\]
where $e_1,e_2$ are the standard coordinate vectors. Take $\Om=\{0,1\}$ with uniform law and action $(v,h)\cdot \omega=h+\omega\pmod2$. Hence the only invariant sets are $\varnothing$ and $\Om$, so the $\Gamma$-action is ergodic. Furthermore, the action of $N$ fixes each $B_0=\{0\}$ and $B_1=\{1\}$ and they are two $N$-ergodic components. Put
\[
 p(v,h)=J^h v,\qquad f_0(v)=|v_1|+\sqrt{|v_2|},\qquad f_1(v)=\sqrt{|v_1|}+|v_2|.
\]
For $x=(v_x,h_x)$ and $y=(v_y,h_y)$, define proper metrics on $\Gr$ by
\[
 d_\omega(x,y)=f_\omega(p(y)-p(x))+\mathbbm{1}_{\{h_x\ne h_y\}}.
\]
Subadditivity of the square root gives the triangle inequality; the map $x\mapsto(p(x),h_x)$ is injective, giving separation. For $g=(w,k)$,
\[
 p(xg^{-1})=J^k p(x)-w, \qquad f_{\omega+k}(J^kz)=f_\omega(z).
\]
The metrics are therefore right-equivariant and determine a cocycle. For this generating set, $\|(v,h)\|_S=\|v\|_1+h$. Since $\sqrt m\leq m$ for nonnegative integers, $c(x,\omega)=d_\omega(e,x)\leq\|x\|_S$. Thus \eqref{all} holds with a zero tail beyond a linear threshold. Nevertheless,
\[
 \E [c((v,h))]= \tfrac12\bigl(|v_1|+|v_2|+\sqrt{|v_1|}+\sqrt{|v_2|}\bigr)+h \geq\tfrac12\|(v,h)\|_S.
\]
This unconditional lower bound holds for powers as well. On component zero, however,
\[
 c(((0,m),0),0)=\sqrt{|m|},
\]
so the conditional mean bound \eqref{aml2} fails, as do directional coercivity and word-linear compact containment. The analogous failure occurs in the other coordinate on component one. The asymptotic metric space here is a product of a line and a snowflaked line, not the normed plane asserted by a nondegenerate abelian Carnot limit: on $N$, the rescalings $(v_1,v_2)\mapsto(v_1/n,v_2/n^2)$ give the limit $|\Delta x|+\sqrt{|\Delta y|}$. This limit has Hausdorff dimension three, whereas a normed plane has dimension two. The finite coset state disappears at scale $n$ because it has bounded distance to $N$. Thus unconditional lower growth does not imply the nondegenerate abelian Carnot limit.
\end{example}

\begin{proposition}[Component norms in fixed coordinates]
\label{prop:component.coordinates}
There is a normalized nonnegative subadditive cocycle over an ergodic p.m.p.\ action of a virtually abelian group that satisfies the uniform bi-Lipschitz and fixed-increment innerness hypotheses of \cite[Theorem~A]{cantrell2017}, with one increment set for every starting environment, but whose radial cost functions are not asymptotically equivalent in fixed group coordinates. Its component limit norms differ in the same horizontal coordinates, although their pointed metric spaces are isometric.
\end{proposition}

\begin{proof}
Keep the group $\Gr$, subgroup $N$, generating set $S$, action on $\Om$, projection $p$, and metric construction of Example~\ref{ex:components}. Its two $N$-ergodic components are $B_0=\{0\}$ and $B_1=\{1\}$, each of probability $1/2$, and $r$ exchanges them. Replace the two functions by
\[
 f_0(v)=|v_1|+2|v_2|,\qquad f_1(v)=2|v_1|+|v_2|.
\]
These norms satisfy $f_{\omega+k}(J^kz)=f_\omega(z)$, so the same calculation proves right-equivariance. The radial cost $c(g,\omega)=d_\omega(e,g)$ is therefore a normalized subadditive cocycle over the same ergodic action. Let $p(v,h)=J^h v$. For $x=(v_x,h_x)$ and $y=(v_y,h_y)$, the identity
\[
 d_S(x,y)=\|p(y)-p(x)\|_1+\mathbbm{1}_{\{h_x\ne h_y\}}
\]
gives $d_S\leq d_\omega\leq2d_S$ for both environments. In particular, the metrics are proper and $\|g\|_S\leq c(g,\omega)\leq2\|g\|_S$.

For any $g=(v,h)$ and either starting environment $\omega$, apply the swap first if $h=1$, then take coordinate steps monotonically to $v$. The swap costs one; the coordinate steps have total cost $f_\omega(J^h v)$. Thus the resulting self-avoiding path has total cost $f_\omega(J^h v)+h=c(g,\omega)$. Its increments belong to $S$, independently of $g$, $\omega$, and the prescribed accuracy. Since the orbit of either environment is all of $\Om$, this verifies the orbit-uniform fixed-increment innerness required in \cite[Theorem~A]{cantrell2017}.

Nevertheless, for $g_n=(ne_1,0)$,
\[
 c(g_n,0)=n,\qquad c(g_n,1)=2n,\qquad \|g_n\|_S=n.
\]
The difference is not $o(\|g_n\|_S)$, so the two radial cost functions are not asymptotically equivalent in the sense used there.

With coset representatives $e,r$, the projection to $N$ is $\uppi_N(v,h)=(p(v,h),0)$. Hence the common cone coordinates are $Q_n(v,h)=p(v,h)/n$, and
\[
 \left|n^{-1}d_\omega(x,y) -f_\omega\bigl(Q_n(y)-Q_n(x)\bigr)\right|\leq n^{-1}.
\]
The image $n^{-1}\Z^2$ becomes dense in $\R^2$, giving the pointed limit $(\R^2,d_{f_\omega},0)$, where $d_{f_\omega}(u,v)=f_\omega(v-u)$. The component norms can also be read directly from the directional costs: for $u=(v,0)\in N$,
\[
 \frac{\E_{B_\omega}[c(u^n)]}{n}=f_\omega(v)
 \qquad(\omega\in\{0,1\},\ n\geq1).
\]
\begin{figure}[tbp]
\centering
\begin{tikzpicture}[
  font=\small,text=black,
  >={Stealth[length=1.6mm]},
  componentZero/.style={draw=blue!65!black,fill=blue!7,
    circle,minimum size=12mm,line width=.8pt},
  componentOne/.style={draw=orange!80!black,fill=orange!9,
    circle,minimum size=12mm,line width=.8pt},
  componentLabel/.style={fill=white,inner sep=1pt,font=\scriptsize}
]
  \node at (2.65,5.55) {(a) The two components};
  \node[componentZero] (component0) at (1.25,3.4) {$B_0$};
  \node[componentOne] (component1) at (4.05,3.4) {$B_1$};
  \draw[->,line width=.7pt]
    (component0) to[bend left=24] node[above] {$r$} (component1);
  \draw[->,line width=.7pt]
    (component1) to[bend left=24] node[below] {$r$} (component0);
  \path[->,line width=.7pt]
    (component0) edge[loop above,looseness=6] node[above] {$N$} (component0)
    (component1) edge[loop above,looseness=6] node[above] {$N$} (component1);
  \node[text=blue!65!black] at (1.25,2.05) {$\phi_{B_0}(e_1)=1$};
  \node[text=orange!80!black] at (4.05,2.05) {$\phi_{B_1}(e_1)=2$};
  \node[font=\footnotesize,align=center] at (2.65,1.05)
    {$B_0=\{0\}$,\quad $B_1=\{1\}$\\[2pt]
     $\p(B_0)=\p(B_1)=\tfrac12$};

  \node at (10,5.55) {(b) Unit balls in the same coordinates};
  \begin{scope}[shift={(10,2.9)},scale=1.85]
    \fill[blue!10] (1,0)--(0,.5)--(-1,0)--(0,-.5)--cycle;
    \fill[orange!10] (.5,0)--(0,1)--(-.5,0)--(0,-1)--cycle;
    \draw[black!40,densely dotted]
      (-1.12,-1.12)--(1.12,1.12);
    \node[componentLabel,anchor=south west] at (.73,.83) {$x_1=x_2$};
    \draw[->,black!65] (-1.27,0)--(1.36,0) node[right] {$x_1$};
    \draw[->,black!65] (0,-1.27)--(0,1.34) node[above] {$x_2$};
    \draw[blue!65!black,line width=1.0pt]
      (1,0)--(0,.5)--(-1,0)--(0,-.5)--cycle;
    \draw[orange!80!black,line width=1.0pt,dashed]
      (.5,0)--(0,1)--(-.5,0)--(0,-1)--cycle;
    \foreach \v/\lab in {.5/{\tfrac12},1/{1}} {
      \draw[black!65] (\v,-.035)--(\v,.035);
      \draw[black!65] (-.035,\v)--(.035,\v);
      \node[componentLabel,below,yshift=-2pt] at (\v,0) {$\lab$};
      \node[componentLabel,left,xshift=-2pt] at (0,\v) {$\lab$};
    }
    \node[componentLabel,below left] at (0,0) {$0$};
  \end{scope}
  \draw[blue!65!black,line width=1.0pt] (7.8,.18)--(8.35,.18);
  \node[anchor=west,font=\footnotesize] at (8.45,.18) {$K_0$};
  \draw[orange!80!black,line width=1.0pt,dashed] (9.55,.18)--(10.1,.18);
  \node[anchor=west,font=\footnotesize] at (10.2,.18) {$K_1=J(K_0)$};
\end{tikzpicture}
\caption{The two components in Proposition~\ref{prop:component.coordinates}.
The subgroup $N$ fixes each environment, whereas $r$ exchanges them.
Their limit unit balls
$K_b=\{v\in\R^2:\phi_{B_b}(v)\leq1\}$ are distinct:
$K_0$ has vertices $(\pm1,0),(0,\pm\tfrac12)$, and
$K_1$ has vertices $(\pm\tfrac12,0),(0,\pm1)$.
Reflection across $x_1=x_2$, given by $J(x_1,x_2)=(x_2,x_1)$,
takes $K_0$ to $K_1$ and is a pointed isometry between the two limit spaces.}
\label{fig:component.norms}
\end{figure}

Consequently,
\[
 \phi_{B_0}(v)=|v_1|+2|v_2|,
 \qquad
 \phi_{B_1}(v)=2|v_1|+|v_2|.
\]
In particular, $\phi_{B_0}(e_1)=1$ whereas $\phi_{B_1}(e_1)=2$.
Their unit balls are distinct diamonds in the same horizontal coordinates, as shown in Figure~\ref{fig:component.norms}.
Nevertheless, $\phi_{B_1}(Jv)=\phi_{B_0}(v)$ makes $J$ a pointed isometry between the two limits.
\end{proof}

This example concerns the fixed-coordinate asymptotic equivalence asserted in \cite[Theorem~A]{cantrell2017} for virtually nilpotent groups. In the passage to a subgroup of finite index in \cite[\S\,2.4]{cantrell2017}, a change of ergodic component must be accompanied by the induced conjugation of horizontal directions. Equation~\eqref{eq:component.norms} exhibits that relation. If the restricted $N$-action is ergodic, there is only one component and the same formula gives invariance of its norm under the finite conjugation action.

\subsection{First-passage percolation on groups of polynomial growth}
\label{sec:additional.FPP}

Consider first-passage percolation on the Cayley graph $\mathcal C(\Gr,S)$ of an infinite finitely generated group of polynomial growth, with finite symmetric generating set $S$. Each unoriented edge has a finite nonnegative random passage time. We assume that their joint law is stationary and ergodic under right translations of $\Gr$. The time $T_\omega(u,v)$ needed to reach $v$ from $u$ is the infimum of the total passage time over all finite paths from $u$ to $v$. Starting from the identity, the set reached by time $t$ is therefore
\[
 A_\omega(t)=\{x\in\Gr:T_\omega(e,x)\le t\}.
\]
We determine the limiting shape of this set and, more generally, the large-scale behavior of the passage times $T_\omega(u,v)$.

To describe the limit, fix a normal torsion-free nilpotent subgroup $N$ of finite index and use the rescaling maps $Q_n=q_n\circ\uppi_N$ from Theorem~\ref{thm:shape.polynomial}. The subgroup $N$ provides coordinates on the asymptotic cone $G_\infty$. All passage times are computed in $\mathcal C(\Gr,S)$: even when $u,v\in N$, a path from $u$ to $v$ may leave $N$. The theorem requires only the upper-tail estimate \eqref{all} and the lower linear growth condition \eqref{small.prob} for passage times from the identity. In particular, no separate innerness assumption on the restriction to $N$ is needed.

\begin{theorem}[Shape theorem for first-passage percolation]
\label{thm:ambient.fpp}
Suppose that the first-passage times $c(x,\omega)=T_\omega(e,x)$ satisfy \eqref{all} and \eqref{small.prob} on $\Gr$. On each $N$-ergodic component $B$, there is a deterministic symmetric norm $\phi_B$ on the horizontal layer of $G_\infty$ such that the following conclusions hold $\p_B$-almost surely.

The rescaled set reached by time $n$ converges in Hausdorff distance to the closed unit ball of the Carnot--Carath\'eodory distance $d_{\phi_B}$:
\[
 d_H^{\,d_\infty}\!\left( Q_n(A_\omega(n)), \{g\in G_\infty:d_{\phi_B}(\mathlcal e,g)\le 1\} \right)\longrightarrow 0.
\]
More generally, for every $R>0$,
\[
 \sup_{u,v\in B_S(e,Rn)} \left| \frac{T_\omega(u,v)}{n} -d_{\phi_B}\bigl(Q_n(u),Q_n(v)\bigr) \right|\longrightarrow 0.
\]
After identifying vertices at zero passage distance, we also have pointed Gromov--Hausdorff convergence
\[
 (\Gr,\tfrac{1}{n}T_\omega,e)\GHto(G_\infty,d_{\phi_B},\mathlcal e).
\]
The limiting unit ball in the coordinates $Q_n$ may depend on $B$. The resulting pointed metric spaces are isometric across components.
\end{theorem}

\begin{proof}
We verify ambient asymptotic innerness and apply the ambient alternative of Theorem~\ref{thm:shape.polynomial}. Work on the event of probability one in \eqref{small.prob}, on which every passage ball about $e$ is finite. For each $x\in\Gr$, a minimizing path from $e$ to $x$ exists. Indeed, choose a finite comparison path of cost $K$. Every competing path of cost at most $K$ lies in the finite set $\{v:c(v)\le K\}$. Erasing loops does not increase cost, and only finitely many simple paths in that set remain. Their minimum therefore attains the infimum defining $T_\omega(e,x)$.

Let $v_0=e,\ldots,v_k=x$ be such a path. Its increments have word length one, and
\[
 c(x)\le\sum_{j=1}^k T_\omega(v_{j-1},v_j) \le\sum_{j=1}^k\tau_\omega(v_{j-1},v_j)=c(x).
\]
The first inequality is subadditivity and the second follows by using the corresponding edge as a competitor. Hence the sum of the increment costs is exactly $c(x)$. The same chains give \eqref{innerness} for every word ball, with $\upvarepsilon_R(n,\omega)=1/n$; for $x=e$, use the empty chain.

Theorem~\ref{thm:shape.polynomial} now gives the component limit norms, their pointed isometries, the uniform two-point convergence, and pointed convergence. The passage metric is symmetric, so Lemma~\ref{lem:stable} on each component gives a symmetric $\phi_B$. The reached-set convergence follows directly from \eqref{eq:thm.polynomial.reached}.
\end{proof}

\begin{remark}
For applications it suffices to control the probability of a cheap crossing. If, for some $h>0$,
\[
 \sum_{m\ge1}\p\bigl(\exists x:\ 2^m\le \|x\|_S<2^{m+1}, T_\omega(e,x)< h2^m\bigr)<\infty,
\]
then Borel--Cantelli gives \eqref{small.prob} with constant $h/2$. In particular, the exponential fast-path estimate \eqref{eq:fast} in the examples implies this criterion: a path of cost less than $h2^m$ to the annulus can be taken simple by erasing loops, and its first $2^m$ edges have total cost less than $h2^m$. Take $h$ equal to the constant multiplying path length in \eqref{eq:fast}. This criterion will be used for the coloring and directed Richardson examples below. When $N$ is abelian, Theorem~\ref{thm:shape.polynomial} also gives the conclusion under \eqref{all} and the conditional mean bound \eqref{aml2} alone.
\end{remark}

\begin{remark}[Ambient paths and endpoint projection]
\label{rem:honeycomb.weighted}
In Example~\ref{ex:honeycomb}, the subgraph induced by $N$ has no
edges, so every nonconstant path between its vertices leaves $N$.
If a stationary ergodic FPP cocycle
$c(x,\omega)=T_\omega(e,x)$ satisfies \eqref{all},
Proposition~\ref{prop:finite.index.metric} gives, almost surely,
for every $R>0$,
\[
 \sup_{x,y\in B_S(e,Rn)}
 \frac{\left|T_\omega(x,y)
 -T_\omega(\uppi_N(x),\uppi_N(y))\right|}{n}
 \longrightarrow0.
\]
Both passage costs are computed in the whole honeycomb graph.
This endpoint comparison retains the costs of the matching edges
$A_zB_z$. Freely contracting those edges would change the unit-weight
distance from $A_0$ to $A_{me_1}$ from $2m$ to $m$, for $m\geq1$.
\end{remark}

\section{Applications to random growth models} \label{sec:examples}

In this section, we present four applications on $\mathcal C(\Gr,S)$. The coloring model permits zero edge weights, while directed Richardson growth has unbounded passage times and can produce an asymmetric limit norm. Chemical distance concerns travel through the infinite cluster of supercritical Bernoulli bond percolation, viewed through an equivariant projection from the whole group. Finally, the frog model describes activation by moving particles. Its birth-vertex chains verify asymptotic innerness and give a shape theorem for every group of polynomial growth degree $D\geq3$.

\begin{lemma}[A fast-path criterion]
\label{lem:application}
Let $\Gr$ be an infinite finitely generated group of polynomial growth, and let $c(x,\omega)=T_\omega(e,x)$ be the passage time for finite nonnegative weights on $\mathcal C(\Gr,S)$, assigned either to unoriented edges or separately to their two orientations. Suppose that for some $a,b,C>0$ and every integer $m\geq1$,
\begin{equation}\label{eq:fast}
 \p\bigl(\exists\text{ self-avoiding path }\gamma\text{ from }e: \ell(\gamma)=m,\ T(\gamma)\leq am\bigr)\leq Ce^{-bm}.
\end{equation}
Then \eqref{small.prob} holds. Moreover, almost surely, every $x\in\Gr$ can be reached from $e$ by a minimizing path.
\end{lemma}

\begin{proof}
By Borel--Cantelli, almost surely there is an $M<\infty$ such that every self-avoiding path from $e$ of length $m\geq M$ has cost greater than $am$. If $\|x\|_S\geq M$, every self-avoiding path from $e$ to $x$ therefore costs at least $a\|x\|_S$. Taking the infimum gives \eqref{small.prob}.

On the same event, fix $x$ and choose a finite path from $e$ to $x$ of cost $K$. Every self-avoiding competitor of cost at most $K$ has length at most $\max\{M,K/a\}$. There are only finitely many such paths, and erasing loops cannot increase cost. Hence the infimum is attained.
\end{proof}

\begin{remark}[Independent, direction-dependent edge weights]
\label{exmpl:independent.FPP}
Assign independent weights to the unoriented edges, or separately to their two orientations, with a common law on each orbit under right translation. In the unoriented case the law may depend on $\{s,s^{-1}\}$; in the oriented case it may depend on $s$. Generators of order two cause no difficulty with either construction. The restricted action of the infinite group $N$ is mixing, hence ergodic. Indeed, for two cylinder events, their finite supports are disjoint after translation by all but finitely many elements of $N$. Independence then factors their probabilities. Approximation by cylinder events proves mixing for arbitrary measurable events.
\end{remark}

\begin{example}[First-Passage Percolation for a Random Coloring of $\Gr$]\label{ex:color}
Let us now consider a dependent Bernoulli FPP model based on the random coloring studied by Fontes and Newman \cite{fontes1993}. Set $\{X_x\}_{x \in \Gr}$ to be a family of i.i.d. random variables taking values in a finite set of colors $\mathlcal{F}$. The model generates color clusters by assigning weight $0$ to edges between sites with same color and weight $1$ otherwise. We define for every edge $u \sim v$
\[
    \tau(u,v) = \mathbbm{1}_{\{X_u\ne X_v\}}.
\]
Set for each self-avoiding path $\gamma \in \mathscr{P}(x,y)$ the random length $T(\gamma)= \sum_{\mathtt{e} \in \gamma}\tau(\mathtt{e})$. The first-passage time is
\[T(x,y) := \inf_{\gamma \in \mathscr{P}(x,y)}T(\gamma).\]

Let $p_{\mathlcal{s}}:= \p(X_x = \mathlcal{s})$; then one can verify that $T(x,y)$ is an FPP model with dependent identically distributed passage times $\tau(x,y) \sim \operatorname{Ber}\left(1-\sum_{\mathlcal{s} \in \mathlcal{F}}p_{\mathlcal{s}}^2\right)$. One can easily see that $c(x):= T(e,x)$ is a subadditive cocycle and the translations $\tht$ are ergodic due to the fact that $\{X_x\}_{x \in \Gr}$ are i.i.d. random variables. 

Since every edge weight is at most one, $c(x)\leq\|x\|_S$, so \eqref{all} holds. Put $\Delta=|S|$ and $p_{\max}=\max_{\mathlcal s\in\mathlcal F}\p(X_e=\mathlcal s)$.

\begin{lemma}\label{lm:aml.colors}
If $p_{\max}<1/(\Delta-1)$, there are $a,b,C>0$ such that \eqref{eq:fast} holds.

\end{lemma}
\begin{proof}
For a fixed self-avoiding path $\gamma=(x_0,\ldots,x_m)$, reveal colors in order. Independence of the fresh color at $x_i$ gives, for $t>0$,
\[
 \E\left[e^{-t\mathbbm{1}_{\{X_{x_i}\ne X_{x_{i-1}}\}}} \mid X_{x_0},\ldots,X_{x_{i-1}}\right] \leq p_{\max}+(1-p_{\max})e^{-t}=:\chi(t).
\]
Iteration and exponential Markov inequality imply
\[
 \p(T(\gamma)\leq am)\leq(e^{ta}\chi(t))^m.
\]
There are at most $\Delta(\Delta-1)^{m-1}$ such paths from $e$. First take $t$ large enough that $(\Delta-1)\chi(t)<1$, and then take $a>0$ small enough that $(\Delta-1)e^{ta}\chi(t)<1$. This proves \eqref{eq:fast}. 

\end{proof}

If $p_{\max}<1/(\Delta-1)$, Lemma~\ref{lm:aml.colors} gives \eqref{eq:fast}, and Lemma~\ref{lem:application} therefore gives \eqref{small.prob}. Theorem~\ref{thm:ambient.fpp} applies to the passage metric on $\mathcal C(\Gr,S)$.

\begin{remark}
The condition $p_{\max}<1/(\Delta-1)$ is a convenient sufficient path-counting condition, not an optimal percolation threshold. For a fixed color of probability $p$, the elementary estimate $\Delta(\Delta-1)^{m-1}p^{m+1}$ for monochromatic paths already excludes an infinite cluster when $(\Delta-1)p<1$. This observation does not require the shape theorem.
\end{remark}
\end{example}

\begin{example}[Richardson growth with directed rates in a random environment]\label{ex:richardson}
Assign independent random rates $\lambda(\mathtt e)$ to the ordered neighboring pairs $\mathtt e=(g,sg)$, $g\in\Gr$, $s\in S$, with a common distribution in each orbit under right translation (i.e., the distribution of $\lambda(g,sg)$ depends only on $s$), and assume
\[
 0<\lambda_-\leq\lambda(\mathtt e)\leq\lambda_+<\infty\quad\text{a.s.}
\]
Both orientations are present. An infected vertex $u$ transmits infection to a healthy neighbor $v$ at rate $\lambda(u,v)$, and infection is permanent. Let $\nu$ be the law of the rate field. Given the rates, let the oriented passage times be conditionally independent with $\tau(\mathtt e)\sim\operatorname{Exp}(\lambda(\mathtt e))$. Write $\p_\lambda$ and $\E_\lambda$ for the conditional probability and expectation, and $\p=\int\p_\lambda\,d\nu(\lambda)$ for the annealed law.

Let $T_\omega(u,v)$ be the infimum of total passage times over finite oriented paths from $u$ to $v$, and put $c(x,\omega)=T_\omega(e,x)$. Under the shift $\tau_{z\cdot\omega}(u,v)=\tau_\omega(uz,vz)$, translation of paths gives the cocycle inequality as in \cref{sec:fpp}. On the joint space of rates and independent exponential seeds, this action preserves the annealed law and its restriction to $N$ is mixing. A fixed environment need not give a stationary quenched law under translations of the passage-time coordinates alone.

\begin{lemma}\label{lm:richardson.all}
There are $\upbeta,C,C'>0$ such that
\[
 \p(c(x)\geq t)\leq C e^{-C't},\qquad t\geq\upbeta\|x\|_S.
\]
\end{lemma}
\begin{proof}
For a word-geodesic of length $n=\|x\|_S$, conditional independence gives,
with $\eta_0=\lambda_-/2$,
\[
 \E_\lambda[e^{\eta_0\tau(\mathtt e)}] =\frac{\lambda(\mathtt e)}{\lambda(\mathtt e)-\eta_0}\leq2, \qquad \p_\lambda(c(x)\geq t)\leq2^n e^{-\eta_0 t}.
\]
Take $\upbeta\geq2\log(2)/\eta_0$ and integrate over the environment.
\end{proof}

\begin{lemma}\label{lm:richardson.aml}
The fast-path bound \eqref{eq:fast} holds. 
Moreover $c(s)>0$ for every $s\in S$ almost surely.
\end{lemma}
\begin{proof}
For a fixed self-avoiding path $\gamma$ of length $m$, conditional independence and
$\lambda(\mathtt e)\leq\lambda_+$ give
\[
 \E_\lambda[e^{-tT(\gamma)}] \leq\left(\frac{\lambda_+}{\lambda_++t}\right)^m.
\]
Choose $t$ large, then $a>0$ small, so that $(|S|-1)e^{ta}\lambda_+/(\lambda_++t)<1$. Counting self-avoiding paths proves \eqref{eq:fast}, uniformly in the environment. 
Every edge weight is strictly positive almost surely, and minimizing paths exist by \cref{lem:application}. Hence the distance between distinct vertices is strictly positive.
\end{proof}

The passage times are unbounded above and are not bounded away from zero, despite the bounds on the rates. Without a deterministic lower bound on the rates, $\p(\tau(\mathtt e)>t)=\E[e^{-\lambda(\mathtt e)t}]$ need not decay exponentially. Lemmas~\ref{lm:richardson.all} and~\ref{lm:richardson.aml} give \eqref{all} and \eqref{eq:fast}, while Lemma~\ref{lem:application} gives \eqref{small.prob} and minimizing paths. Their increments have word length one and their total passage cost is exact, so \eqref{innerness} holds with $\upvarepsilon_R(n,\omega)=1/n$. The ambient alternative of Theorem~\ref{thm:shape.polynomial} therefore applies under the annealed law. Mixing of the $N$-action makes the horizontal limit norm deterministic, and it may be asymmetric.

For a concrete asymmetric case, take $\Gr=H_3(\Z)$ with
$S=\{\mathtt X^{\pm1},\mathtt Y^{\pm1}\}$ as in
Example~\ref{ex:Heisenberg.group}. Choose independent rates uniform
on $[10,11]$ in direction $\mathtt X$ and on $[1/2,1]$ in the other
three directions. Take $N=\Gr$ and put $v_{\mathtt X}=\alpha(\mathtt X)$.
The direct $\mathtt X$-path gives
$\E[c(\mathtt X^n)]\leq n/10$, hence
$\phi(v_{\mathtt X})\leq1/10$.

At parameter $11$, the conditional bounds for
$\E_\lambda[e^{-11\tau(\mathtt e)}]$ are $1/2$ in
direction $\mathtt X$ and $1/12$ in the other directions.
A path to $\mathtt X^{-n}$ has $n$ more $\mathtt X^{-1}$-steps
than $\mathtt X$-steps. Weighting these two types of letters
additionally by $e^2$ and $e^{-2}$, respectively, multiplies
its Laplace bound by $e^{2n}$. The sum of the single-letter weights is
\[
 r_*=\tfrac12e^{-2}+\tfrac1{12}(e^2+2)<1.
\]
For self-avoiding paths the conditional Laplace transform factors.
Exponential Markov inequality, followed by bounding the resulting
deterministic sum by all generator words of each length $m$, therefore gives
\[
 \p\!\left(c(\mathtt X^{-n})\leq\tfrac{3n}{22}\right)
 \leq e^{-n/2}\sum_{m\geq1}r_*^m
 =\frac{r_*}{1-r_*}e^{-n/2}.
\]
Borel--Cantelli and the directional limit imply
\[
 \phi(v_{\mathtt X})\leq\tfrac1{10}
 <\tfrac3{22}\leq\phi(-v_{\mathtt X}).
\]
Thus asymmetry persists in the limiting geometry.
\end{example}

\begin{example}[Chemical distance in supercritical bond percolation]
\label{ex:chemical}
Let $\mathcal G=\mathcal C(\Gr,S)$, where $\Gr$ is infinite, finitely generated, and of polynomial growth degree $D$, and put $p_c=p_c^{\mathrm{bond}}(\mathcal G)$. Assign independent uniform $(0,1)$ variables $U_\eta$, $\eta\in E(\mathcal G)$, and declare an edge $\eta$ open at parameter $p$ when $U_\eta\le p$. All chemical-distance results below concern this Bernoulli bond model.

Fix $p\in(p_c,1]$. Almost surely there is a unique infinite open cluster; write $\mathscr C_{\infty,p}$ for its vertex set. Chemical distance is the minimum number of edges in an open path joining two cluster vertices. For fixed $p$, suppress the parameter in $\mathscr C_\infty$, $d_{\mathrm{chem},\omega}$, and $\pi_\omega$. To define finite costs on the whole group, let $\pi_\omega(x)$ be a nearest cluster vertex in the word metric. Resolve ties by the smallest of independent uniform $(0,1)$ vertex marks, independent of the percolation variables. Probabilities and expectations include these auxiliary marks. This fixes cluster vertices and commutes with right translation:
\[
 \pi_{z\cdot\omega}(x)=\pi_\omega(xz)z^{-1}.
\]
Thus $D_\omega(x,y)=d_{\mathrm{chem},\omega} (\pi_\omega(x),\pi_\omega(y))$ is a finite symmetric equivariant pseudometric, and $c(x,\omega)=D_\omega(e,x)$ is a normalized subadditive cocycle. The product action is mixing on every infinite subgroup.

\end{example}

\begin{corollary}[Chemical-distance shape theorem]\label{cor:chemical.shape}
Consider the bond model and the nearest-cluster projection in Example~\ref{ex:chemical}. Fix a normal torsion-free nilpotent subgroup $N$ of finite index and the maps $Q_n$ as in Theorem~\ref{thm:shape.polynomial}. For every fixed $p\in(p_c,1]$, there is a deterministic symmetric horizontal norm $\phi_p$ such that, almost surely, for every $R>0$,
\begin{equation}\label{eq:chemical.two.point}
 \sup_{x,y\in B_S(e,Rn)} \left|n^{-1}D_\omega(x,y) -d_{\phi_p}(Q_n(x),Q_n(y))\right|\longrightarrow0.
\end{equation}
Moreover,
\[
 (\mathscr C_\infty,n^{-1}d_{\mathrm{chem},\omega},\pi_\omega(e)) \GHto(G_\infty,d_{\phi_p},\mathlcal e).
\]
For the chemical ball $A_\omega^{\mathrm{chem}}(n)= \{v\in\mathscr C_\infty: d_{\mathrm{chem},\omega}(\pi_\omega(e),v)\le n\}$, we also have
\[
 d_H^{\,d_\infty}\!\left(Q_n(A_\omega^{\mathrm{chem}}(n)), \{g:d_{\phi_p}(\mathlcal e,g)\le1\}\right)\longrightarrow0.
\]
\end{corollary}
\begin{proof}
Put $r_\omega(x)=d_S(x,\pi_\omega(x))$. By \cite[Proposition~3.2]{gorski2025chemical}, for suitable $C_0,a_0>0$,
\begin{equation}\label{eq:chemical.projection.tail}
 \p(r_\omega(x)>t)\le C_0e^{-a_0\sqrt t} \qquad(x\in\Gr,\ t\ge1).
\end{equation}
The union bound on word balls and Borel--Cantelli give, almost surely,
\begin{equation}\label{eq:chemical.uniform.projection}
 M_R(n,\omega):=\max_{\|x\|_S\le Rn}r_\omega(x)=o(n) \qquad\text{for every }R>0.
\end{equation}
Indeed, for fixed $R,\delta>0$, the failure probability is at most $C_Rn^De^{-a_0\sqrt{\delta n}}$, which is summable; intersect over integer $R$ and rational $\delta$.

To verify \eqref{all}, take $K\ge1$ from \cite[Theorem~1]{gorski2025chemical}. For $x\ne e$ and $t>8K\|x\|_S$, put $r=t/(8K)$. If both projection errors are at most $r$, every possible pair of projected endpoints has word distance less than $t/K$. Their theorem and a union bound over their possible values yield
\begin{align*}
 \p(c(x)\ge t) &\le 2C_0e^{-a_0\sqrt r} +\sum_{\substack{d_S(e,u)\le r\\d_S(x,v)\le r}} \p\bigl(u\leftrightarrow v,\, d_{\mathrm{chem},\omega}(u,v)\ge t\bigr)\\ &\le 2C_0e^{-a_0\sqrt r}+C(1+t)^{2D}e^{-at/K} \le C_1e^{-a_1\sqrt t}.
\end{align*}
The sum retains the connection event without conditioning the percolation law on the selected projections. For $x=e$, use $c(e)=0$. Hence \eqref{all} holds with $\upbeta=8K$ and $g(t)=\min\{1,C_1e^{-a_1\sqrt t}\}$, whose weighted tail integral is finite.

Since chemical distance dominates word distance,
\begin{equation}\label{eq:chemical.lower}
 c(x)\ge\|x\|_S-r_\omega(e)-r_\omega(x) \ge\tfrac12\|x\|_S \quad\text{eventually as }\|x\|_S\to\infty,
\end{equation}
where the last bound follows by taking $n=\|x\|_S$ in \eqref{eq:chemical.uniform.projection}. This proves \eqref{small.prob} with $b=1/2$ despite possible zero distances. An open geodesic $w_0=\pi_\omega(e),\ldots,w_m=\pi_\omega(x)$ gives the chain $e,w_0,\ldots,w_m,x$. Its two projection links cost zero, its other links cost one, and its total cost is exactly $c(x)$. Its largest word jump is at most $\max\{r_\omega(e),1,r_\omega(x)\}$. Thus $\upvarepsilon_R(n,\omega)=\max\{1,M_R(n,\omega)\}/n$ verifies \eqref{innerness}, with no chain-cost error; use the empty chain for $x=e$.

The ambient alternative in Theorem~\ref{thm:shape.polynomial} now applies. The restricted $N$-action is mixing and the projected chemical distance is symmetric, so the limit norm $\phi_p$ is deterministic and symmetric. This proves \eqref{eq:chemical.two.point}. The projection identifies the zero-distance quotient isometrically with $\mathscr C_\infty$, proving the asserted pointed convergence.

For the reached sets, writing $A_\omega(n)=\{x:c(x)\le n\}$, we have $A_\omega^{\mathrm{chem}}(n)=\pi_\omega(A_\omega(n))$. Condition~\eqref{small.prob} confines $A_\omega(n)$ to a word ball of radius $O(n)$. Let $L=\max_j\|z_{(j)}\|_S$. Uniformly over $x\in A_\omega(n)$, \eqref{eq:chemical.uniform.projection} gives
\[
 d_S\bigl(\uppi_N(x),\uppi_N(\pi_\omega(x))\bigr)
 \le r_\omega(x)+2L=o(n).
\]
Comparability of the word metrics on $N$ and uniform word-cone convergence therefore give
$d_\infty(Q_n(x),Q_n(\pi_\omega(x)))=o(1)$ uniformly over the reached set. Hence projection changes its rescaled image by a vanishing Hausdorff distance.
\end{proof}

\begin{corollary}[Continuity of the chemical limit geometry]
\label{cor:chemical.continuity}
In the setting of Example~\ref{ex:chemical}, let $\phi_p$ be the norm in Corollary~\ref{cor:chemical.shape}. Fix the subgroup $N$ and the cone coordinates independently of $p$, and let $\rho$ be the horizontal norm defining the ambient word-cone distance $d_\infty$. At $p=1$ we have $\phi_1=\rho$. For every $p_0\in(p_c,1)$ there are finite constants $C,C'$, depending on $p_0$ and $\mathcal C(\Gr,S)$, such that for all $p,q\in[p_0,1]$,
\begin{equation}\label{eq:chemical.norm.continuity}
 |\phi_p(v)-\phi_q(v)|\le C|p-q|\rho(v) \qquad(v\in\mathfrak v_1).
\end{equation}
Moreover,
\begin{equation}\label{eq:chemical.metric.continuity}
 |d_{\phi_p}(g,h)-d_{\phi_q}(g,h)| \le C'|p-q|d_\infty(g,h) \qquad(g,h\in G_\infty).
\end{equation}
In particular, the limit distances depend continuously on $p$, uniformly on compact subsets of $G_\infty\times G_\infty$.
\end{corollary}
\begin{proof}
Write $D_p$ for the projected chemical distance at parameter $p$ under the common edge-uniform construction in Example~\ref{ex:chemical}. We apply \cite[Theorem~3]{gorski2025chemical} to Bernoulli bond percolation, the model used in its proof through Russo's formula for edge variables. Its expectations are unconditional and concern the nearest-cluster pseudometric defined at the beginning of Section~3 of that paper. Thus, for all $p,q\in[p_0,1]$,
\[
 \limsup_{\|x\|_S\to\infty} \frac{|\E D_p(e,x)-\E D_q(e,x)|}{\|x\|_S} \le C|p-q|.
\]
This estimate is unaffected by the nearest-cluster tie convention. Indeed, let $H_p(x)$ be the largest chemical distance between any two nearest cluster vertices to $x$. The projection-tail estimate and the same union bound as above, now with both endpoints near $x$, give $\p(H_p(x)>t)\le C_p e^{-a_p\sqrt t}$ after changing constants. Thus $\E H_p(x)=\E H_p(e)<\infty$, and changing the projection changes $\E D_p(e,x)$ by at most $2\E H_p(e)$, independently of $x$. For each fixed pair $p,q$, these endpoint errors vanish after division by $\|x\|_S$, so the same constant $C(p_0)$ applies to our projection convention.

For $u\in N$ with $\alpha(u)\ne0$, the directional characterization of the limit norm and word-cone convergence give
\[
 \frac{\E D_p(e,u^n)}n\longrightarrow\phi_p(\alpha(u)), \qquad \frac{\|u^n\|_S}n\longrightarrow\rho(\alpha(u)).
\]
Apply the preceding estimate along $u^n$. Homogeneity and continuity extend the resulting inequality from the horizontal lattice to $\mathfrak v_1$, proving \eqref{eq:chemical.norm.continuity}.

Chemical distance dominates word distance, so the integrable projection errors also give $\phi_p\ge\rho$ by the same directional limits. At $p=1$ every edge is open, the projection is the identity, and $D_1=d_S$, whence $\phi_1=\rho$. Equation~\eqref{eq:chemical.norm.continuity}, applied with $q=1$, gives
\[
 \rho\le\phi_p\le\bigl(1+C(1-p)\bigr)\rho\le M\rho
 \qquad(p\in[p_0,1]),\qquad M=1+C.
\]
Integrating along horizontal curves yields $d_{\phi_p}\le(1+C|p-q|)d_{\phi_q}$ and the reverse inequality. Since $d_{\phi_p},d_{\phi_q}\le M d_\infty$, this proves \eqref{eq:chemical.metric.continuity} with $C'=CM$.
\end{proof}

\begin{example}[The Frog Model]\label{ex:frog}

The frog model, studied by Telcs and Wormald \cite{telcs1999} and by Alves, Machado and Popov \cite{alves2002}, describes an awakening process on $\mathcal C(\Gr,S)$. Initially there is one particle at each vertex, and only the particle at $e$ is active. Active particles perform independent simple random walks and awaken the particles at the sites they visit. We extend the abelian shape theorem of Coletti and de Lima \cite{coletti2021} to all infinite finitely generated groups of polynomial-growth degree $D\geq3$. For a continuous-time Poisson version with finite particle lifetimes, Angel, de la Riva, Hermon and Shi \cite{angel2026} establish existence of a phase transition on quasi-transitive graphs of superlinear polynomial growth and sharpness on transitive graphs.

Assign independent uniform instructions $\xi_i^z\in S$ to every birth vertex $z$ and age $i\geq1$. For $\omega=(\xi_i^z)_{z,i}$, set
\[
 S_0^z=z,\qquad S_m^z=\xi_m^z\cdots\xi_1^z z,\qquad (\tht_u\omega)_i^z=\xi_i^{zu}(\omega).
\]
These translations preserve the product law. Their restriction to any infinite subgroup is mixing, because finite sets of birth vertices supporting cylinder events become disjoint after all but finitely many translations.

Let $t_\omega(x,y)=\inf\{m\in\N_0:S_m^x=y\}$ be the hitting time counted from the particle's activation. The activation time of the particle at $v$ in the process started at $u$ is
\begin{equation}\label{eq:frog.activation}
 T_\omega(u,v)=\inf\left\{ \sum_{j=1}^k t_\omega(v_{j-1},v_j): k\in\N_0,\ v_0=u,\ v_k=v\right\},
\end{equation}
with the empty chain allowed when $u=v$. Consecutive vertices in this chain need not be neighbours. Since $S_m^z(\tht_u\omega)=S_m^{zu}(\omega)u^{-1}$, we have $T_{\tht_u\omega}(x,y)=T_\omega(xu,yu)$. Thus $c(x,\omega)=T_\omega(e,x)$ is a normalized subadditive cocycle and $T_\omega(u,v)=c(vu^{-1},u\cdot\omega)$. Every activation chain takes at least the word distance between its endpoints, so
\begin{equation}\label{eq:frog.lower}
 T_\omega(u,v)\geq d_S(u,v).
\end{equation}
The costs are directed: interchanging the endpoints can change the activation time.
\end{example}

\begin{lemma}[Growth bounds and asymptotic innerness]\label{lem:frog.innerness}
For the frog model in Example~\ref{ex:frog} with $D\geq3$, the activation cocycle satisfies \eqref{all}, \eqref{small.prob}, and \eqref{innerness} on $\Gr$. Its approximating chains may be chosen self-avoiding, with total activation cost exactly $c(x,\omega)$.
\end{lemma}
\begin{proof}
Proposition~2.10 of \cite{coletti2021}, after rescaling its integer time parameter, gives $C,c_0,\varkappa>0$ and $\upbeta>1$ such that
\[
 \p\bigl(T_\omega(e,x)\geq t\bigr) \leq C\exp(-c_0t^\varkappa) \qquad(t>\upbeta\|x\|_S).
\]
This implies \eqref{all}, while \eqref{eq:frog.lower} gives \eqref{small.prob} with $b=1$. The tail bound, polynomial volume growth, and Borel--Cantelli also show that, for every integer $R\geq1$, there is a deterministic integer $C_R$ such that eventually $c(x,\omega)\leq C_Rn$ for all $x\in B_S(e,Rn)$.

The exit-time estimate in \cite[Proposition~2.3]{coletti2021} controls the maximum displacement up to a given time. Both cited propositions apply to the symmetric simple random walks considered here on polynomial-growth Cayley graphs of degree $D\geq3$, including bipartite graphs. For positive integers $A,C$ and $\delta>0$, the union bound over birth vertices gives
\[
 \p\bigl(J_n(A,C)>\delta n\bigr) \leq C_A n^D e^{-c_{C,\delta}n},\quad \text{with}\quad J_n(A,C):=\max_{\substack{\|z\|_S\leq An\\0\leq m\leq Cn,\ m\in\N_0}} d_S(z,S_m^z).
\]
Consequently $J_n(A,C)=o(n)$ almost surely, on one event for all integers $A,C$, by Borel--Cantelli and countably many rational $\delta>0$.

Work on this common event. The preceding bound implies $c(x,\omega)<\infty$ for every $x\in\Gr$. Finite hitting-time chain costs are integers, so the infimum defining $T_\omega(e,x)$ is attained; erase repetitions to obtain a self-avoiding minimizing chain $e=v_0,\ldots,v_k=x$. Its birth vertices satisfy $\|v_j\|_S\leq\sum_{i=1}^jt_\omega(v_{i-1},v_i)\leq c(x)$, and each leg takes at most $c(x)$ steps. Moreover,
\begin{equation}\label{eq:frog.exact.chain}
 c(x)\leq\sum_jT_\omega(v_{j-1},v_j)
 \leq\sum_jt_\omega(v_{j-1},v_j)=c(x).
\end{equation}
For all $x\in B_S(e,Rn)$ and large $n$, every leg therefore has word displacement at most $J_n(C_R,C_R)=o(n)$. Taking $\upvarepsilon_R(n,\omega)=J_n(C_R,C_R)/n$ verifies \eqref{innerness} uniformly over the ball, with zero cost error. For $x=e$, use the empty chain.
\end{proof}

The chains follow particles from their birth vertices, each starting at age zero. Cutting an older particle's walk at an arbitrary site would use its remaining instructions, rather than those of the particle born there. The construction above avoids this mismatch; it needs no fixed finite set of spatial increments.

\begin{corollary}[Shape theorem for the frog model]\label{cor:frog.polynomial}
For the model in Example~\ref{ex:frog} on an infinite finitely generated group of polynomial-growth degree $D\geq3$, fix $N$ and $Q_n$ as in Theorem~\ref{thm:shape.polynomial}. There is a deterministic horizontal norm $\phi$, possibly asymmetric, such that almost surely, for every $R>0$,
\begin{equation}\label{eq:frog.two.point}
 \sup_{u,v\in B_S(e,Rn)}
 \left|n^{-1}T_\omega(u,v)-d_\phi(Q_n(u),Q_n(v))\right|
 \longrightarrow0.
\end{equation}
The rescaled reached sets $Q_n(\{x:T_\omega(e,x)\leq n\})$ converge in $d_\infty$-Hausdorff distance to the closed forward unit ball of $d_\phi$, and
\[
 (\Gr,n^{-1}T_\omega,e)\GHto(G_\infty,d_\phi,\mathlcal e)
\]
in the directed pointed sense of the introduction. The limit distance is homogeneous and bi-Lipschitz equivalent to $d_\infty$.
\end{corollary}
\begin{proof}
Apply the ambient alternative in Theorem~\ref{thm:shape.polynomial} using Lemma~\ref{lem:frog.innerness}. The restricted action of the infinite subgroup $N$ is mixing, as observed above, so there is one $N$-ergodic component and the norm is deterministic in the coordinates $Q_n$.
\end{proof}

\subsection*{Funding}

Research supported by grants \#2019/19056-2, \#2023/13453-5, \#2024/06021-4 and \#2025/27064-6, S\~ao Paulo Research Foundation (FAPESP). It was also supported by the National Institute of Science and Technology in Stochastic Modeling and Complexity (INCT-NUMEC), funded by CNPq (grant no. 408590/2024-6).

\subsection*{Declaration of AI use}

During the revision of this manuscript, the authors used OpenAI's GPT--6.1 Sol Ultra and Gemini's 4--Neon as auxiliary tools to reorganize the exposition, revise the text, identify potential logical gaps and develop figures and alternative arguments. They also assisted with checking examples and
editing the \LaTeX{} source. The authors are responsible for the final statements, proofs, and references.

\bibliographystyle{abbrvnat}
\bibliography{references}

\end{document}